\documentclass{article}

\usepackage{arxiv}

\usepackage{amsmath, amsthm, amssymb, epsfig, sectsty, graphicx, float, url}
\newcommand{\defeq}{\overset{\text{\tiny def}}{=}}

\usepackage{algorithm, algorithmicx, algpseudocode}
\usepackage{float}
 \usepackage{subcaption}
 \usepackage{caption}

\usepackage[utf8]{inputenc} % allow utf-8 input
\usepackage[T1]{fontenc}    % use 8-bit T1 fonts
\usepackage{url}            % simple URL typesetting
\usepackage{booktabs}       % professional-quality tables
\usepackage{amsfonts}       % blackboard math symbols
\usepackage{nicefrac}       % compact symbols for 1/2, etc.
\usepackage{microtype}      % microtypography
\usepackage{lipsum}
\numberwithin{equation}{section}

\usepackage{amsmath, amsthm, amssymb, epsfig, sectsty, graphicx, float, url}
\usepackage{algorithm, algorithmicx, algpseudocode}
\usepackage{float}
 \usepackage{subcaption}
 \usepackage{caption}
 
\usepackage[colorinlistoftodos,bordercolor=orange,backgroundcolor=orange!20,line
color=orange,textsize=scriptsize]{todonotes}
\catcode`\@=11
\catcode`\@=12

\numberwithin{equation}{section}

\usepackage{bm}

\theoremstyle{plain}
\newtheorem{theorem}{Theorem}[section]
\newtheorem{lemma}[theorem]{Lemma}

\newtheorem{proposition}[theorem]{Proposition}

\theoremstyle{definition}

\newtheorem{remark}[theorem]{Remark}

\newcommand{\R}{\mathbb{R}}

\newcommand{\PP}{\mathbb{P}}

\newcommand{\A}{\mathbb{A}}

\newcommand{\N}{\mathbb{N}}

\newcommand{\X}{\mathbb{S}}

\newcommand{\LL}{\mathcal{L}}

\usepackage[utf8]{inputenc} % allow utf-8 input
\usepackage[T1]{fontenc}    % use 8-bit T1 fonts
\usepackage{url}            % simple URL typesetting
\usepackage{booktabs}       % professional-quality tables
\usepackage{amsfonts}       % blackboard math symbols
\usepackage{nicefrac}       % compact symbols for 1/2, etc.
\usepackage{microtype}      % microtypography

\title{Scalable First-Order Methods for Robust MDPs}

\author{%
   Julien Grand-Cl{\'e}ment \\
Columbia University \\
   \texttt{jg3728@columbia.edu} \\
   \And
  Christian Kroer\\
Columbia University\\
  \texttt{ck2945@columbia.edu} \\
}

\begin{document}
%\linenumbers
\maketitle

\begin{abstract}
Robust Markov Decision Processes (MDPs) are a powerful framework for modeling sequential decision making problems with model uncertainty. This paper proposes the first first-order framework for solving robust MDPs. Our algorithm interleaves primal-dual first-order updates with approximate Value Iteration updates. By carefully controlling the tradeoff between the accuracy and cost of Value Iteration updates, we achieve an ergodic convergence rate of
$O \left( A^{2} S^{3}\log(S)\log(\epsilon^{-1}) \epsilon^{-1} \right)$ for the best choice of parameters on ellipsoidal and Kullback-Leibler $s$-rectangular uncertainty sets, where $S$ and $A$ is the number of states and actions, respectively.
Our dependence on the number of states and actions is significantly better (by a factor of $O(A^{1.5}S^{1.5})$) than that of pure Value Iteration algorithms.
In numerical experiments on ellipsoidal uncertainty sets we show that our algorithm is significantly more scalable than state-of-the-art approaches. Our framework is also the first one to solve robust MDPs with $s$-rectangular KL uncertainty sets.
\end{abstract}
\section{Introduction}\label{sec:intro}
%\noindent\textbf{Markov Decision Processes.}
In this paper we focus on solving robust Markov Decision Processes (MDPs) with finite set of states and actions.
Markov decision process models are widely used in decision-making \citep{Bertsekas,Puterman}. In the classical MDP setting, for each state $s \in \X$,  the decision maker chooses a probability distribution $\bm{x}_{s}$ over the set of actions $\A$. The decision maker incurs a cost $\sum_{a=1}^{|\A|} x_{sa}c_{sa}$ for some non-negative scalars $c_{sa}$ and then randomly enters a new state, according to transition kernels $\bm{y}=\left\{ \bm{y}_{sa}  \in \Delta(|\X|) \right\}_{sa}$ over the next state, where $\Delta(|\X|)$ is the simplex of size $|\X|$. Given a discount factor $\lambda$, the goal of the decision-maker is to minimize the infinite horizon discounted expected cost 
% $R(\bm{x},\bm{y})$, defined as
$
 C(\bm{x},\bm{y})=E^{\bm{x},\bm{y}} \left[ \sum_{t=0}^{\infty} \lambda^{t}c_{s_{t}a_{t}} \; \bigg| \; \bm{s}_{0} \sim \bm{p}_{0} \right].
$
%Algorithms for solving MDPs include Policy Iteration, Value Iteration (VI) and linear-programming-based algorithms \citep{Puterman,Bertsekas}.
%
%An MDP is described by a set of states $\X$, a set of actions $\A$, a transition kernel $\bm{y}$ which gives transition probabilities $\boldsymbol{y}_{sa} \in \R^{|\X|}_{+}$ for all state-action pair $(s,a)$, costs $c_{sa}$ for each state-action pair $(s,a)$ and a discount factor $ \lambda \in (0,1)$. The goal of the decision maker is to choose a  policy $\bm{x}$ which assigns a probability distribution over the set of actions for each state ( i.e. $\bm{x} \in \Pi=\left( \Delta(|\A|) \right)^{|\X|} $, where $\Delta(|\A|)$ is the simplex of dimension $|\A|$) in order to minimize the infinite horizon discounted expected cost $R(\bm{x},\bm{y})$, defined as
%$
% R(\bm{x},\bm{y})=E^{\bm{x},\bm{y}} \left[ \sum_{t=0}^{\infty} \lambda^{t}c_{s_{t}a_{t}} \; \bigg| \; \bm{s}_{0} \sim \bm{p}_{0} \right],
%$
%where $s_{t}$ is the state at period $t \in \N$ and $a_{t}$ is the action chosen at period $t$ following the probability distribution $(x_{s_{t}a})_{a} \in \R_{+}^{|\A|}$. The vector $\bm{p}_{0} \in \R_{+}^{|\X|}$ is a given initial probability distribution.
%For nominal MDPs, an optimal policy can be found in the set $\Pi$ of stationary, Markovian and deterministic policies. 

%\noindent\textbf{Robust Markov Decision Processes.}
The cost of a policy can be highly sensitive to the exact kernel parameters $\bm{y}$.
We consider a robust approach where the uncertainty in $\boldsymbol{y}$ is adversarially selected from an \textit{uncertainty set} $\PP$ centered around a nominal estimation $\bm{y}^{0}$ of the true transition kernel. Our goal is to solve the \textit{robust MDP} problem
$
\min_{\bm{x} \in \Pi} \; \max_{\boldsymbol{y} \in \PP} \; C(\bm{x},\boldsymbol{y})
$ \citep{Iyengar, Nilim, Kuhn, Goh, GGC}, which has found applications in healthcare~\citep{steimle-1, steimle-2, Goh, grand2020robust}.
We focus on $s$-rectangular uncertainty sets, where
$
\PP = \underset{s \in \X}{\times} \; \PP_{s}, \text{ for } \; \PP_{s} \subseteq \R^{|\A| \times |\X|}_{+},
$
% For $s$-rectangular uncertainty sets, an optimal policy exists in the class of stationary and Markovian policies. However, there may not be any deterministic optimal policy.
and solving the robust MDP problem is equivalent to computing the fixed point of the \textit{Bellman operator}, thus allowing a \textit{value iteration} (VI) algorithm \citep{Kuhn}.

We focus on two specific classes of $s$-rectangular uncertainty sets. \textit{Kullback-Leibler} (KL) uncertainty sets are constructed from density estimation and naturally appear as approximations of the confidence intervals for the maximum likelihood estimates of $\bm{y}$ given some historical transition data \citep{Iyengar}.
\textit{Ellipsoidal} uncertainty sets are widely used because of their tractability and the probabilistic guarantees of the optimal solutions of the robust problems \citep{ben2000robust,bertsimas2019probabilistic}. 
For ellipsoidal uncertainty sets, the value iteration algorithm involves solving a convex program with a quadratic constraint at every epoch. 
While this can be done in polynomial time with modern Interior Point Methods (IPMs, \cite{lobo1998applications}), this requires inverting matrices at every step of the IPM which can be intractable for large MDP instances. 
Typically, for $S = | \X|, A = | \A|$, the complexity of VI to return an $\epsilon$-solution to the robust MDP problem with ellipsoidal uncertainty sets is $O \left( A^{3.5}S^{4.5}\log^{2}(\epsilon^{-1}) \right).$ This may prove prohibitive for large instances. For KL uncertainty sets, we are not aware of any tractable algorithm for solving $s$-rectangular robust MDP with KL uncertainty sets, even  though they are well understood in Distributionally Robust Optimization \citep{hu2013kullback}.

%\noindent\textbf{First Order Methods for Saddle Point Optimization.}
Many problems in machine learning and game theory can be written in the form $
\min_{\bm{x} \in X} \max_{\bm{y} \in Y} \LL(\bm{x},\bm{y}),
$
where $\LL$ is a convex-concave function and $X,Y$ are reflexive Banach spaces, e.g. regularized finite-sum loss minimization, imaging models, and sequential two-player zero-sum games \citep{ChambollePock2011,kroer2018faster}. Even though convex duality often allows reformulating this saddle-point problem as a single convex program, first-order methods (FOMs) such as Chambolle \& Pock's Primal-Dual Algorithm (PDA, \cite{ChambollePock2011}), Mirror Descent \citep{nemirovsky1983problem} or Mirror Prox \citep{nemirovski2004prox}
are typically preferred for large instances. This is due to the expensive matrix calculations involved in IPMs or the simplex algorithm. 
Naively, one may hope to apply FOMs directly to the robust MDP problem, which looks superficially similar to the saddle-point problem. However, since the robust MDP problem is not convex-concave FOMs may fail to converge.

% the size of the resulting optimization problem may be very large. Therefore, it may not be possible to run Interior Point Methods (IPM), which require matrix inversions, on these very large instances. In contrast, the class of \textit{First-Order Methods} (FOM), including Primal-Dual algorithm (PDA, \cite{ChambollePock2011}, \cite{ChambollePock16}), Excessive Gap Technique (\cite{nesterov2005excessive}), Mirror Descent (\cite{nemirovsky1983problem}, \cite{beck2003mirror}) and Mirror Prox (\cite{nemirovski2004prox}, \cite{BenTal-Nemirovski}) has been proven efficient to solve saddle-point problems \eqref{eq:saddle-point}. However, these methods require $\LL$ to be a convex-concave function, which is not the case for the cost function $(\bm{x},\bm{y}) \mapsto R(\bm{x},\bm{y})$ (see more details in Section \ref{sec:RMDP}). Therefore, FOMs are not readily implementable for solving the Robust MDP problem \eqref{eq:Robust-MDP}.

Our main contributions can be summarized as follows.

\noindent{\bf A First-Order Method for Robust MDP.} We present a new algorithmic framework for solving robust MDPs that is significantly more scalable than previous methods. Our algorithm adapts FOMs for solving static zero-sum games to a dynamic setting with varying payoff matrices. Only cheap proximal mappings need to be computed at each iteration. Our framework interleaves FOM updates with occasional approximate VI updates. By carefully controlling the pace of VI updates, and developing bounds on the change in the payoff matrices of the zero-sum games, we show that the ergodic \textit{average} of \textit{policies} generated by our framework converges (in value) at a rate of nearly $\frac{1}{T}$ in terms of the number of FOM steps $T$. Note the critical difference with the classical analysis of Value Iteration algorithms, which rely on the \textit{last-iterate} convergence of the sequence of \textit{vector} iterates.
%it is possible to nearly get a $\frac{1}{T}$ convergence rate 
%We prove theoretical convergence guarantees for our algorithm, and show that the dependence on instance size is better than for prior methods. 
%In order to turn the robust MDP problem \eqref{eq:Robust-MDP} into a dynamic sequence of changing zero-sum games, 

\noindent{\bf Suitable Proximal Setups.} Our algorithmic framework is general and works for any uncertainty set for which a suitable proximal setup exists. We instantiate our algorithm on KL and ellipsoidal uncertainty sets. 
%We study four proximal setups of independent interest: $\ell_{1}$ and $\ell_{2}$ proximal setups for Cartesian products of simplexes intersected with either an $\ell_{2}$ or a KL ball. 
To the best of our knowledge, our algorithm is the first to address $s$-rectangular KL uncertainty sets for robust MDPs, and is the most scalable for ellipsoidal $s$-rectangular uncertainty sets in terms of number of states and actions.

\noindent{\bf Empirical Performance.} We focus our numerical experiments on ellipsoidal and KL uncertainty sets.  We investigate several proximal setups, and find that an $\ell_{2}$ setup performs better than an $\ell_{1}$ setup, despite better theoretical guarantees for the $\ell_{1}$ setup. Similar observations have been made for numerical performance on stationary saddle-point optimization~\citep{ChambollePock16,GKG20}. Finally, we show that our approach is significantly more scalable than state-of-the-art VI setups, both on random instances and on applications inspired from healthcare and machine replacement.

%We compare our algorithms to state-of-the-art VI setups, and show that our approach is significantly more scalable.
%To the best of our knowledge, our algorithm is the most scalable algorithm for solving Robust MDPs with $s$-rectangular ellipsoidal uncertainty sets.

\textbf{Related work.} We now present a summary of the most related literature.

\noindent\textbf{Approximate value iteration and Regularized MDP.} For the \textit{nominal} MDP setting,  approximate Value Iteration \citep{Farias-VanRoy,petrik2010optimization,scherrer2015approximate} considers inexact Bellman updates which arise from sample-based errors or function approximation; note that contrary to works involving value function approximation, we are solving the MDP up to any chosen accuracy.
\citet{geist2019theory} adds KL regularization to the VI update for nominal MDP and relate this to Mirror Descent. In contrast to \citet{geist2019theory}, we focus on robust MDPs, where VI requires computing the robust Bellman operator. Our framework is based on a connection to zero-sum games with changing payoff matrices and we use FOMs to efficiently approximate the robust Bellman update. This is very different from \citet{geist2019theory}, where regularized VI itself is treated as a FOM.

\noindent\textbf{Faster value iteration algorithms.} For nominal MDPs, several algorithms have been proposed to accelerate the convergence of VI, including Anderson mixing \citep{ref-a, ref-c} and acceleration and momentum schemes \citep{GGC-AVI}. However, these methods modify the VI algorithm itself, and do not accelerate the computation of each Bellman update.

\noindent\textbf{Faster Bellman updates.} For $s,a$-rectangular uncertainty sets, robust Bellman updates were studied for uncertainty sets defined by balls for the $\ell_{1}$ and weighed $\ell_{1}$ norms (\cite{Iyengar, Ho}), $\ell_{2}$ norm \citep{Iyengar}, KL-divergence \citep{Nilim} and $\ell_{\infty}$ norm \citep{givan1997bounded}.  To the best of our knowledge, the only paper on fast computation of robust Bellman updates for $s$-rectangular uncertainty sets is \citet{Ho} which considers weighted $\ell_{1}$-balls for $\PP_{s}$ and attains a complexity of $O \left( A S^{3} \log (A S^{2}) \log(\epsilon^{-1})\right)$ for finding an $\epsilon$-solution to the robust MDP problem.
%
% $O \left( S^{2} A \log (S^{2}A) \right)$ for the update $F(\bm{v})_{s}$ on each state $s$.
%Note that their algorithm solves $F(v)_{s}$ to optimality and does not depend of an accuracy $\epsilon$.
%Therefore, using their algorithm for the updates, the number of arithmetic operations to find an $\epsilon$-approximation to $\bm{v}^{*}$ is $ O \left( S^{3} A \log (S^{2}A) \log(\epsilon^{-1})\right).$
This complexity result relies on linear programming theory and cannot directly be extended to other settings for $\PP_{s}$ (e.g. ellipsoidal or KL uncertainty sets).

\section{Preliminaries on Robust MDP}\label{sec:RMDP}
\textbf{Notation.} We write $|\X|=S,|\A|=A$ and we assume $S < + \infty, A < + \infty.$ 
Given a policy $\bm{x} \in \Pi = \left(\Delta(A)\right)^{S}$ and a kernel $\bm{y} \in \PP$, we define the one-step cost vector $\bm{c}_{\bm{x}} \in \R^{S}$ and the  value vector $\bm{v}^{\bm{x},\bm{y}} \in \R^{S}$  as $c_{\bm{x},s} = \sum_{a=1}^{A} x_{sa}c_{sa}, v^{\bm{x},\bm{y}}_{s} = E^{\bm{x}, \bm{y}} \left[ \sum_{t=0}^{\infty} \lambda^{t}c_{s_{t}a_{t}} \; \bigg| \; s_{0} = s \right], \forall \; s \in \X$ .
% Note that we have the following key recursion:
%$ \bm{v}^{\bm{x},\bm{y}} = \bm{c}_{\bm{x}} + \lambda \bm{L}_{\bm{x},\bm{y}}\bm{v}^{\bm{x},\bm{y}},$
%which in turn implies that
%$\bm{v}^{\bm{x},\bm{y}} = \left( \bm{I} - \lambda \bm{L}_{\bm{x},\bm{y}} \right)^{-1}\bm{c}_{\bm{x}} = \sum_{t=0}^{+ \infty} \lambda^{t} \bm{L}^{t}_{\bm{x},\bm{y}} \bm{c}_{\bm{x}}.$
%The infinite-horizon discounted cost associated with a policy $\bm{x}$ and a kernel $\bm{y}$ defined in \eqref{eq-expreward} becomes
%\begin{equation}\label{eq:discounted-reward}
%R(\bm{x},\bm{y}) = \bm{p}_{0}^{\top}\bm{v}^{\bm{x},\bm{y}}.
%\end{equation}

\noindent
\textbf{Value Iteration.}
We first define Value Iteration (VI) for general $s$-rectangular uncertainty sets. Let $
\PP = \underset{s \in \X}{\times} \; \PP_{s}, \text{ for } \; \PP_{s} \subseteq \R^{|\A| \times |\X|}_{+}. 
$ and let us define the (robust) \textit{Bellman operator} $F: \R^{S} \rightarrow \R^{S}$, where for $\bm{v} \in \R^{S}$,
\begin{equation}\label{eq:T_max-def}
F(\bm{v})_{s} = \min_{\bm{x}_{s} \in \Delta(A)} \max_{\bm{y}_{s} \in \PP_{s}} \left\{ \sum_{a=1}^{A} x_{sa} \left( c_{sa} + \lambda \bm{y}_{sa}^{\top}\bm{v} \right) \right\},
\end{equation} 
for each $s \in \X$.
Note that with the notation $F^{\bm{x},\bm{y}}(\bm{v})_{s} = \sum_{a=1}^{A} x_{sa} \left( c_{sa} + \lambda \cdot \bm{y}_{sa}^{\top}\bm{v} \right)$, we can also write
$
F(\bm{v})_{s} = \min_{\bm{x}_{s} \in \Delta(A)} \max_{\bm{y}_{s} \in \PP_{s}} F^{\bm{x},\bm{y}}(\bm{v})_{s},
$
which shows that the robust VI update is a stationary saddle-point problem.
Solving the robust MDP problem is equivalent to computing $\bm{v}^{*}$, the fixed-point of $F$:
\begin{equation}\label{eq:v-star-fixed-point}
v_{s}^{*} = \min_{\bm{x}_{s} \in \Delta(A)} \max_{\bm{y}_{s} \in \PP_{s}} \{ \sum_{a=1}^{A} x_{sa} \left( c_{sa} + \lambda \cdot \bm{y}_{sa}^{\top}\bm{v}^{*} \right) \}, \forall \; s \in \X.
\end{equation}
Since $F$ is a contraction with factor $\lambda$, this can be done with the Value Iteration (VI) Algorithm:
\begin{equation}\label{alg:VI}\tag{VI}
\bm{v}_{0} \in \R^{S}, \bm{v}_{\ell+1} = F(\bm{v}_{\ell}), \forall \; \ell \geq 0.
\end{equation}

\noindent 
\ref{alg:VI} returns a sequence $\{\bm{v}_{\ell}\}_{\ell \geq 0}$ such that
$\| \bm{v}^{\ell+1} - \bm{v}^{*} \|_{\infty} \leq \lambda \cdot \| \bm{v}^{\ell} - \bm{v}^{*} \|_{\infty}, \forall \; \ell \geq 0.$
An optimal pair $(\bm{x}^{*},\bm{y}^{*})$ can be computed as any pair attaining the $\min \max$ in $F(\bm{v}^{*})$. An $\epsilon$-optimal pair can be computed as a solution to \eqref{eq:T_max-def}, when $\| \bm{v} - F(\bm{v}) \|_{\infty} <  \epsilon (1-\lambda) (2\lambda)^{-1}$ \citep{Puterman}. Our algorithm relies on approximately solving \ref{eq:T_max-def} as part of VI; controlling $\epsilon_{\ell}$, the  accuracy of epoch $\ell$ of VI, plays a crucial role in the analysis of our algorithm. In Appendix \ref{app:proof-sec-lemmas}, we present \textit{approximate} Value Iteration,  where the Bellman update $F(\bm{v}^{\ell})$ at epoch $\ell$ is only computed up to accuracy $\epsilon_{\ell}$.

\noindent \textbf{Ellipsoidal and KL uncertainty sets.}
We will show specific results for two types of $s$-rectangular uncertainty sets, though our algorithmic framework applies more generally, as long as appropriate proximal mappings can be computed. We consider KL $s$-rectangular uncertainty sets where $\PP_{s}$ equals
\begin{equation}\label{eq:KL-uncertainty-set}
 \{\left( \bm{y}_{sa} \right)_{a \in \A} \in \left(\Delta(S)\right)^{A} \; | \; \sum_{a \in \A} KL(\bm{y}_{sa},\bm{y}^{0}_{sa}) \leq \alpha\},
\end{equation}
and \textit{ellipsoidal} $s$-rectangular uncertainty sets
where $\PP_{s}$ equals
\begin{equation}\label{eq:uncertainty-norm-2}
  \begin{aligned}
\{\left( \bm{y}_{sa} \right)_{a \in \A} \in \left(\Delta(S)\right)^{A} \; | \; \sum_{a \in \A} \dfrac{1}{2} \| \bm{y}_{sa} - \bm{y}^{0}_{sa} \|_{2}^{2} \leq \alpha\}. 
  \end{aligned}
\end{equation}

Note that \eqref{eq:uncertainty-norm-2} is different from the ellipsoidal uncertainty sets considered in \citet{ben2000robust}, which also adds box constraints. However,  \citet{bertsimas2019probabilistic} shows that the same probabilistic guarantees exist for \eqref{eq:uncertainty-norm-2} as in the case of the uncertainty sets considered in \citet{ben2000robust}. For solving $s$-rectangular KL uncertainty sets, no algorithm is known (contrary to the significantly more conservative $s,a$-rectangular case). \citet{Kuhn} solves $s$-rectangular ellipsoidal uncertainty sets \eqref{eq:uncertainty-norm-2} using conic programs; we choose to instantiate VI differently in this case as follows. Using min-max convex duality twice, we can reformulate each  of the $S$ min-max programs \eqref{eq:T_max-def} into a larger convex program with linear objective and constraints, and one quadratic constraint (see \eqref{eq:T_v_simu} in Appendix \ref{app:simu}). Using IPMs each program can be solved up to $\epsilon$ accuracy in $O \left( A^{3.5}S^{3.5} \log(1/\epsilon) \right)$ arithmetic operations (\cite{BenTal-Nemirovski}, Section 4.6.2). Therefore, the complexity of \eqref{alg:VI} is
 \begin{equation}\label{eq:complexity-Kuhn}
O \left( A^{3.5} S^{4.5}\log^{2}(\epsilon^{-1}) \right).
\end{equation}
%In \cite{Kuhn}, the authors considers a factor model $\bm{y}_{sa} = \bm{y}^{0}_{sa} + \bm{K}_{sa} \bm{\xi}$, for $\bm{y}^{0}_{sa} \in \Delta(S), \bm{K} \in \R^{S \times q}, \bm{\xi} \in \Theta \subset \R^{q}$,  where $\Theta$ is defined by the intersection of $L$ quadratic constraints on $\bm{\xi}$. From Corollary 3 in \cite{Kuhn}, the number of arithmetic operations to find an $\epsilon$-approximation to $\bm{v}^{*}$ with an interior point method is
%\[ O \left( (q+A+L)^{1/2} (qL+A)^{3}S \log^{2}(\epsilon^{-1}) + q A S^{2} \log (\epsilon^{-1} ) \right). \]
%In all generality\footnote{With $\bm{\xi}=\left( \left(\bm{y}_{1a}\right)_{a \in \A},  ...,\left(\bm{y}_{Sa}\right)_{a \in \A} \right) \in \R^{S \times A \times S}$, $\bm{K}_{sa}$ a matrix with $S$ rows and $S \times A \times A$ columns and $\bm{0}_{S \times S}$ everywhere except the $(s,a)$-th block with $\bm{I}_{S \times S}$. We also need $L \geq S$ constraints to bind each $\left(\bm{y}_{sa} \right)_{a \in \A}$ for each $s \in \X$, since the set $\Theta$ is common across all $(s,a) \in \X \times \A$.},
% for an $s$-rectangular uncertainty set we have $q=S \times A \times S$.  This implies that the total number of arithmetic operations to return an $\epsilon$-approximation to $\bm{v}^{*}$ with interior point method is
%\begin{equation*}
%O \left( (S^{2}A+A+S)^{1/2} (S^{3}A+A)^{3}S \log^{2}(\epsilon^{-1}) +  A^{2} S^{3} \log (\epsilon^{-1} ) \right),
%\end{equation*}
%which is equivalent to
%\begin{equation}\label{eq:complexity-Kuhn}
%O \left( S^{11}A^{3.5}\log^{2}(\epsilon^{-1}) \right).
%\end{equation}
As mentioned earlier, this becomes intractable as soon as the number of states becomes on the order of hundreds, as highlighted in our numerical experiments of Section \ref{sec:simu}.
%Note that with $q=SA$, this gives 
%$O_{\lambda} \left( A^{3.5}S^{4.5}L^{3} \log^{2}(1/\epsilon) \right)$

%\begin{remark} For $\epsilon_{\ell} \sim 1/\ell$ that gives
%\[ \| \bm{v}^{\ell} - \bm{v}^{*} \|_{\infty} = O \left( \lambda^{\ell}  \| \bm{v}^{*} - \bm{v}^{0} \|_{\infty} + \dfrac{1}{\ell} \right),\]
%\[ \| \bm{v}^{\ell} - \bm{v}^{\ell-1} \|_{\infty} = O \left( \lambda^{\ell}  \| \bm{v}^{1} - \bm{v}^{0} \|_{\infty} + \dfrac{1}{\ell} \right).\]
%\end{remark}
%We will also use the following lemma in the analysis of our algorithm. It shows that as the iteration $\ell$ increases, $(\bm{x}^{\ell},\bm{y}^{\ell})$ are increasingly good candidates for the \emph{next} update $F(\bm{v}^{\ell+1})$. We present a detailed proof in Appendix \ref{app:proof-sec-lemmas}.
%\begin{lemma}\label{lem:warm-start-VI}
%Let $(\bm{x}^{\ell},\bm{y}^{\ell})$ be an $\epsilon_{\ell}$-solution to $F(\bm{v}^{\ell})$ at iteration $\ell \geq 1$. Let $\bm{v}^{\ell+1} = F^{\bm{x}^{\ell},\bm{y}^{\ell}}(\bm{v}^{\ell})$. 
%Then 
%\[  \| F(\bm{v}^{\ell+1}) - F^{\bm{x}^{\ell},\bm{y}^{\ell}}(\bm{v}^{\ell+1}) \|_{+ \infty} \leq 2 \lambda^{\ell+1} \left( \| \bm{v}^{1} - \bm{v}^{0} \|_{\infty} + \sum_{t=0}^{\ell} \dfrac{\epsilon_{t} + \epsilon_{t-1}}{\lambda^{t}} \right) + \epsilon_{\ell}.\]
%\end{lemma}
%%% Local Variables:
%%% mode: latex
%%% TeX-master: "../main"
%%% End:

\section{First-Order Methods for Robust MDPs}

We start by briefly introducing first-order methods (FOMs) in the context of our problem, and 
giving a high-level overview of our first-order framework for solving robust MDPs.
A FOM is a method that iteratively produces pairs of solution candidates $\bm{x}^t,\bm{y}^t$, where the $t$'th solution pair is derived from $\bm{x}^{t-1},\bm{y}^{t-1}$ combined with a first-order approximation to the direction of improvement at $\bm{x}^{t-1},\bm{y}^{t-1}$. Using only first-order information is desirable for large-scale problems because second-order information eventually becomes too slow to compute, meaning that even a single iteration of a second-order method ends up being intractable.
See e.g. \citet{beck2017first} or \citet{BenTal-Nemirovski} for more on FOMs.

Our algorithmic framework is based on the observation that there exists a collection of matrices $\bm{K}^{*}_{s}: \Delta(A) \times \PP_{s} \rightarrow \R$, for $s \in \X$, such that computing an optimal solution $\bm{x}^{*},\bm{y}^{*}$ to the robust MDP problem boils down to solving $S$ bilinear saddle-point problems (BSPPs), each of the form
\begin{align}
\min_{\bm{x}_{s} \in \Delta(A)} \max_{\bm{y}_{s} \in \PP_{s}} \langle \bm{c}_s, \bm{x}_s \rangle +  \langle \bm{K}^*_s\bm{x}_s, \bm{y}_s \rangle.
  \label{eq:bspp}
\end{align}

This is a straightforward consequence of the Bellman equation \ref{eq:T_max-def} and its reformulation using $F^{\bm{x},\bm{y}}$. The matrix $\bm{K}^*_s$ is the payoff matrix associated with the optimal value vector $v^*$. If we knew $\bm{K}^*_s$, then we could solve \eqref{eq:bspp} by applying existing FOMs for solving BSPPs.

Now, obviously we do not know $\bm{v}^*$ before running our algorithm. However, we know that Value Iteration constructs a sequence $\{\bm{v}_{\ell}\}_{\ell \geq 0}$ which converges to $\bm{v}^*$. Letting  $\{\bm{K}^{\ell}_s\}_{\ell \geq 0}$ be the associated payoff matrices for each value-vector estimate $\bm{v}_{\ell}$ and state $s$, we thus have a sequence of payoff matrices converging to $\bm{K}^*_s$ for each $s$. We will apply a FOM to such a sequence of BSPPs $\{\bm{K}^{\ell}_s\}_{\ell \geq 0}$ based on approximate Value Iteration updates.

Our algorithmic framework, which we call FOM-VI, works as follows. We utilize an existing primal-dual FOM for solving problems of the form \eqref{eq:bspp}, where the FOM should be of the type that generates a sequence of iterates $\bm{x}_t,\bm{y}_t$, with an ergodic convergence rate on the time-averaged iterates.
Even though such FOMs are designed for a \emph{fixed} BSPP with a single payoff matrix $\bm{K}$, we apply the FOM updates to a changing sequence of payoff matrices $\{\bm{K}^{\ell}_s\}_{\ell = 1}^k$. For each payoff matrix $\bm{K}^{\ell}_{s}$ we apply $T_\ell$ iterations of the FOM, after which we apply an approximate VI update to generate $\bm{K}^{\ell+1}_{s}$. We refer to each step $\ell$ with a payoff matrix $\bm{K}^{\ell}_{s}$ as an \emph{epoch}, while \emph{iteration} refers to steps of our FOM. We will apply many iterations per epoch.

The convergence rate of our algorithm is, intuitively, based on the following facts: (i) the average of the iterates generated during epoch $\ell$ provides a good estimate of the VI update associated with $\bm{v}^\ell$, and (ii) the sequence of payoff matrices generated by the approximate VI updates is changing in a controlled manner, such that $\bm{K}^{\ell}_{s}$ and $\bm{K}^{\ell+1}_{s}$ are not too different.

These facts allow us to show that the averaged strategy across \emph{all} epochs $\ell$ converges to a solution to \eqref{eq:bspp} without too much degradation in the convergence rate compared to having run the same number of iterations directly on \eqref{eq:bspp}. 

\subsection{First-Order Method Setup}
In this paper, we use the PDA algorithm of \citet{ChambollePock16} as our FOM, but the derivations could also be performed with other FOMs whose convergence rate is based on applying a telescoping argument to a sum of descent inequalities, e.g.  mirror prox of~\citet{nemirovski2004prox} or saddle-point mirror descent of~\citet{BenTal-Nemirovski}; the latter would yield a slower rate of convergence.

We now describe PDA as it applies to BSPPs such as \eqref{eq:bspp}, for an arbitrary payoff matrix $\bm{K}$ and some state $s$.
PDA relies on what we will call a \emph{proximal setup}. A proximal setup consists of a set of norms $\| \cdot \|_X, \| \cdot \|_Y$ for the spaces of $\bm{x}_s,\bm{y}_s$, as well as 
\emph{distance-generating functions} $\psi_{X}$ and $\psi_Y$, which are  1-strongly convex with respect to $\| \cdot \|_{X}$ on $\Delta(A)$ and  $\| \cdot \|_{Y}$ on $\mathbb{P}_s$, respectively. 
Using the distance-generating functions,  PDA uses the \emph{Bregman divergence}
\begin{align*}
  D_{X}(\bm{x},\bm{x'}) = \psi_{X}(\bm{x'}) - \psi_{X}(\bm{x}) - \langle \nabla \psi_{X}(\bm{x}), \bm{x'} - \bm{x} \rangle
\end{align*}
to measure the distance between two points $\bm{x}, \bm{x'} \in \Delta(A)$.
The Bregman divergence $D_Y$ is defined analogously.

The convergence rate of PDA then depends on the maximum Bregman divergence distance $\Theta_{X}=\max_{\bm{x},\bm{x'}\in \Delta(A)}D_{X}(\bm{x},\bm{x'})$ between any two points, and the maximum norm $R_{X}=\max_{\bm{x}\in \Delta(A)}\|\bm{x}\|_X$ on $\Delta(A)$. The quantities $\Theta_{Y}$ and $R_Y$ are defined analogously on $\mathbb{P}_s$.

Given $D_X$ and $D_Y$, the associated \emph{prox mappings} are 
\begin{align*}
  \textrm{prox}_x(\bm{g}_x, \bm{x'}_s) &= \arg\min_{\bm{x}_s \in \Delta(A)} \langle \bm{g}_x,\bm{x} \rangle + D_X(\bm{x}_s, \bm{x'}_s),\\
  \textrm{prox}_y(\bm{g}_y, \bm{y'}_s) &= \arg \max_{\bm{y}_s \in \mathbb{P}_s} \langle \bm{g}_y ,\bm{y}_s\rangle - D_{Y}(\bm{y}_s,\bm{y'}_s).
\end{align*}
Intuitively, the prox mappings generalize taking a step in the direction of the negative gradient, as in gradient descent. Given some gradient $\bm{g}_x$, $\textrm{prox}_x(\bm{g}_x, \bm{x'}_s)$ moves in the direction of improvement, but is penalized by the Bregman divergence $D_X(\bm{x}_s,\bm{x'}_s)$, which attempts to ensure that we stay in a region where the first-order approximation is still good.

Given step sizes $\tau,\sigma > 0$ and current iterates $\bm{x}_s^t, \bm{y}_s^t$,
PDA generates the iterates for $t+1$ by taking prox steps in the negative gradient direction given the current strategies:
\begin{equation} \label{eq:pda_iterates} 
\begin{aligned}
\bm{x}_s^{t+1} & = \textrm{prox}_x(\tau\bm{K}^\top\bm{y}^t_s, \bm{x}_s^t),\\
\bm{y}_s^{t+1} & = \textrm{prox}_y(\sigma\bm{K}(2\bm{x}^{t+1}_s - \bm{x}^t_s), \bm{x}_s^t),\\
% \bm{x}_s^{t+1} & = \arg \min_{\bm{x}_s \in \Delta(A)} \langle \bm{K}\bm{x}_s,\bm{y}^t_s\rangle + \dfrac{1}{\tau}D_{X}(\bm{x}_s, \bm{x}_s^t),\\
% \bm{y}_s^{t+1} & = \arg \max_{\bm{y}_s \in \mathbb{P}_s} \langle \bm{K}(2\bm{x}_s^{t+1} - \bm{x}_s^t) ,\bm{y}_s\rangle - \dfrac{1}{\sigma} D_{Y}(\bm{y}_s,\bm{y}_s^t).
\end{aligned}
\end{equation}
Note that for the $\bm{y}_s^{t+1}$ update, the ``direction of improvement'' is measured according to the extrapolated point $2\bm{x}^{t+1}_s - \bm{x}^t_s$, as opposed to at either $\bm{x}^t_s$ or $\bm{x}^{t+1}_s$. If a simpler single current iterate is used to take the gradient for $\bm{y}_s^{t+1}$, then the overall PDA setup yields an algorithm that converges at a $O(1/\sqrt{T})$ rate. The extrapolation is used to get a stronger $O(1/T)$ rate.

%  let $ (\hat{\bm{x}},\hat{\bm{y}}) = PD_{\sigma,\tau}(\bm{x'},\bm{x''},\bm{y'},\bm{y''})$ where
% \begin{align*}
% \hat{\bm{x}}_s & = \arg \min_{\bm{x}_s \in \Delta(A)} \langle \bm{K}\bm{x}_s,\bm{y''}_s\rangle + \dfrac{1}{\tau}D_{X}(\bm{x}_s,\bm{x'}_s),\\
% \hat{\bm{y}}_s & = \arg \min_{\bm{y}_s \in \mathbb{P}_s} - \langle \bm{K}\bm{x''}_s,\bm{y}_s\rangle + \dfrac{1}{\sigma} D_{Y}(\bm{y}_s,\bm{y'}_s).
% \end{align*}

%  PDA repeats the iteration
% \begin{equation}\label{eq:PD-update}\tag{PDA}
% \left(\bm{x}^{t+1}_s, \bm{y}^{t+1}_s \right) = PD_{\sigma,\tau} \left(\bm{x}^{t}_s,\bm{y}^{t}_s, 2\bm{x}^{t+1}_s-\bm{x}^{t}_s, \bm{y}^{t}_s \right).
% \end{equation} 

Let
$\Omega = 2 \left(\Theta_{X}/\tau + \Theta_{Y}/\sigma \right)$ and let $\tau,\sigma > 0$ be such that, for $L_{\bm{K}} \geq \sup_{\| \bm{x} \|_{X} \leq 1, \| \bm{y} \|_{Y} \leq 1} \langle \bm{Kx},\bm{y} \rangle$, we have
\begin{equation}\label{eq:step-size-general}
\left( \dfrac{1}{\tau} - L_{f} \right) \dfrac{1}{\sigma} \geq L_{\bm{K}}^{2}.
\end{equation} 

After  $T$ iterations of PDA, we can construct weighted averages $(\bm{\bar{x}}^{T},\bm{\bar{y}}^{T}) = (1/S_{T}) \sum_{t=1}^{T} \omega_{t} (\bm{x}_{t},\bm{y}_{t})$ of all iterates, using weights $\omega_{1},...,\omega_{T}$ and normalization factor $S_{T} = \sum_{t=1}^{T} \omega_{t}$.
In the case of a static BSPP, if the stepsizes are chosen such that they satisfy \eqref{eq:step-size-general}, then the average of the iterates from PDA satisfies the convergence rate:
\[ \max_{\bm{y} \in \PP_{s} } \langle \bm{K}\bm{\bar{x}}^{T},\bm{y}  \rangle-  \min_{\bm{x} \in \Delta(A)} \langle \bm{K}\bm{x},\bm{\bar{y}}^{T} \rangle \leq \Omega  \omega_{T}/S_{T}.\]

Here we are using the weighted average of iterates, as in  \citet{GKG20}, see Appendix \ref{app:lem-norm-K}.
Since PDA applies the two prox mappings \eqref{eq:pda_iterates} at every iteration, it is crucial that these prox mappings can be computed efficiently. Ideally, in time roughly linear in the dimension of the iterates. A significant part of our contribution is to show that this is indeed the case for several important types of uncertainty sets.

In our setting, where the payoff matrix $\bm{K}^\ell_s$ in the BSPP is changing over time, the existing convergence rate for PDA does not apply. Instead, we have to consider how to deal with the error that is introduced in the process due to the changing payoffs.

% \begin{proposition}[\cite{ChambollePock16,GKG20}]\label{prop:PD-descent-lemma}
% Let
% $\Omega = 2 \left(\Theta_{X}/\tau + \Theta_{Y}/\sigma \right)$ and let $\tau,\sigma > 0$ be such that, for $L_{\bm{K}} \geq \sup_{\| \bm{x} \|_{X} \leq 1, \| \bm{y} \|_{Y} \leq 1} \langle \bm{Kx},\bm{y} \rangle$, we have
% \begin{equation}\label{eq:step-size-general}
% \left( \dfrac{1}{\tau} - L_{f} \right) \dfrac{1}{\sigma} \geq L_{\bm{K}}^{2}.
% \end{equation} 
% Then, after  $T$ iterations of PDA, the weighted average $(\bm{\bar{x}}^{T},\bm{\bar{y}}^{T}) = (1/S_{T}) \sum_{t=1}^{T} \omega_{t} (\bm{x}_{t},\bm{y}_{t})$ of all iterates, using weights $\omega_{1},...,\omega_{T}$ and normalization factor $S_{T} = \sum_{t=1}^{T} \omega_{t}$, satisfies the following inequality:
% \[ \max_{\bm{y} \in \PP_{s} } \langle \bm{K}\bm{\bar{x}}^{T},\bm{y}  \rangle-  \min_{\bm{x} \in \Delta(A)} \langle \bm{K}\bm{x},\bm{\bar{y}}^{T} \rangle \leq \Omega  \omega_{T}/S_{T}.\]
% \end{proposition}

\subsection{First-Order Method Value Iteration (FOM-VI)}
We now describe our algorithm in detail, as well as the choices of $\| \cdot \|_{X}, \|\cdot \|_{Y}$ that lead to tractable proximal updates \eqref{eq:pda_iterates}. As we have described in \eqref{eq:bspp}, given a vector $\bm{v} \in \R^{S}$ and $s \in \X$, the matrix $\bm{K}_{s} \in \R^{A \times A \times S}$ is defined such that
\[ F^{\bm{x},\bm{y}}_{s}(\bm{v}) =\langle \bm{c}_s, \bm{x}_s \rangle +  \langle \bm{K}_{s}\bm{x}_s, \bm{y}_s \rangle.\] 
We will write this as $\bm{K} = \bm{K}[\bm{v}]$. 
The pseudocode for the FOM-VI algorithm is given in Algorithm~\ref{alg:PD-RMDP}.
\begin{algorithm}[]
\caption{First-order Method for Robust MDP with $s$-rectangular uncertainty set.}\label{alg:PD-RMDP}
\begin{algorithmic}[1]
  \State \textbf{Input}
  Number of epochs $k$, number of iterations per epoch $T_{1}, ..., T_{k}$, weights $\omega_{1}, ..., \omega_{T}$, and stepsizes $\tau,\sigma$.
\State\textbf{ Initialize} $\ell=1, \bm{v}^{\ell}=\bm{0},$ and $\bm{x}^{0}, \bm{y}^{0}$ at random.
\For{epoch $ \ell=1,..., k$}
\For{$s \in \X$} 
% \State  Set 
% %\begin{equation}\label{eq:PD-ell-s}
% $
% \bm{K}^{\ell} = \bm{K}[\bm{v}^{\ell}].
% $
%\end{equation}
\State Set $\tau_{\ell} = T_{1} + ... + T_{\ell-1}$
\For{$t = \tau_\ell, \ldots, \tau_\ell+T_\ell$}
\State $\bm{x}_s^{t+1} = \textrm{prox}_x(\tau\bm{K}_s[\bm{v}^{\ell}]^\top\bm{y}^t_s, \bm{x}^t_s)$
\State $\bm{y}_s^{t+1} = \textrm{prox}_y(\sigma\bm{K}_s[\bm{v}^{\ell}](2\bm{x}^{t+1}_s - \bm{x}^t_s), \bm{y}^t_s)$
% \State $\bm{x}_s^{t+1} & = \arg \min_{\bm{x}_s \in \Delta(A)} \langle \bm{K}\bm{x}_s,\bm{y}^t_s\rangle + \dfrac{1}{\tau}D_{X}(\bm{x}_s, \bm{x}_s^t)$
% \State $\bm{y}_s^{t+1} & = \arg \max_{\bm{y}_s \in \mathbb{P}_s} \langle \bm{K}(2\bm{x}_s^{t+1} - \bm{x}_s^t) ,\bm{y}_s\rangle - \dfrac{1}{\sigma} D_{Y}(\bm{y}_s,\bm{y}_s^t)$
\EndFor
% \State Starting with $\tau_{\ell} = T_{1} + ... + T_{\ell-1}$, compute $(\bm{x}_{\tau_{\ell}+1},\bm{y}_{\tau_{\ell}+1}), ..., (\bm{x}_{\tau_{\ell}+T_{\ell}},\bm{y}_{\tau_{\ell}+T_{\ell}})$ by running  $T_{\ell}$ iterations of PDA with payoff matrix $\bm{K}^\ell$, where iteration $\tau_\ell$ starts at $(\bm{x}_{\tau_{\ell}},\bm{y}_{\tau_{\ell}})$. \label{alg:PD-T-ell-times}
\State $S_\ell = \sum_{t = \tau_\ell}^{\tau_{\ell}+T_{\ell}} \omega_t$ \label{alg:normalization const}
\State $\bar{\bm{x}}^{\ell}_{s} = \sum_{t = \tau_\ell}^{\tau_{\ell}+T_{\ell}} \frac{\omega_t}{S_\ell} \bm{x}_s^t, \bar{\bm{y}}^{\ell}_{s} = \sum_{t = \tau_\ell}^{\tau_{\ell}+T_{\ell}} \frac{\omega_t}{S_\ell}\bm{y}_s^t$  \label{alg:PD-compute-averages}
\State  Update $v^{\ell+1}_{s} = F^{\bar{\bm{x}}^{\ell}_{s},\bar{\bm{y}}^{\ell}_{s}}(\bm{v}^{\ell})_{s}$. \label{alg:PD-udate-v-ell}
	 \EndFor
\EndFor
\State \textbf{Output} $\bar{\bm{x}}^{T}_{s} = \sum_{t = 1}^{T} \frac{\omega_t}{S_T} \bm{x}_s^t, \bar{\bm{y}}^{T}_{s} = \sum_{t = 1}^{T} \frac{\omega_t}{S_T}\bm{y}_s^t, \forall s \in \mathbb{S}$  
 
\end{algorithmic}
\end{algorithm}
At each epoch $\ell$, we have some current estimate $v^\ell$ of the value vector, which is used to construct the payoff matrix $\bm{K}^\ell$ for the $\ell$'th BSPP. For each state $s$, we then run $T_\ell$ iterations of PDA, where, crucially, the first such iteration starts from the last iterates $(\bm{x}_{\tau_{\ell}},\bm{y}_{\tau_{\ell}})$ generated at the previous epoch. The average iterate constructed from just these $T_\ell$ iterations is then used to construct the next value vector $v^{\ell+1}_s$ via an approximate VI update (lines \ref{alg:PD-compute-averages} and \ref{alg:PD-udate-v-ell}). Finally, after the last epoch $k$, we output the average of all the iterates generates across all the epochs, using the weights $\omega_{1}, ..., \omega_{T}$.

We prove that FOM-VI satisfies the following convergence rate. We state our results for the two special cases where the norms $\| \cdot \|_X$ and $\| \cdot \|_Y$ are both equal to the $\ell_1$ norm (we call this the \emph{$\ell_1$ setup}) or $\ell_2$ norm (\emph{$\ell_2$ setup}) on the spaces $\Delta(A), \mathbb{P}_s$. FOM-VI could also be instantiated with other norms.
The proof is in Appendix~\ref{app:PD-ABCDE}.
\begin{theorem} \label{thm:fom-vi rate} 
  Assume that the stepsizes $\tau, \sigma$ are such that \eqref{eq:step-size-general} holds, and for each epoch $\ell$, we set $T_\ell = \ell^q$ for some $q \in \mathbb{N}$.
 Let $\bar{\bm{x}}^{T},\bar{\bm{y}}^{T}$ be the averages of the FOM-VI iterates using the weights $w_{1},...,w_{T}$. %, and let $S_{T} = \sum_{t=1}^{T} \omega_{t}$.
Then for all states $s \in \X$,
\begin{align*}
 \max_{\bm{y} \in \PP_{s}} F^{\bar{\bm{x}}^{T},\bm{y}}(\bm{v}^{*})_{s} - \min_{\bm{x} \in \Delta(A)} F^{\bm{x}, \bar{\bm{y}}^{T}}(\bm{v}^{*})_{s} 
  \leq  O\left( C R_{X}R_{Y} \left( \dfrac{\Theta_{X}}{\tau} + \dfrac{\Theta_{Y}}{\sigma} \right) \left( \dfrac{\lambda^{T^{1/(q+1)}}}{T^{1/(q+1)}} + \dfrac{1}{T^{q/(q+1)}} \right) \right),
\end{align*}
with $C=1$ in the $\ell_1$ setup, and $C=\sqrt{S}$ in the $\ell_2$ setup.
\end{theorem}

\subsection{Tractable proximal setups for Algorithm \ref{alg:PD-RMDP}}\label{sec:tractable-prox-updates}

In the previous section we saw that FOM-VI instantiated with appropriate proximal setups yields an attractive convergence rate. For a given  proximal setup, the convergence rate in Theorem~\ref{thm:fom-vi rate} depends on the maximum-norm quantities $R_X, R_Y$ and the polytope diameter measures $\Theta_X,\Theta_Y$. However, another important issue was previously not discussed: in order to run FOM-VI we must compute the iterates $x_s^{t+1},y_s^{t+1}$, which means that the updates in \eqref{eq:pda_iterates} must be fast to compute (ideally in closed form).
We next present several tractable proximal setups for Algorithm \ref{alg:PD-RMDP}. %For the sake of space, we present our results for proximal updates for KL uncertainty sets in Appendix \ref{app:KL-uncertainty-set}.

\noindent
\textbf{Tractable updates for $\Delta(A)$.}
Since decision space for $x$ is a simplex, we can apply well-known results to get a proximal setup.
For the $\ell_2$ setup (i.e. where $\| \cdot \|_X = \| \cdot \|_2$), we can set $\psi_X(\bm{x}) = (1 / 2)\|\bm{x}\|_2^2$, in which case $D_X$ is the squared Euclidean distance. For this setup, $\Theta_X = 1$, and $\bm{x}_s^{t+1}$ can be computed in $A\log A$ time, using a well-known algorithm based on sorting~\citep{BenTal-Nemirovski}. 
%The Lipschitz constant $L_{\bm{K}}$ is on the order of $O(\sqrt{n} + \sqrt{m})$.

For the $\ell_{1}$ setup, (i.e. where $\| \cdot \|_X = \| \cdot \|_1$), we set $\psi_X(\bm{x}) = \textsc{Ent}(\bm{x}) \defeq \sum_i x_i \log x_i$ (i.e. the negative entropy), in which case $D_X$ is the \emph{KL divergence}. The advantage of this setup is that the strong convexity is with respect to the $\ell_1$ norm, which makes the Lipschitz associated to the payoff matrix a constant (as opposed to $\sqrt{S}$ for the $\ell_2$ norm), while the polytope diameter is only $\Theta_X = \log A$. Finally, $\bm{x}_s^{t+1}$ can be computed in closed form. Thus, from a theoretical perspective, the $\ell_1$ setup is more attractive than the $\ell_2$ setup for $\Delta(A)$. This is well-known in the literature.

In all cases, $R_X = 1$, since $\bm{x}$ comes from a simplex.
% The proximal updates for the $x$-player are the classical proximal updates for the simplex (\cite{BenTal-Nemirovski}). 
% For the $\ell_{2}$ setup, the proximal update \eqref{eq:prox_update_x_simple} boils down to computing the Euclidean projection of a vector onto the simplex of dimension $A$, and for the $\ell_{1}$ setup, it boils down to optimizing the sum of a linear form and a KL divergence onto the simplex. %We present a summary in Table \ref{tab:prox-setups}.

\noindent
\textbf{Tractable updates for ellipsoidal uncertainty.}
The proximal updates for $y$ turn out to be more complicated. In the first place, they depend heavily on the form of $\mathbb{P}_s$. First, we present our results for the case where $\mathbb{P}_s$ is an ellipsoidal s-rectangular uncertainty set as in \eqref{eq:uncertainty-norm-2}.
We present both $\ell_1$ and $\ell_2$ setups.

In the  $\ell_2$ setup for ellipsoidal uncertainty, 
we let $\| \cdot \|_{Y} $ be the $\ell_{2}$ norm, and $ \psi_{Y}(\bm{y}) = (1/2)\|\bm{y}\|_2^2$.  The Bregman divergence $D_{Y}(\bm{y},\bm{y'})$ is then simply the squared Euclidean distance.
In this case, we get that $R_Y = \sqrt{A}$, since the squared norm of each individual simplex is at most one, and then we take the square root. The polytope diameter is $\Theta_Y = 2 A$ for the same reason. We show in Proposition~\ref{prop:prop-fast-y-update-ell2-ell1} below that the iterate $\bm{y}^{t+1}_s$ can be computed efficiently.

In the  $\ell_1$ setup for ellipsoidal uncertainty, 
we let $\| \cdot \|_{Y} $ be the $\ell_{1}$ norm, and $ \psi_{Y}(\bm{y}) = (A/2) \sum_{a=1}^A \textsc{Ent}(\bm{y}_a)$, where $\textsc{Ent}(\bm{y}_a)$ is the negative entropy function.  The Bregman divergence $D_{Y}(\bm{y},\bm{y'}) = (A/2) \sum_{a=1}^A \textsc{KL}(\bm{y}_a, \bm{y'}_a)$ is then a sum over KL divergences on each action.
In this case, we get that $R_Y = A$, since we are taking the $\ell_1$ norm over $A$ simplexes, while the polytope diameter is $\Theta_Y = A^2\log S$. %We show in Proposition~\ref{prop:prop-fast-y-update-ell2-ell1} below that the iterate $\bm{y}^{t+1}_s$ can be computed efficiently.

%In this section we focus on the following \textit{ellipsoidal} uncertainty set to model uncertainty in the transition kernel. We present our results for KL uncertainty set in Appendix \ref{app:KL-uncertainty-set}.
% Due to space constraints we focus on our result for ellipsoidal uncertainty sets here. 
%We present our results for KL uncertainty set in Appendix \ref{app:KL-uncertainty-set}.
Proposition \ref{prop:prop-fast-y-update-ell2-ell1} shows that for both our $\ell_{2}$-based and $\ell_{1}$-based setup for $y$, the next iterate can be computed efficiently.
We present a detailed proof in Appendix \ref{app:prop-fast-y-update}.
% We let $ \textsc{Ent}(\bm{y}) = \sum_{s'=1}^{S} y_{s'}\log(y_{s'})$ and $\textsc{KL}(\bm{y},\bm{y'}) = \sum_{s'=1}^{S} y_{s'}\log(y_{s'}/y'_{s'})$ for $(\bm{y},\bm{y'}) \in \Delta(S) \times \Delta(S).$
\begin{proposition}\label{prop:prop-fast-y-update-ell2-ell1}
 % ($\ell_{2}$ setup). Let $\| \cdot \|_{Y} = \| \cdot \|_{2}, \psi_{Y} = (1/2) \| \cdot \|^{2}_{2}, D_{Y}(\bm{y},\bm{y'}) = (1/2) \| \bm{y} - \bm{y'} \|^{2}_{2}$.
For the $\ell_2$ setup, the proximal update \eqref{eq:pda_iterates} with uncertainty set \eqref{eq:uncertainty-norm-2} can be approximated up to $\epsilon$ in a number of arithmetic operations of $O\left( AS \log(S) \log(\epsilon^{-1}) \right)$.

% ($\ell_{1}$ setup) Let $\| \cdot \|_{Y} = \| \cdot \|_{1}, \psi_{Y}(y) = (A/2) \sum_{a=1}^A \textsc{Ent}(\bm{y}_a), D_{Y}(\bm{y},\bm{y'}) = (A/2) \sum_{a=1}^A \textsc{KL}(\bm{y}_a, \bm{y'}_a)$.
%,where
%$ \textsc{entropy}(\bm{y}) = \sum_{s'=1}^{S} y_{s'}\log(y_{s'}),  \textsc{KL}(\bm{y},\bm{y'}) = \sum_{s'=1}^{S} y_{s'}\log(y_{s'}/y'_{s'}), \forall \; (\bm{y},\bm{y'}) \in \Delta(S) \times \Delta(S).$
For the $\ell_1$ setup,  the proximal update \eqref{eq:pda_iterates} with uncertainty set \eqref{eq:uncertainty-norm-2} can be approximated up to $\epsilon$ in a number of arithmetic operations of $O\left( AS \log^2(\epsilon^{-1})\right)$.
\end{proposition}

\noindent
\textbf{Tractable updates for KL uncertainty.}
As in the case of ellipsoidal uncertainty, we present both $\ell_1$ and $\ell_2$ setups for KL uncertainty. The setups are exactly the same as for ellipsoidal uncertainty (i.e. same norms, distance functions, and Bregman divergences), and all constants remain the same. The reason that all constants remain the same is because our bounds on the maximum norms and $\Theta_Y$, for both uncertainty set types, are based on bounding these values over the bigger set consisting of the Cartesian product of $A$ simplexes. The question thus becomes whether \eqref{eq:pda_iterates} can be computed efficiently (for $y$) when $D_{Y}$ is the sum over KL divergences on each action. We present our results in the following proposition; a detailed proof can be found in Appendix \ref{app:KL-uncertainty-set}.

\begin{proposition}\label{prop:prop-fast-y-update-KL}
% \begin{enumerate}
% \item ($\ell_{2}$ setup) Let $\| \cdot \|_{Y} = \| \cdot \|_{2}, \psi_{Y} = (1/2) \| \cdot \|^{2}_{2}, D_{Y}(\bm{y},\bm{y'}) = (1/2) \| \bm{y} - \bm{y'} \|^{2}_{2}, \forall \; (\bm{y},\bm{y'}) \in \Delta(S) \times \Delta(S).$

For the $\ell_2$ setup, the proximal update \eqref{eq:pda_iterates} with uncertainty set \eqref{eq:KL-uncertainty-set} can be approximated up to $\epsilon$ in a number of arithmetic operations in $O\left( AS \log^2(\epsilon^{-1})\right)$.
% \item ($\ell_{1}$ setup) Let $\| \cdot \|_{Y} = \| \cdot \|_{1},$
% \[ \psi_{Y}(\bm{y}) = (A/2) \sum_{a=1}^A \textsc{entropy}(\bm{y}_a),\]
% \[ D_{Y}(\bm{y},\bm{y'}) = (A/2) \sum_{a=1}^A \textsc{KL}(\bm{y}_a, \bm{y'}_a).\]

For the $\ell_1$ setup, the proximal update \eqref{eq:pda_iterates} with uncertainty set \eqref{eq:KL-uncertainty-set} can be approximated up to $\epsilon$ in a number of arithmetic operations in $O\left( AS \log(\epsilon^{-1}) \right)$.
% \end{enumerate}
\end{proposition}

\begin{remark} At a cursory reading, our results in Propositions \ref{prop:prop-fast-y-update-ell2-ell1} and \ref{prop:prop-fast-y-update-KL} may seem similar to those of \cite{Nilim} and \cite{Iyengar}. Both authors introduce bisection algorithms for computing Bellman updates, but these are for the simpler case of $(s,a)$-rectangular uncertainty sets. In that case, the Bellman updates can be computed by enumerating the set of actions $a \in \A$, since an optimal solution exists among the set of pure actions. In contrast, in our setting the optimal $x\in \Delta(A)$ may require randomization, which is why we must solve a min-max problem as in \eqref{eq:T_max-def}.
  % We would like to highlight the difference between our results of Propositions \ref{prop:prop-fast-y-update-ell2-ell1}-\ref{prop:prop-fast-y-update-KL} and the results of \cite{Nilim}, Section 6, who introduce a bisection algorithm to compute the Bellman update for $(s,a)$-rectangular, KL uncertainty set, which reduces to computing the minimum over $a \in \A$ of the optimization programs $\max_{\bm{y} \in \U_{sa}} \bm{y}^{\top}\bm{v}$, for $\U_{sa} = \{ \bm{y} \in \Delta(S) | KL(\bm{y},\bm{y}^{0}_{sa}) \leq \alpha \}.$  Also, \cite{Iyengar} introduces a bisection algorithm for the Bellman update for $(s,a)$-rectangular, ellipsoidal uncertainty set, i.e. for $\U_{sa} = \{ \bm{y} \in \Delta(S) | (1/2)\| \bm{y} - \bm{y}^{0}_{sa} \|_{2}^{2} \leq \alpha \}$. We would like to note that their objective function is a linear form, while the objective of the proximal update \eqref{eq:prox_update_y_simple} involves the sum of a linear form and a Bregman divergence. Moreover, we consider $s$-rectangular uncertainty set, for which the Bellman update remains a min-max formulation optimization program (and does not boil down to an enumeration of maximization optimization programs).
\end{remark}

\subsection{Complexity of Algorithm \ref{alg:PD-RMDP}}
Armed with our various proximal setups, we can finally state the performance guarantees provided by FOM-VI explicitly for the various setups. Since the constants for the $\ell_1$ and $\ell_2$ setups are the same for both KL and ellipsoidal uncertainty sets, we start by stating a single theorem which gives a bound on the error after $T$ iterations for either type of uncertainty set.
The following theorem works for any polynomial scheme for choosing the iterate weights when averaging, as well as how many FOM iterations to perform in-between each VI update. Details are given in Appendix \ref{app:complexity-table}. 
\begin{theorem}\label{th:main-complexity} Let $p,q \in \N$ and at time step $t \geq 0$, let  the iterate weight be $\omega_{t} = t^{p}$, and the number of FOM iterations at epoch $\ell$ be $T_{\ell} = \ell^{q}.$ After $T$ iterations of Algorithm \ref{alg:PD-RMDP}, $ \max_{\bm{y} \in \PP_{s}} F^{\bar{\bm{x}}^{T},\bm{y}}(\bm{v}^{*})_{s} - \min_{\bm{x} \in \Delta(A)} F^{\bm{x}, \bar{\bm{y}}^{T}}(\bm{v}^{*})_{s} $ is upper bounded by
\begin{itemize}
\item $O \left(A^{2} \sqrt{\dfrac{\log(S)}{\log(A)}} \left( \dfrac{1}{T^{q/(q+1)}}+ \dfrac{\lambda^{T^{1/(q+1)}} }{T^{1/(q+1)}} \right) \right)$ in the $\ell_{1}$ setup,
\item $O \left( AS \left( \dfrac{1}{T^{q/(q+1)}}+ \dfrac{\lambda^{T^{1/(q+1)}} }{T^{1/(q+1)}} \right)  \right)$ in the $\ell_{2}$ setup.
\end{itemize}
\end{theorem}
The careful reader may notice that the choice of $p \in \mathbb{N}$ in our polynomial averaging scheme does not figure in the bound of Theorem~\ref{th:main-complexity}: any valid choice of $p$ leads to the same bound. However, in practice the choice of $p$ turns out to be very important as we shall see later. 
%The asymptotic theoretical convergence rate does not depend on the averaging scheme $p$; but $p$ turns out to matter in practice.
%Note that $O(1/T^{q/(q+1)})$ is the term with the slowest theoretical convergence rate.
Secondly, the reader may notice an interesting dependence on $q$: the term $O(1/T^{q/(q+1)})$ gets better as $q$ increases; while larger $q$ worsens the exponential rate with base $\lambda$ in the term $O(\lambda^{T^{1/(q+1)}} /T^{1/(q+1)})$. For any fixed $q$, the dominant term is $O(1/T^{q/(q+1)})$.

\noindent \textbf{Complexity for ellipsoidal uncertainty sets.}
We will now combine Proposition~\ref{prop:prop-fast-y-update-ell2-ell1}, which gives the cost per iteration of FOM-VI, with Theorem~\ref{th:main-complexity}, to get a total complexity of FOM-VI when considering both the number of iterations and cost per iteration. 

First, let us consider $q=2$, which is the setup we will focus on in our experiments. The complexity of the $\ell_{1}$ setup is $O \left( A^{4}S^{2}  \left( \dfrac{\log(S)}{\log(A)}\right)^{0.75} \log^2(\epsilon^{-1}) \epsilon^{-1.5} \right)$ and for the $\ell_{2}$ setup it is $O \left(A^{2.5}S^{3.5} \log(S) \log(\epsilon^{-1}) \epsilon^{-1.5} \right).$ These results are better than the complexity of VI in terms of the  number of states and actions. This comes at the cost of the dependence on the desired accuracy $\epsilon$, which is worse  than for \ref{alg:VI}. This is of course expected when applying a first-order method rather than IPMs.
However, in practice we expect that our algorithms will be preferable when solving problems with large $A$ and $S$, as is often the case with first-order methods. Indeed, we find numerically  that this occurs for $S,A \geq 50$ on ellipsoidal uncertainty sets (see Section \ref{sec:simu}). %In particular, IPM method may spend a large amount of time computing the inverse of matrices; meanwhile our algorithm is making progress by computing cheaper proximal updates.

Next, let us consider what happens as $q$ gets large. In that case, the complexity of the $\ell_{1}$ setup approaches $O \left( A^{3}S^{2} \left( \log(S) / \log(A)\right)^{0.5} \log^{2}(\epsilon^{-1})\epsilon^{-1} \right)$, while the complexity of the $\ell_{2}$ setup approaches $O \left( A^{2}S^{3} \log(S) \log(\epsilon^{-1}) \epsilon^{-1} \right).$ This last complexity result is $O(A^{1.5}S^{1.5})$ better than the VI complexity \eqref{eq:complexity-Kuhn} in terms of instance size.

Next let us compare the $\ell_1$ and $\ell_2$ setups. When $S=A$, the $\ell_{2}$ and $\ell_{1}$ setup have better dependence on number of states and actions than \ref{alg:VI} (by 2 order of magnitudes). If the number of actions $A$  is considered a constant, then the $\ell_{1}$ has better convergence guarantees than the $\ell_{2}$ setup. However, each proximal update in the $\ell_{1}$ setup requires two interwoven binary searches over Lagrange multipliers, which can prove time-consuming in practice, as we show in our numerical experiments.

\noindent \textbf{Complexity for KL uncertainty sets.}
Similarly to ellipsoidal uncertainty sets, we can analyze our performance on KL uncertainty sets. 
Again we combine Proposition \ref{prop:prop-fast-y-update-KL} with Theorem~\ref{th:main-complexity}. For $q=2$, the $\ell_{1}$ setup has complexity $O\left( A^{4}S^{2}  \left( \dfrac{\log(S)}{\log(A)}\right)^{0.75}\log(\epsilon^{-1})\epsilon^{-1.5} \right)$ for returning an $\epsilon$-optimal solution, while the $\ell_{2}$ setup has complexity $O\left( A^{2.5}S^{3.5}\log(\epsilon^{-1})\epsilon^{-1.5} \right)$. For large $q$, the complexity approaches $O\left( A^{3}S^{2} \left( \dfrac{\log(S)}{\log(A)}\right)^{0.5}\log(\epsilon^{-1})\epsilon^{-1} \right)$ for the $\ell_{1}$ setup and $O\left( A^{2}S^{3} \log(\epsilon^{-1})\epsilon^{-1} \right)$ for the $\ell_{2}$ setup.  To the best of our knowledge, this is the first algorithmic result for $s$-rectangular KL uncertainty sets.

Finally, note that in terms of storage complexity, all our setups
 only need to store the current value vector $\bm{v}^{\ell} \in \R^{S}$ and the running weighted average $(\bar{\bm{x}}^{\ell},\bar{\bm{y}}^{\ell})$ of the iterates. In total, we need to store $O\left(S^{2}A \right)$ coefficients, which is the same as the number of decision variables of a solution.
%  We can store  in order to reduce the spatial complexity of Algorithm \ref{alg:PD-RMDP}. At VI iteration $\ell$ we only need to store 
% \begin{remark}
% In \emph{Distributionally Robust MDP} with Wasserstein distance (based on an $\ell_{2}$ metric) from a nominal density with finite support, the Bellman update (\cite{yang2017convex}, Equation (9)) is very similar to the Bellman update in our setting \eqref{eq:T_max-def}. Our FOM-based algorithmic framework can be extended to solve this Distributionally Robust MDP with complexity similar  to our ellipsoidal setting, see Appendix \ref{app:WDRMDP} for more details.
% \end{remark}

%%% Local Variables:
%%% mode: latex
%%% TeX-master: "../main_fom_rmdp_aaai21"
%%% End:

\section{Numerical experiments}\label{sec:simu}
In this section we study the performance of our approach numerically. We focus here on ellipsoidal uncertainty sets, where we can compare our methods to Value Iteration.  We present results for KL uncertainty sets in Appendix~\ref{app:simu-KL}.

  \begin{figure*}[htp]
  \begin{subfigure}{0.24\textwidth}
\centering
  \includegraphics[width=1.0\linewidth]{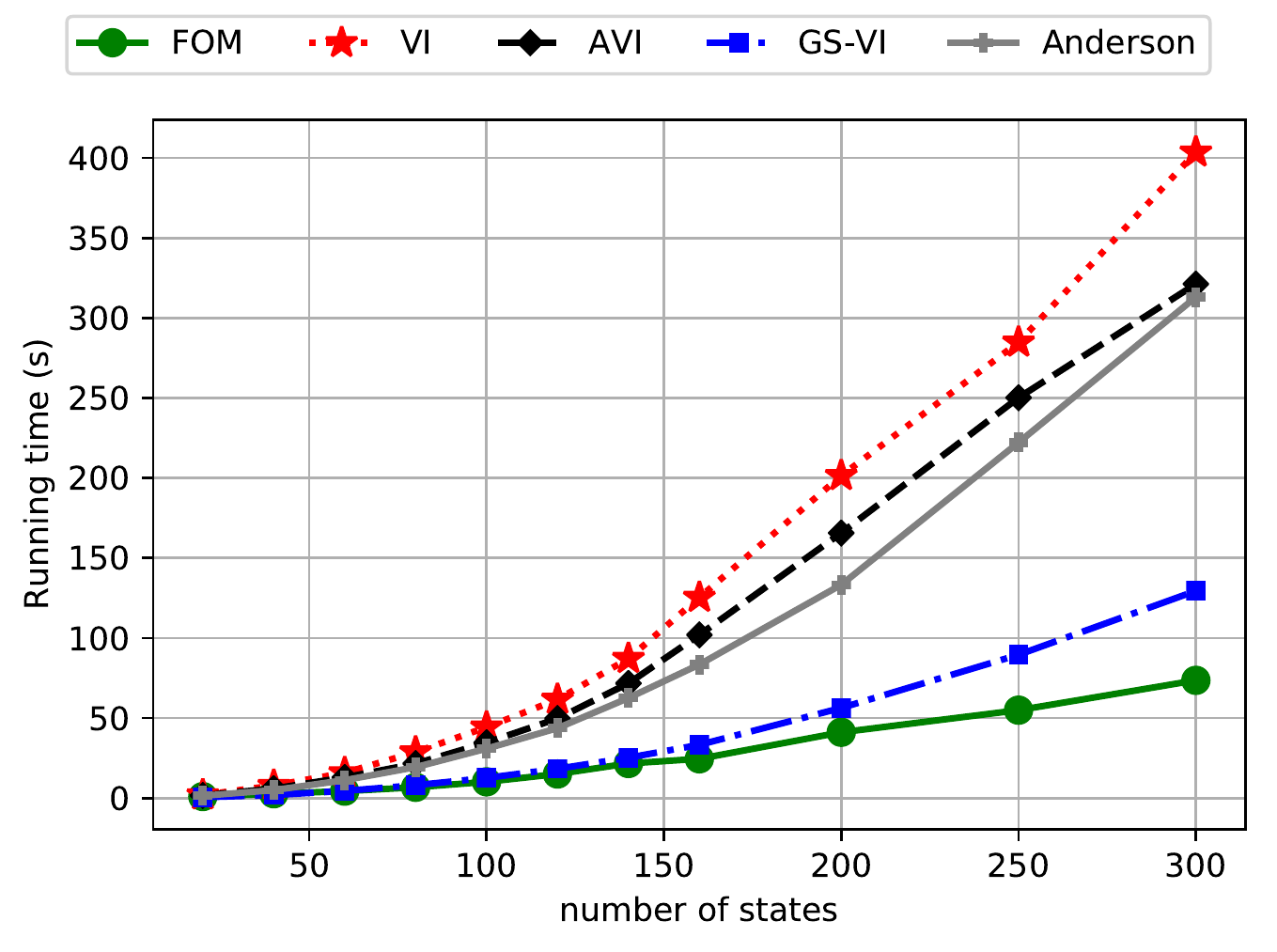}
  \captionof{figure}{Healthcare instance.}
  \label{fig:comparison_real_instance}
  \end{subfigure}
  \begin{subfigure}{0.24\textwidth}
\centering
  \includegraphics[width=1.0\linewidth]{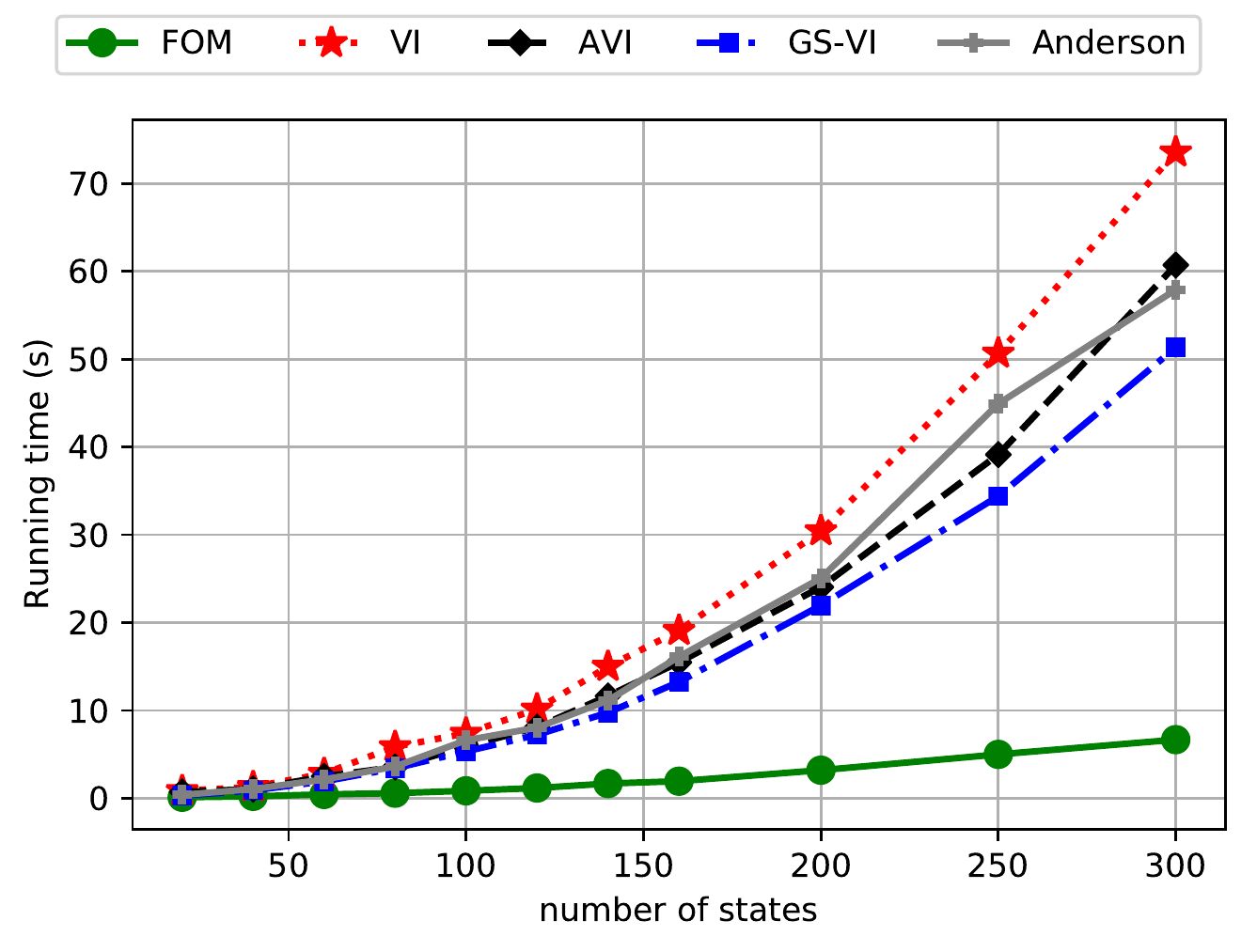}
  \captionof{figure}{Machine instance.}
  \label{fig:comparison_machine}
  \end{subfigure}
\begin{subfigure}{0.24\textwidth}
\centering
  \includegraphics[width=1.0\linewidth]{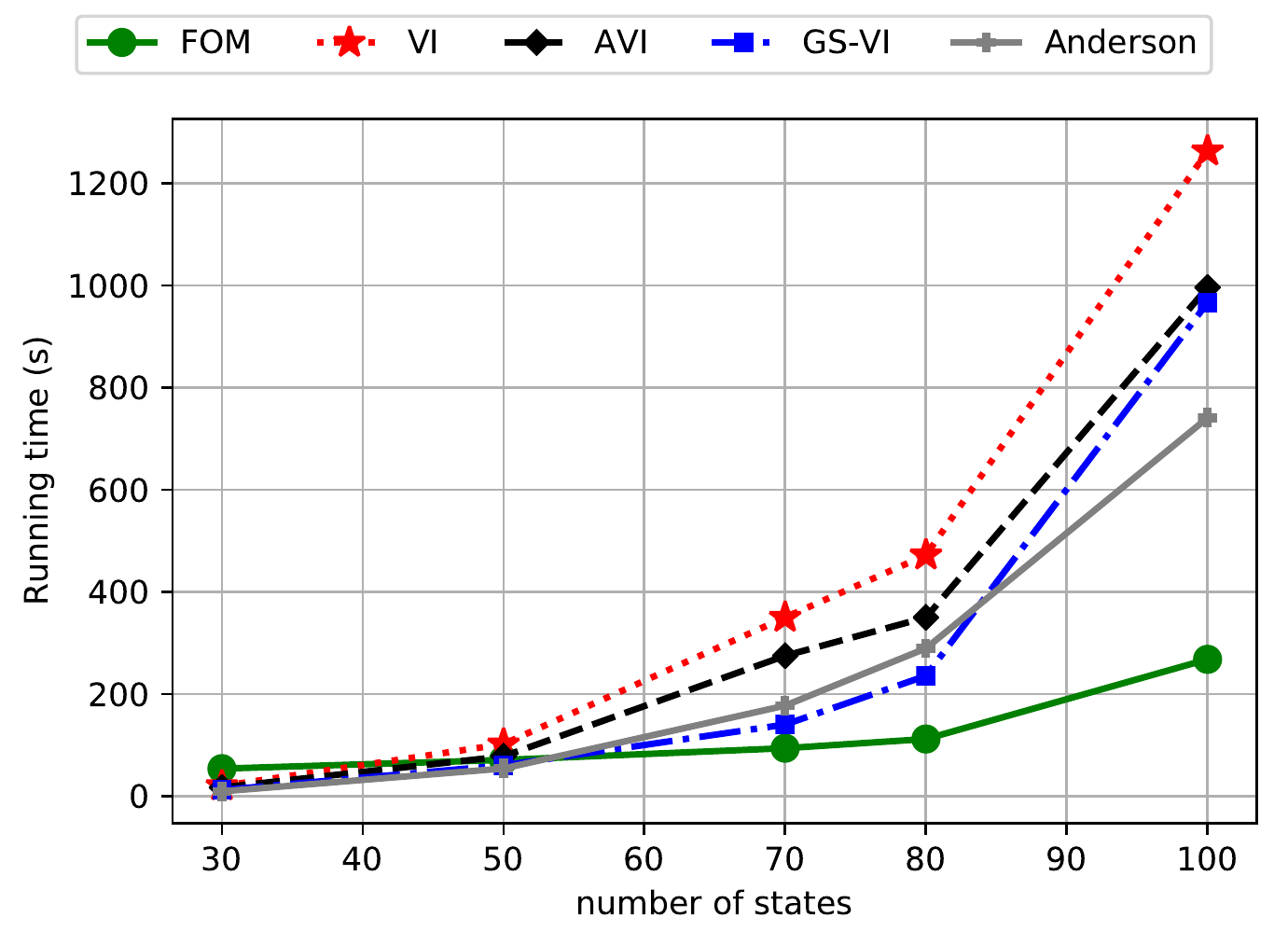}
  \captionof{figure}{ Garnet, high connectivity.}
  \label{fig:comparison_alg_all_small}
  \end{subfigure}
  \begin{subfigure}{0.24\textwidth}
\centering
  \includegraphics[width=1.0\linewidth]{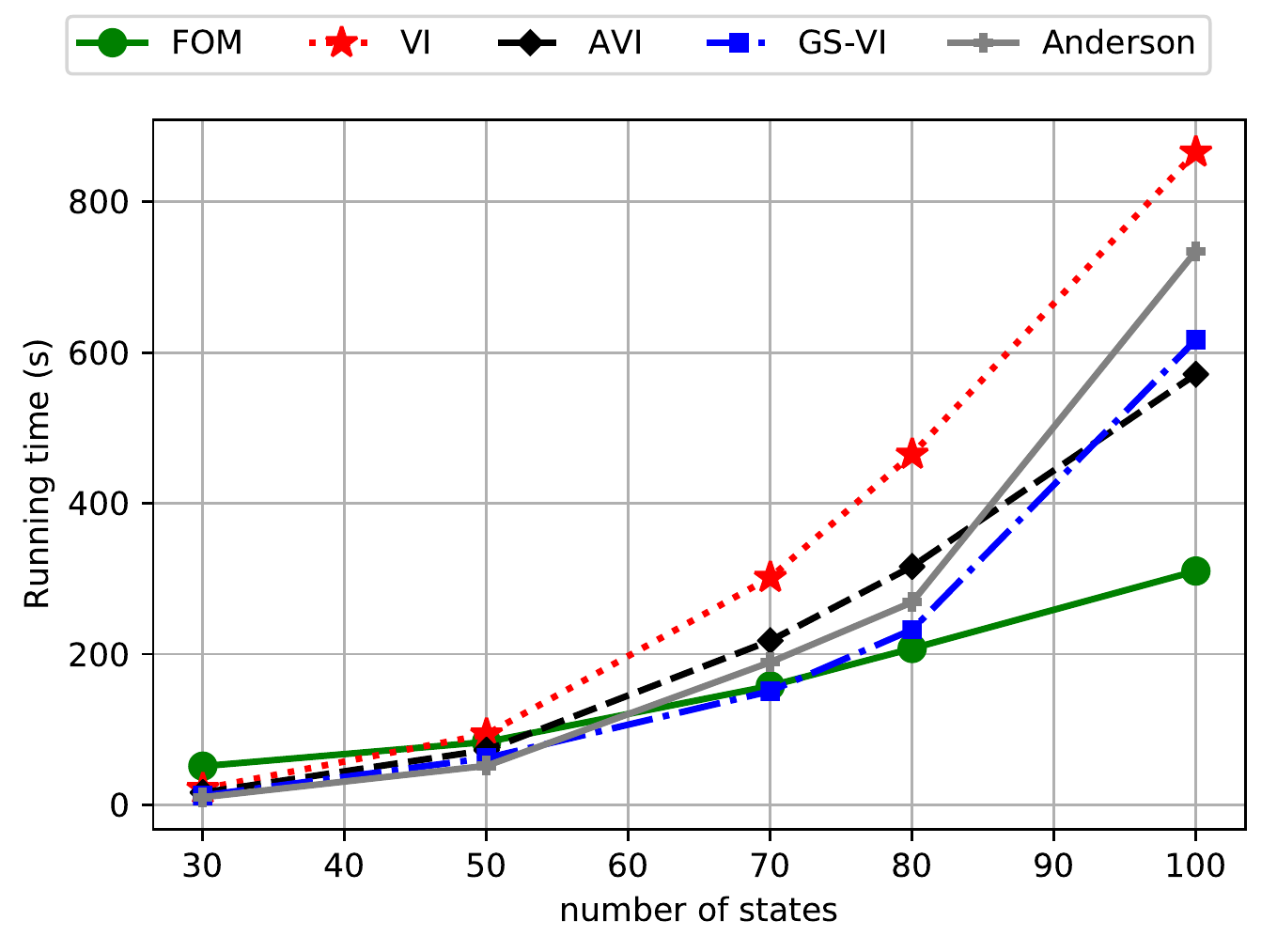}
  \captionof{figure}{ Garnet, low connectivity.}
  \label{fig:comparison_alg_all_small-low-density}
  \end{subfigure}
  \caption{Comparison of FOM-VI with variants of Value Iteration on four MDP domains.}
\end{figure*}

\noindent
\textbf{Duality gap in the robust MDP problem.}
For a given policy-kernel pair $(\bm{x},\bm{y})$, we measure the performance as the duality gap 
%\begin{equation}\label{eq:DG-RMDP}\tag{DG}
 $\text{(DG)}=\max_{\bm{y'} \in \PP} R(\bm{x},\bm{y'}) - \min_{\bm{x'} \in \Pi} R(\bm{x'},\bm{y}).$
%\end{equation}
%However, computing $\max_{\bm{y'} \in \PP} R(\bm{x},\bm{y'})$ is itself a heavy computational task which becomes intractable for more than about 100 states and actions (see details in Appendix \ref{app:upper-bounds}). We show this measure when feasible, i.e. for small to medium size instances.
% (Figures \ref{fig:comparison_norms_pq_11}-\ref{fig:DG_epoch},  \ref{fig:DG_weight} and \ref{fig:comparison_alg_all_small}).
%\noindent\textbf{Duality gap in \eqref{eq:v-star-fixed-point}.} Our second performance measure is
%\begin{equation}\label{eq:upper-bound-1}\tag{UB-1}
% \dfrac{2 \lambda}{1-\lambda} \| \bm{v} - F(\bm{v}) \|_{\infty} + \max_{s \in \X} \{ \max_{\bm{y'} \in \PP} F^{\bm{x},\bm{y'}}(\bm{v})_{s} - \min_{\bm{x'} \in \Pi} F^{\bm{x'}, \bm{y}}(\bm{v})_{s} \}.
%\end{equation}
%% $\text{(UB-1)}=
%% 2 \lambda(1-\lambda)^{-1} \| \bm{v} - F(\bm{v}) \|_{\infty} + \max_{s \in \X} \{ \max_{\bm{y'} \in \PP} F^{\bm{x},\bm{y'}}(\bm{v})_{s} - \min_{\bm{x'} \in \Pi} F^{\bm{x'}, \bm{y}}(\bm{v})_{s} \}$.
%\eqref{eq:upper-bound-1} is an upper bound on the duality gap in  \eqref{eq:v-star-fixed-point}.
%%$\epsilon^{*}(\bm{x},\bm{y}) = \max_{s \in \X} \{ \max_{\bm{y'} \in \PP} F^{\bm{x},\bm{y'}}(\bm{v}^{*})_{s} - \min_{\bm{x'} \in \Pi} F^{\bm{x'}, \bm{y}}(\bm{v}^{*})_{s} \}.$
%We are free to pick any $v$ in \eqref{eq:upper-bound-1}; in our experiments we use  $\bm{v}^{\ell}$ at each epoch $\ell$.
%\eqref{eq:upper-bound-1} is significantly faster to compute than (DG) (see Appendix \ref{app:upper-bounds}) but is only an upper bound. %on $\epsilon^{*}(\bm{x},\bm{y})$.
Note that (DG)$\leq \epsilon$ implies that $\bm{x}$ is $2 \epsilon$-optimal in the robust MDP problem.

\noindent
\textbf{Best empirical setup of Algorithm \ref{alg:PD-RMDP}.}
For the sake of conciseness, our extensive comparisons of the various proximal setups and parameter choices ($p, q \in \N$) are presented in Appendix \ref{app:comparison-our-algorithm}. Here we focus on the conclusions. The proximal setup with the best empirical performance is the $\ell_{2}$ setup where $(\| \cdot \|_{X}, \| \cdot \|_{Y}) = (\ell_{2},\ell_{2}),$ even though its theoretical guarantees may be worse than the $\ell_{1}$ setup (for large state space); this is similar to the matrix-game setting \citep{GKG20}.
 For averaging the PD iterates, an increasing  weight scheme, i.e. $p \geq 1$ in $\omega_{t} = t^{p}$, is clearly stronger (this is again similar to the matrix-game setting). We also recommend setting $q=2$ (or even larger), as this leads to better empirical performance for the true duality gap (DG) in the settings where we could compute that duality gap. 

\subsection{Comparison with Value Iteration}
%We have seen in the previous section (Figure \ref{fig:p_one}-\ref{fig:q_two}) that \eqref{eq:upper-bound-1} is a good upper bound on the true duality gap of the current iterates in \eqref{eq:Robust-MDP}. 
We present our comparisons with the \ref{alg:VI} algorithm in Figures \ref{fig:comparison_real_instance}-\ref{fig:comparison_alg_all_small-low-density}. We also compare FOM-VI with Gauss-Seidel VI (\textit{GS-VI}, \cite{Puterman}), Anderson VI (\textit{Anderson}, \cite{ref-c}), and \textit{Accelerated VI} (\textit{AVI}, \citet{GGC-AVI}), see Appendix \ref{app:simu}.
The y-axis shows the number of seconds it takes each algorithm to compute an $\epsilon$-optimal policy, for $\epsilon=0.1$.
 Following our analysis of the various setups for our algorithm, these plots focus on the $\ell_{2}$ setup with $(p,q)=(2,2)$.

\noindent\textbf{Empirical setup.} 
All the simulations are implemented in Python 3.7.3, and were performed on a laptop with 2.2 GHz Intel Core i7 and 8 GB of RAM. We use Gurobi 8.1.1 to solve any linear or quadratic optimization problems involved.  In order to obtain an $\epsilon$-solution of the robust MDP problem with the value iteration algorithms, we use the stopping condition $\| \bm{v}_{s+1} - \bm{v}_{s} \|_{\infty} \leq \epsilon \cdot (1-\lambda) \cdot (2 \lambda)^{-1}$ (Chapter 6.3 in \citet{Puterman}). We stop Algorithm \ref{alg:PD-RMDP} as soon as (DG) $\leq \epsilon /2 $. We initialize the algorithms with $\bm{v}_{0}=\bm{0}$.
At epoch $\ell$ of VI AVI and Anderson, we warm-start each computation of $F(\bm{v}^{\ell})$ with the optimal solution obtained from the previous epoch $\ell-1$. 

We consider two type of instances for our simulation. The first type of instances is inspired from real-life application and consists of a healthcare management instance and a machine replacement instance. The second type is based on random \textit{Garnet} MDPs, a class of random MDP instances widely used for benchmarking algorithms.

\noindent\textbf{Results for healthcare instances.} We consider an MDP instance inspired from a healthcare application. We model the evolution of a patient's health using a Markov chain, using a simplification of the models used in  \citet{Goh, grand2020robust}. Note that such a model is prone to errors as (i) the Markovian assumption is only an approximation of the true dynamics of the patient's health, (ii) the presence of unobservable confounders may introduce biases in our observed transitions. Therefore, it is important to account for model mispecification in this setting. More specifically, we consider an MDP where there are $S-1$ health states, one `\textit{mortality}' state and three actions (drug level), corresponding to high, medium and low drug levels. The state $1$ corresponds to a healthy condition while the state $S-1$ is more likely to lead to mortality. The goal of the decision maker is to prescribe a given drug dosage (low/high/medium) at every state, in order to keep the patient alive (avoiding the mortality state), while minimizing the invasiveness of the treatment.  We observe $N=60$ samples around the nominal kernel transitions, presented in Figures \ref{fig:MDP-1}-\ref{fig:MDP-3} in the appendices, and we construct ellipsoidal uncertainty sets with radius $\alpha=\sqrt{SA}$. Figure \ref{fig:comparison_real_instance} shows the results, where our algorithm outperforms \ref{alg:VI} by about one order of magnitude on this structured and simple MDP instance, even though GS-VI performs well better than \ref{alg:VI} too. Additionally, our algorithm scales much better with instance size.

 \noindent\textbf{Results for Machine Replacement Problems}
 We also consider a machine replacement problem studied by \citet{Delage} and \citet{Kuhn}. The problem is to design a replacement policy for a line of machines. The states of the MDP represent age phases of the machine and the actions represent different repair or replacement options. Even though the transition parameters can be estimated from historical data, one often does not have access to enough historical data to exactly assess the probability of a machine breaking down when in a given condition. Additionally, the historical data may contain errors; this warrants the use of a robust model for finding a good replacement policy.
 In particular, the machine replacement problem involves a machine whose set of possible conditions are described by $S$ states. There are two actions: \textit{repair} and \textit{no repair}.
The first $S-2$ states are operative states.   The states $1$ to $S-2$ model the condition of the machine, with $1$ being perfect condition and $S-2$ being worst condition. There is a cost of 0 for states $1, ..., S-3$; letting the machine reach the worst operative state $S-2$ is penalized with a cost of $20$. The last two states $S-1$ and $S$ are states representing when the machine is being repaired.  The state $S-1$ is a standard repair state and has a cost of 2, while the last state $S$ is a longer and more costly repair state and has cost 10. The initial distribution is uniform across states. Figures describing the MDP can be found in Appendix \ref{app:machine-instance}. On this instance, FOM-VI clearly outperforms every variants of \ref{alg:VI}, as seen on Figure \ref{fig:comparison_machine}.
%
%  \begin{figure}[H]
%\centering
%  \includegraphics[width=1.0\linewidth]{figure_2/Algorithm_Comparisons_healthcare}
%  \captionof{figure}{Comparison on our healthcare instance.}
%  \label{fig:comparison_real_instance}
%  \end{figure}

\noindent\textbf{Random Garnet MDP instances.}
 We generate Garnet MDPs (Generalized Average Reward Non-stationary Environment Test-bench, \citet{garnet,bhatnagar2007naturalgradient}), which are an abstract class of MDPs parametrized by a branching factor $n_{branch}$, equal to the proportion of reachable next states from each state-action pair $(s,a)$. Garnet MDPs are a popular class of finite MDPs used for benchmarking algorithms for MDPs \citep{garnet-1,garnet-2,garnet-3}.
The parameter $n_{branch}$ controls the level of connectivity of the underlying Markov chains. We test our algorithm for high connectivity ($n_{branch}=50 \%$, Figure \ref{fig:comparison_alg_all_small}) and low connectivity ($n_{branch}=20 \%$, Figure \ref{fig:comparison_alg_all_small-low-density}) in our simulations. We draw the cost parameters at random uniformly in $[0,10]$ and we fix a discount factor $\lambda=0.8$. The radius $\alpha$ of the $\ell_{2}$ ball from the uncertainty set \eqref{eq:uncertainty-norm-2} is set to $\alpha = \sqrt{n_{branch} \times A }.$

In Figures \ref{fig:comparison_alg_all_small}-\ref{fig:comparison_alg_all_small-low-density}, we note that for smaller instances, the performance of FOM-VI is similar to both VI, AVI, GS-VI and Anderson. This is expected: our algorithm has worse convergence guarantees in terms of the dependence in $\epsilon$, but better guarantees in terms of the number of state-actions $S,A$. When the number of states and actions grows larger, FOM-VI performs significantly better than the three other methods.

%\begin{figure}[H]
%\centering
%  \includegraphics[width=1.0\linewidth]{figure_2/Algorithm_Comparisons_Garnet_50}
%  \captionof{figure}{Comparison on random Garnet MDP instances  with high connectivity (50 \% of reachable next states from any state-action pair).}
%  \label{fig:comparison_alg_all_small}
%  \end{figure}
%  
%  \begin{figure}[H]
%\centering
%  \includegraphics[width=1.0\linewidth]{figure_2/Algorithm_Comparisons_Garnet_20}
%  \captionof{figure}{Comparison on random Garnet MDP instances with low connectivity (20 \% of reachable next states from any state-action pair).}
%  \label{fig:comparison_alg_all_small-low-density}
%  \end{figure}

%%% Local Variables:
%%% mode: latex
%%% TeX-master: "../main"
%%% End:

\bibliographystyle{plainnat} 
\bibliography{FOM_RMDP}
\appendix

\newpage

\section{Proofs of Section \ref{sec:RMDP}}\label{app:proof-sec-lemmas}
\subsection{Some useful lemmas}
The next lemmas give bounds on some sums that appear in the proof of Proposition \ref{prop:rate-PD-general}. 
\begin{lemma}\label{lem:lambda-sum-1}
Let $\lambda \in (0,1)$ and $k, n \in \N$. Then
\[ \sum_{\ell=1}^{k} \lambda^{\ell} \ell^{n} \leq O \left( \dfrac{k^{n}\lambda^{k}}{(1-\lambda)^{n+1}} \right).\]
\end{lemma}

\begin{proof}[Proof of Lemma \ref{lem:lambda-sum-1}]
Let us define $f: x \mapsto \dfrac{1-x^{k}}{1-x} = \sum_{\ell=1}^{k} x^{\ell}.$
Then $f^{(n)}(x) =  \sum_{\ell=1}^{k} x^{\ell} \ell (\ell-1) ... (\ell-n+1)$, and 
\[ \sum_{\ell=1}^{k} \lambda^{\ell} \ell^{n} = O \left(\sum_{\ell=1}^{k} x^{\ell} \ell (\ell-1) ... (\ell-n+1) \right). \]
We can conclude by computing the $n$-th derivative of $f$ as the $n$-th derivative of $x \mapsto \dfrac{1-x^{k}}{1-x} $.
\end{proof}

\begin{lemma}\label{lem:lambda-sum-4}
Let $\lambda \in (0,1)$ and $q \geq 0$. Then there exists a constant $M_{\lambda,q}$ such that
\[ \sum_{t=1}^{\ell} \dfrac{1}{t^{q}\lambda^{t}} \leq M_{\lambda,q} \dfrac{1}{\ell^{q} \lambda^{\ell}}.\]
%with $M_{\lambda,q} \rightarrow + \infty$ for $\lambda \rightarrow 1$ and $M_{\lambda,q} \rightarrow + 1$ for $q \rightarrow 0.$
\end{lemma}
\begin{proof}[Proof of Lemma \ref{lem:lambda-sum-4}.]
Let $\lambda^{+} = \dfrac{1+\lambda}{2}$. Note that we always have $\lambda< \lambda^{+} < 1.$
For $x = 1/\lambda$, we have $x > x^{+} = 1/\lambda^{+} > 1$.
For $f_{\ell}(x) = \sum_{t=1}^{\ell} \dfrac{x^{t}}{t^{q}},$ we have
\begin{align*}
f'_{\ell}(x) & = \sum_{t=1}^{\ell} \dfrac{x^{t-1}}{t^{q-1}} 
 \leq \ell^{1-q} \sum_{t=1}^{\ell} x^{t-1} 
 \leq \ell^{1-q} \dfrac{x^{\ell}-1}{x-1}.
\end{align*}
This proves that
\begin{align*}
f_{\ell}(x) - f_{\ell}(x^{+}) & = \ell^{1-q} \int_{u=x^{+}}^{x} \dfrac{u^{\ell}-1}{u-1} du  \leq \ell^{1-q} \dfrac{1}{x^{+}-1} \int_{u=x^{+}}^{x} (u^{\ell}-1) du   \leq \ell^{1-q} \dfrac{1}{x^{+}-1} \dfrac{1}{\ell+1} [u^{\ell+1}-u]_{x^{+}}^{x},
\end{align*}
and finally that
\begin{equation}\label{eq:f_ell_x}
f_{\ell}(x) = f_{\ell}(x^{+}) + \ell^{1-q} \dfrac{1}{x^{+}-1} \dfrac{1}{\ell+1} (x^{\ell+1}-x - x^{+ \; \ell+1}+x^{+}).
\end{equation}
We will prove that the right-hand side of \eqref{eq:f_ell_x} is itself a $O \left( \dfrac{x^{\ell}}{\ell^{q}} \right)$ as $\ell \rightarrow + \infty$. The proof relies on the fact that $x > x^{+}$, and therefore that $ \left( \dfrac{x^{+}}{x} \right)^{\ell} \ell^{m} = o(1)$, for any $m \geq 0$.

First, since $x^{+} < x$, we note that
\begin{align*}
\dfrac{\ell^{q}}{x^{\ell}} \left(\ell^{1-q} \dfrac{1}{x^{+}-1} \dfrac{1}{\ell+1} \left( x^{\ell+1}-x - x^{+ \; \ell+1}+x^{+} \right) \right) = \dfrac{1}{x^{+}-1} \dfrac{\ell}{\ell+1} ( x- \dfrac{x}{x^{\ell}} - x^{+}\left(\dfrac{x^{+}}{x}\right)^{\ell} +\dfrac{x^{+}}{x^{\ell}}  = O \left( 1 \right).
\end{align*}

Now for $\dfrac{\ell^{q}}{x^{\ell}} f_{\ell}(x^{+}) $ we need to distinguish between the potential values of $q \in \R_{+}$.

\textbf{Proof for} $q=0$.
\[ \dfrac{\ell^{q}}{x^{\ell}} f_{\ell}(x^{+}) = \dfrac{1}{x^{\ell}}  \sum_{t=1}^{\ell} x^{+ \; t} = O \left( \dfrac{x^{+ \; \ell}}{x^{\ell}} \right) = o(1).\]

\textbf{Proof for} $q \in (0,1)$.
\begin{align*}
\dfrac{\ell^{q}}{x^{\ell}} f_{\ell}(x^{+}) = \dfrac{\ell^{q}}{x^{\ell}}  \sum_{t=1}^{\ell} \dfrac{ x^{+ \; t}}{t^{q}} & \leq  \dfrac{\ell^{q}}{x^{\ell}}  \sum_{t=1}^{\ell}  x^{+ \; t} \leq  O \left( \dfrac{\ell x^{+ \; \ell}}{x^{\ell}} \right) = o(1).
\end{align*}

\textbf{Proof for} $q =1$.
\[ \dfrac{\ell^{q}}{x^{\ell}} f_{\ell}(x^{+}) = \dfrac{\ell}{x^{\ell}}  \sum_{t=1}^{\ell} \dfrac{ x^{+ \; t}}{t} \leq  \dfrac{\ell}{x^{\ell}} x^{+ \; \ell} \log(\ell) =  o(1).\]

\textbf{Proof for} $q \geq 1$.
\begin{align*}
\dfrac{\ell^{q}}{x^{\ell}} f_{\ell}(x^{+}) & = \dfrac{\ell^{q}}{x^{\ell}}  \sum_{t=1}^{\ell} \dfrac{ x^{+ \; t}}{t^{q}} \leq  \dfrac{\ell^{q}}{x^{\ell}}  x^{+ \; \ell} \sum_{t=1}^{\ell}  \dfrac{1}{t^{q}}  \leq  O \left( \dfrac{\ell^{q} x^{+ \; \ell}}{x^{\ell}} \right) = o(1).
\end{align*}

\end{proof}

\subsection{Approximate Value Iteration}
We present a variant of Value Iteration where each sub-problem $F(\bm{v})_{s}$ is solved approximately.
\begin{proposition}\label{prop:error_at_t}
Suppose that for every VI epoch $\ell \geq 1$, we solve the min-max problem \eqref{eq:T_max-def} up to precision $\epsilon_{\ell}>0$, i.e. we compute $(\bm{x}^{\ell},\bm{y}^{\ell})$ such that
\begin{align*}
\bm{v}^{\ell+1}   = F^{\bm{x}^{\ell},\bm{y}^{\ell}}(\bm{v}^{\ell}),
\| \bm{v}^{\ell+1} - F(\bm{v}^{\ell}) \|_{\infty}  \leq \epsilon_{\ell}.
\end{align*}
Then we have, for any $\ell \geq 1$,
\[ \| \bm{v}^{\ell+1} - \bm{v}^{*} \|_{\infty} \leq \lambda \| \bm{v}^{\ell} - \bm{v}^{*} \|_{\infty} +\epsilon_{\ell} ,\]
\[ \| \bm{v}^{\ell+1} - \bm{v}^{\ell} \|_{\infty} \leq \lambda \| \bm{v}^{\ell} - \bm{v}^{\ell-1} \|_{\infty} + \epsilon_{\ell} + \epsilon_{\ell-1}.\]
In particular, this implies 
\[ \| \bm{v}^{\ell} - \bm{v}^{*} \|_{\infty} \leq \lambda^{\ell} \left( \| \bm{v}^{*} - \bm{v}^{0} \|_{\infty} + \sum_{t=0}^{\ell-1} \dfrac{\epsilon_{t}}{\lambda^{t}} \right),\]
\[ \| \bm{v}^{\ell+1} - \bm{v}^{\ell} \|_{\infty} \leq \lambda^{\ell} \left( \| \bm{v}^{1} - \bm{v}^{0} \|_{\infty} + \sum_{t=0}^{\ell} \dfrac{\epsilon_{t} + \epsilon_{t-1}}{\lambda^{t}} \right).\]
\end{proposition}
\begin{proof}[Proof of  Proposition \ref{prop:error_at_t}]
We have
\begin{align*}
\| \bm{v}^{*} - \bm{v}^{t+1} \|_{\infty} & = \|F( \bm{v}^{*}) - \bm{v}^{t+1}  \|_{\infty} \\
 & = \|F( \bm{v}^{*}) -F(\bm{v}^{t}) + F(\bm{v}^{t}) - \bm{v}^{t+1}  \|_{\infty} \\
 & \leq \|F( \bm{v}^{*}) -F(\bm{v}^{t}) \|_{\infty} + \| F(\bm{v}^{t}) - \bm{v}^{t+1}  \|_{\infty} \\ 
 & \leq \lambda \| \bm{v}^{*} - \bm{v}^{t} \|_{\infty} + \| F(\bm{v}^{t}) - \bm{v}^{t+1} \|_{\infty} \\
 & \leq \lambda \| \bm{v}^{*} - \bm{v}^{t} \|_{\infty} + \epsilon_{t}.
\end{align*}
Similarly,
\begin{align*}
\| \bm{v}^{\ell+1} - \bm{v}^{\ell} \|_{\infty} & \leq \| \bm{v}^{\ell+1} - F(\bm{v}^{\ell}) + F(\bm{v}^{\ell}) - \bm{v}^{\ell} \|_{\infty} \\
& \leq \| \bm{v}^{\ell+1} - F(\bm{v}^{\ell})  \|_{\infty} +  \| F(\bm{v}^{\ell}) - \bm{v}^{\ell}\|_{\infty}  \\
& \leq \epsilon_{\ell} + \| F(\bm{v}^{\ell}) - \bm{v}^{\ell} \|_{\infty} \\
& \leq \epsilon_{\ell} + \| F(\bm{v}^{\ell}) - F(\bm{v}^{\ell-1})\|_{\infty}  + \| F(\bm{v}^{\ell-1})-  \bm{v}^{\ell} \|_{\infty} \\
& \leq \epsilon_{\ell} + \lambda \| \bm{v}^{\ell} - \bm{v}^{\ell-1}\|_{\infty} + \| F(\bm{v}^{\ell-1})-  \bm{v}^{\ell} \|_{\infty} \\
& \leq \epsilon_{\ell} + \lambda \| \bm{v}^{\ell} - \bm{v}^{\ell-1} \|_{\infty} + \epsilon_{\ell-1.}
\end{align*}
The rest of the lemma follows directly from iterating the recursions on $\| \bm{v}^{*} - \bm{v}^{t+1} \|_{\infty} $ and on $\| \bm{v}^{\ell+1} - \bm{v}^{\ell} \|_{\infty}$.
\end{proof}
Note that this analysis is classical and is close to the case of approximate policy iteration for non-robust MDPs \citep{gabillon2013approximate, scherrer2015approximate}. While we treat the term $\epsilon_{\ell}$ as a (chosen) error term in our algorithm, we would like to note that we can think of the term $\epsilon_{\ell}$ as some random noise, coming from either function approximations or sample-based estimations (see Section 4 in \citet{scherrer2015approximate}). 
%Thus this framework can a priori encapsulates \textit{model-free} approaches with 
%\begin{proof}[Proof of Lemma \ref{lem:warm-start-VI}]
%We have
%\begin{align*}
%\| F(\bm{v}^{\ell+1}) - F^{\bm{x}^{\ell},\bm{y}^{t}}(\bm{v}^{\ell+1}) \|_{ \infty} & = \| F(\bm{v}^{\ell+1}) -F(\bm{v}^{\ell}) + F(\bm{v}^{\ell}) - \bm{v}^{\ell+1} + \bm{v}^{\ell+1} - F^{\bm{x}^{\ell},\bm{y}^{\ell}}(\bm{v}^{\ell+1}) \|_{ \infty} \\
%& \leq \| F(\bm{v}^{\ell+1}) -F(\bm{v}^{\ell}) \|_{+ \infty} +\| F(\bm{v}^{\ell}) - \bm{v}^{\ell+1} \|_{+ \infty} +  \| \bm{v}^{\ell+1} - F^{\bm{x}^{\ell},\bm{y}^{\ell}}(\bm{v}^{\ell+1}) \|_{\infty} \\
%& \leq \lambda \| \bm{v}^{\ell+1} - \bm{v}^{\ell} \|_{\infty} + \epsilon_{\ell} +  \lambda \| \bm{v}^{\ell+1} - \bm{v}^{\ell} \|_{\infty} \\
%& \leq 2 \lambda \| \bm{v}^{\ell+1} - \bm{v}^{\ell} \|_{\infty} + \epsilon_{\ell},
%\end{align*}  
%and we can conclude by using Proposition \ref{prop:error_at_t}.
%\end{proof}

\section{Details on Primal-Dual Algorithm}\label{app:lem-norm-K}
In this section and the following one, we use the notation
$\LL^{\bm{K}}: X \times Y \rightarrow \R$ for the operator such that
\[ \LL^{\bm{K}}(\bm{x},\bm{y}) = \langle \bm{x}, \bm{c}_{s} \rangle + \lambda \langle \bm{Kx}, \bm{y} \rangle.\]
In particular, this implies that, if $\bm{v} \in \R^{S}$ is such that $\langle \bm{Kx}, \bm{y} \rangle = \sum_{a \in \A} x_{sa}\bm{y}_{sa}^{\top}\bm{v}$, we have
\begin{equation}\label{def:LL}
\LL^{\bm{K}}(\bm{x},\bm{y}) = \sum_{a \in \A} x_{sa} \left( c_{sa} + \lambda \bm{y}_{sa}^{\top} \bm{v}\right) = F^{\bm{x},\bm{y}}_{s}(\bm{v}).
\end{equation}
We present more details about the convergence rate of PDA. In particular, we have the following proposition \citep{ChambollePock16, GKG20}.
\begin{proposition}[\cite{ChambollePock16,GKG20}]\label{prop:descent-detail}
For $(\bm{x},\bm{y}), (\bm{x'}, \bm{y'}) \in X \times Y$, let $A(\bm{x},\bm{y},\bm{x'},\bm{y'})$ such that
\begin{align*}
A(\bm{x},\bm{y},\bm{x'},\bm{y'}) = & \dfrac{1}{\tau}D_{X}(\bm{x},\bm{x'}) + \dfrac{1}{\sigma}D_{Y}(\bm{y},\bm{y'}) - \langle \bm{K}(\bm{x}-\bm{x'}),\bm{y}-\bm{y'} \rangle.
\end{align*}
Let a scalar $\Omega \geq 0$ and some step sizes $\tau, \sigma$ such that for all $(\bm{x},\bm{y}), (\bm{x'}, \bm{y'}) \in X \times Y$,
\begin{equation}\label{eq:definition-A}
 0 \leq A(\bm{x},\bm{y},\bm{x'},\bm{y'})  \leq \Omega.
\end{equation}
Consider running PDA on the associated min-max problem for $T$ iterations. Consider weights $\omega_{1},...,\omega_{T}$ and $S_{T} = \sum_{t=1}^{T} \omega_{t}$.
Then we have the critical inequality: $\forall \; (\bm{x},\bm{y}) \in X \times Y$,
\begin{equation}\label{eq:critical_inequality}
\begin{aligned}
\LL^{\bm{K}}(\bm{x}^{t+1},\bm{y}) - \LL^{\bm{K}}(\bm{x},\bm{y}^{t+1}) \leq & A(\bm{x},\bm{y},\bm{x}^{t},\bm{y}^{t})  - A(\bm{x},\bm{y},\bm{x}^{t+1},\bm{y}^{t+1}).
\end{aligned}
\end{equation}

Additionally, summing up \eqref{eq:critical_inequality}, we have, for all $(\bm{x},\bm{y}) \in X \times Y$,
\begin{align*}
\sum_{t=1}^{T} \omega_{t} \left( \LL^{\bm{K}}(\bm{x}^{t},\bm{y}) - \LL^{\bm{K}}(\bm{x},\bm{y}^{t}) \right) \leq & \omega_{0} A[\bm{x},\bm{y},\bm{x}^{0},\bm{y}^{0}]+ \omega_{T}\Omega   - \omega_{1} \Omega  - \omega_{T} A[\bm{x},\bm{y},\bm{x}^{T},\bm{y}^{T}].
\end{align*}
In particular, for  $(\bm{\bar{x}}^{T},\bm{\bar{y}}^{T}) = (1/S_{T}) \sum_{t=1}^{T} \omega_{t} (\bm{x}_{t},\bm{y}_{t})$,  for all $(\bm{x},\bm{y}) \in X \times Y$,
\begin{align*}
\LL^{\bm{K}}(\bm{\bar{x}}^{T},\bm{y}) - \LL^{\bm{K}}(\bm{x},\bm{\bar{y}}^{T}) & \leq  \sum_{t=1}^{T} \dfrac{\omega_{t}}{S_{T}} \left( \LL^{\bm{K}}(\bm{x}^{t},\bm{y}) - \LL^{\bm{K}}(\bm{x},\bm{y}^{t}) \right)
\end{align*}
and therefore
\begin{align*}
\LL^{\bm{K}}(\bm{\bar{x}}^{T},\bm{y}) - \LL^{\bm{K}}(\bm{x},\bm{\bar{y}}^{T}) & \leq \Omega \dfrac{ \omega_{T}}{S_{T}}.
\end{align*}
\end{proposition}
We also prove the following lemma.
\begin{lemma}\label{lem:norm-k}Recall that 
\[ L_{\bm{K}}= \sup_{\| \bm{x} \|_{X} \leq 1, \| \bm{y} \|_{Y} \leq 1} \langle \bm{Kx}, \bm{y} \rangle.\]

For $(\| \cdot \|_{X},\| \cdot \|_{Y}) = (\| \cdot \|_{2} \|, \| \cdot \|_{2})$, $L_{ \bm{K}} = \lambda \| \bm{v}\|_{2}$.

For $(\| \cdot \|_{X},\| \cdot \|_{Y}) = (\| \cdot \|_{1} \| , \| \cdot \|_{1})$, $L_{ \bm{K} } = \lambda \| \bm{v} \|_{\infty}$.
\end{lemma}
\begin{proof}
Recall that $\bm{K}: \R^{A} \rightarrow \R^{A \times S}$ is defined as, 
$\forall \; (\bm{x},\bm{y}) \in \R^{A} \times \R^{A \times S}$,
\begin{align*}
\langle\bm{Kx},\bm{y}\rangle  & = \lambda \sum_{a=1}^{A} x_{a}\bm{y}_{a}^{\top}\bm{v}  = \lambda \sum_{a=1}^{A} \sum_{s'=1}^{S} x_{a}y_{as'}v_{s'}.
\end{align*}
In particular, $ K_{a's',a} = 1_{\{ a=a'\}} \lambda v_{s'}, \forall \; a,a' \in \A, s' \in \X.$
\begin{enumerate}
\item $\ell_{2}$ setup.
%\paragraph{$(\| \cdot \|_{X},\| \cdot \|_{Y}) = (\| \cdot \|_{2} \| \cdot \|_{2})$.}
By definition, $L_{\bm{K}}$ is the square root of maximum modulus of the eigenvalues of $\bm{K}^{\top}\bm{K} \in \R^{A \times A}.$ Let $\bm{x} \in \R^{A}.$
Then for $a' \in \A, s' \in \X$,
\[ \left( \bm{Kx} \right)_{a's'} = \lambda x_{a'}v_{s'}.\]
Therefore, by definition of matrix-vector product,
\begin{align*}
\left( \bm{K}^{\top}\bm{K} \bm{x} \right)_{a} & = \sum_{a'',s''} \left( \bm{K}^{\top} \right)_{a,s''a''} \left( \bm{Kx} \right)_{s''a''} = \sum_{a'',s''} \left( \bm{K}^{\top} \right)_{a,s''a''} \lambda x_{a''}v_{s''} =  \sum_{a'',s''} \left( \bm{K}\right)_{s''a'',a} \lambda x_{a''}v_{s''} \\
& = \sum_{a'',s''} 1_{\{ a = a'' \} } \lambda v_{s''} \lambda x_{a''}v_{s''}  = \lambda^{2} \left( \sum_{s'' \in \X} v_{s''}^{2} \right)  x_{a}  = \lambda^{2} \| \bm{v} \|_{2}^{2} x_{a}.
\end{align*}
This directly implies that $L_{ \bm{K} } = \lambda \| \bm{v} \|_{2}.$
\item $\ell_{1}$ setup.
%\paragraph{$(\| \cdot \|_{X},\| \cdot \|_{Y}) = (\| \cdot \|_{1} \| \cdot \|_{1})$.}
This is straightforward from the definition of $\bm{K}$, the definition $L_{\bm{K}}= \sup_{\| \bm{x} \|_{X} \leq 1, \| \bm{y} \|_{Y} \leq 1} \langle \bm{Kx}, \bm{y} \rangle,$ as well as the fact that $\bm{v} \geq \bm{0}$.
\end{enumerate}
\end{proof}
\section{Proof of Theorem \ref{thm:fom-vi rate}}\label{app:PD-ABCDE}
Our proof proceeds in several steps. We first show how to choose an upper bound $\Omega$ and step-sizes $\sigma,\tau$ uniformly across all epochs. We then show that the duality gap $\max_{\bm{y} \in \PP_{s}} F^{\bar{\bm{x}}^{T},\bm{y}}(\bm{v}^{*})_{s} - \min_{\bm{x} \in \Delta(A)} F^{\bm{x}, \bar{\bm{y}}^{T}}(\bm{v}^{*})_{s}$ can be bounded by the sum of 5 terms $e_{1}, ..., e_{5}$. We then give the dependency of $e_{1}, ..., e_{5}$ in terms of $T$ the number of PD iterations.

\paragraph{Upper bound $\Omega$ and step sizes $\sigma, \tau$.}
Our goal here is to define a scalar $\Omega$ common across all epochs of Algorithm \ref{alg:PD-RMDP}.
Note that for a given matrix $\bm{K}$ and some step sizes $\sigma, \tau$, the scalar $\Omega$ is defined as satisfying \eqref{eq:definition-A}.  From Remark 2 of \citet{ChambollePock16}, a possible choice for $\Omega$ is
\begin{equation}\label{eq:def-Omega}
\Omega= 2 \left( \Theta_{X}/\tau + \Theta_{Y}/\sigma \right)
\end{equation}
as soon as
\begin{equation}\label{eq:sigma-tau-L-square}
\dfrac{1}{\sigma} \dfrac{1}{\tau} \geq L_{\bm{K}}^{2}
\end{equation}
Therefore, we want to find $\sigma, \tau$ such that \eqref{eq:sigma-tau-L-square} holds for any matrix $\bm{K}$ visited by our algorithm. We then define $\Omega$ as in \eqref{eq:def-Omega}.

Note that Lemma \ref{lem:norm-k} gives the value of $L_{\bm{K}}$ for the $\ell_{1}$ and the $\ell_{2}$ setup. A naive choice of step sizes is then simply $\sigma = \tau = L_{\bm{K}}^{-1}$. However, a better choice is one where $\Theta_{X}/\tau = \Theta_{Y}/\sigma$, since $\Omega$, defined in \eqref{eq:def-Omega}, will appear in our upper bound on the error of our algorithm.
Under the condition \eqref{eq:sigma-tau-L-square}, we can choose $\tau = \left( \sqrt{\Theta_{X}/\Theta_{Y}} \right) L_{\bm{K}}^{-1}$ and $\sigma = L_{\bm{K}}^{-2}\tau^{-1}$. Note that this \textit{asymmetric} choice of step sizes essentially accounts for the difference in dimension between the space $\Delta(A)$ of $x$ and the space $\PP_{s} \subset \left( \Delta(S) \right)^{A} \subset \R^{A \times S}$ of $y$. 

The exact values of $\sigma,\tau$ in terms of $S,A$ are not needed here, as long as \eqref{eq:sigma-tau-L-square} is satisfied for all matrix $\bm{K}$ visited by our algorithm. Since $L_{\bm{K}}$ changes in the $\ell_{1}$ or $\ell_{2}$ setup and in order to keep a convergence rate independent of the choice of the proximal setup here, we differ giving the exact values of $\sigma,\tau$ to Appendix \ref{app:conv-ell-1}.
%
%
%For the $\ell_{1}$ setup, following Lemma \ref{lem:norm-k}, we can choose $\sigma=\tau \leq \dfrac{1}{ \lambda \| \bm{v}^{\ell} \|_{\infty}},$ for any epoch $\ell$. Note that at epoch $\ell$, by construction, the vector $\bm{v}^{\ell}$ corresponds to the reward obtained after $\ell$ periods by the sequence $(\bar{\bm{x}}^{\tau_{\ell}},\bar{\bm{y}}^{\tau_{\ell}}, ..., \bar{\bm{x}}^{0},\bar{\bm{y}}^{0}).$
%This implies that
%$ \| \bm{v}^{\ell} \|_{\infty} \leq r_{\infty}(1-\lambda^{\ell+1})(1-\lambda)^{-1} \leq r_{\infty}(1-\lambda)^{-1},$ where $r_{\infty} = \max_{s,a}  c_{sa}$.
%Therefore in the $\ell_{1}$ setup we can choose
%\begin{align*}
%\sigma & = \tau = \dfrac{1-\lambda}{\lambda r_{\infty}}, \\
%\Omega & = \dfrac{ 2 \lambda r_{\infty}}{1-\lambda} \left( \Theta_{X} + \Theta_{Y} \right).
%\end{align*}
%In the $\ell_{2}$, the same argument along with the equivalence between the $\ell_{2}$ and the $\ell_{1}$ norms yields
%\begin{align*}
%\sigma & = \tau = \dfrac{1-\lambda}{\lambda r_{\infty} \sqrt{S}}, \\
%\Omega & = \dfrac{ 2 \lambda r_{\infty} \sqrt{S}}{1-\lambda} \left( \Theta_{X} + \Theta_{Y} \right).
%\end{align*}

\paragraph{Bounding the duality gap.}
Recall the definition of $\LL^{\bm{K}}$ as in \eqref{def:LL}.
Let us focus on the error for $s \in \X$. Then
\begin{align}
& \LL^{\bm{K}^{*}}(\bar{\bm{x}}^{T},\bm{y}) - \LL^{\bm{K}^{*}}(\bm{x},\bar{\bm{y}}^{T})   \\
& \leq  \dfrac{1}{S_{T}} \left( \sum_{\ell =1}^{k} \sum_{t=\tau_{\ell}}^{\tau_{\ell}+T_{\ell}} \omega_{t} ( \LL^{\bm{K}^{*}}(\bm{x}^{t},\bm{y}) - \LL^{\bm{K}^{*}}(\bm{x},\bm{y}^{t})) \right) \nonumber 
\\ 
 &  \leq \dfrac{1}{S_{T}} \left( \sum_{\ell =1}^{k} \sum_{t=\tau_{\ell}}^{\tau_{\ell}+T_{\ell}} \omega_{t} (  \LL^{\bm{K}^{\ell}}(\bm{x}^{t},\bm{y}) - \LL^{\bm{K}^{\ell}}(\bm{x},\bm{y}^{t})) \right) \label{eq:bound_L_eq_1} \\
& + \dfrac{1}{S_{T}} \left( \sum_{\ell =1}^{k} \sum_{t=\tau_{\ell}}^{\tau_{\ell}+T_{\ell}} \omega_{t} (  \LL^{\bm{K}^{*}-\bm{K}^{\ell}}(\bm{x}^{t},\bm{y}) - \LL^{\bm{K}^{*}-\bm{K}^{\ell}}(\bm{x},\bm{y}^{t})) \right). \label{eq:bound_L_eq_2}
%& +  \dfrac{2}{S_{T}}  \left( \sum_{\ell=1}^{k}  \| \bm{A}^{*} - \bm{A}^{\ell} \|_{\infty}  \cdot \left(\sum_{t=\tau_{\ell}}^{\tau_{\ell}+T_{\ell}} \omega_{t} \right) \right) \\
%& \leq \dfrac{1}{S_{T}} \left( \sum_{\ell =1}^{k} \omega_{\tau_{\ell}} \Omega_{\ell} \right) + D,
\end{align}
\paragraph{Bounds on  \eqref{eq:bound_L_eq_1}.}
Let us first focus on the term at \eqref{eq:bound_L_eq_1}. By applying Proposition \ref{prop:descent-detail}, we obtain
\begin{align}
& \dfrac{1}{S_{T}} \left( \sum_{\ell =1}^{k} \sum_{t=\tau_{\ell}}^{\tau_{\ell}+T_{\ell}} \omega_{t} (  \LL^{\bm{K}^{\ell}}(\bm{x}^{t},\bm{y}) - \LL^{\bm{K}^{\ell}}(\bm{x},\bm{y}^{t})) \right) \\
 & = \dfrac{1}{S_{T}} \left( \sum_{\ell =1}^{k} \sum_{t=\tau_{\ell}}^{\tau_{\ell}+T_{\ell}} (\omega_{t+1}-\omega_{t}) (  A^{\bm{K}^{\ell}}(\bm{x},\bm{y},\bm{x}^{t},\bm{y}^{t}) \right) \nonumber \\
&   \leq  \dfrac{1}{S_{T}}\omega_{T} \Omega + \dfrac{1}{S_{T}} \sum_{\ell=1}^{k} \omega_{\tau_{\ell}} \LL^{\bm{K}^{\ell} - \bm{K}^{\ell-1}}(\bm{x}-\bm{x}^{t},\bm{y}-\bm{y}^{t}) \label{eq:bound-1} \\
&  \leq \dfrac{1}{S_{T}}\omega_{T} \Omega + \dfrac{4R_{X}R_{Y}}{S_{T}} \sum_{\ell=1}^{k} \omega_{\tau_{\ell}} L_{ \bm{K}^{\ell} - \bm{K}^{\ell-1}}  \nonumber \\
& \leq  \dfrac{1}{S_{T}}\omega_{T} \Omega + \dfrac{4R_{X}R_{Y}C}{S_{T}} \sum_{\ell=1}^{k} \omega_{\tau_{\ell}} \| \bm{v}^{\ell} - \bm{v}^{\ell-1} \|_{\infty} \label{eq:bound-2} \\
&  \leq  \dfrac{1}{S_{T}}\omega_{T} \Omega \nonumber \\
&  + \dfrac{4R_{X}R_{Y}C}{S_{T}} \sum_{\ell=1}^{k} \omega_{\tau_{\ell}}  \lambda^{\ell} (\| \bm{v}^{1} - \bm{v}^{0} \|_{\infty}  \\
& + \sum_{t=0}^{\ell-1} \left( \dfrac{\Omega_{t}}{T_{t}} + \dfrac{\Omega_{t-1}}{T_{t-1}} \right) \dfrac{1}{\lambda^{t}} ) \label{eq:bound-3}  \\
&  \leq  \dfrac{1}{S_{T}}\omega_{T} \Omega + \dfrac{4 err_{1,0}R_{X}R_{Y}C}{S_{T}} \sum_{\ell=1}^{k} \omega_{\tau_{\ell}}  \lambda^{\ell} \nonumber \\
& + \dfrac{4R_{X}R_{Y}\Omega C}{S_{T}} \sum_{\ell=1}^{k} \omega_{\tau_{\ell}}  \lambda^{\ell} \left( \sum_{t=0}^{\ell-1} \left( \dfrac{1}{T_{t}} + \dfrac{1}{T_{t-1}} \right) \dfrac{1}{\lambda^{t}} \right)\nonumber \\
& \leq e_{1} + e_{2} + e_{3},
\end{align}
where $C=1$ in the $\ell_{1}$ setup and $C=\sqrt{S}$ in the $\ell_{2}$ setup, where \eqref{eq:bound-1} follows from telescoping, \eqref{eq:bound-2} follows from Lemma \eqref{lem:norm-k}. Inequality \eqref{eq:bound-3} follows from Proposition \ref{prop:error_at_t} and $\epsilon_{t} = O\left (\Omega_{\ell}/T_{\ell} \right)$ in Proposition \ref{prop:descent-detail}, and 
\begin{align*}
e_{1}  & = \dfrac{1}{S_{T}}\omega_{T} \Omega,
e_{2}  = \dfrac{4err_{1,0}R_{X}R_{Y} C}{S_{T}} \sum_{\ell=1}^{k} \omega_{\tau_{\ell}}  \lambda^{\ell} , \\
e_{3} &  =\dfrac{4R_{X}R_{Y}\Omega C}{S_{T}} \sum_{\ell=1}^{k} \omega_{\tau_{\ell}}  \lambda^{\ell} \left( \sum_{t=0}^{\ell-1} \left( \dfrac{1}{T_{t}} + \dfrac{1}{T_{t-1}} \right) \dfrac{1}{\lambda^{t}} \right).
\end{align*}
\paragraph{Bounds on  \eqref{eq:bound_L_eq_2}.}
Note that 
\begin{align*}
 & \dfrac{1}{S_{T}} \left( \sum_{\ell =1}^{k} \sum_{t=\tau_{\ell}}^{\tau_{\ell}+T_{\ell}} \omega_{t} (  \LL^{\bm{K}^{*}-\bm{K}^{\ell}}(\bm{x}^{t},\bm{y}) - \LL^{\bm{K}^{*}-\bm{K}^{\ell}}(\bm{x},\bm{y}^{t})) \right) \leq \dfrac{2 R_{X}R_{Y}}{S_{T}}  \left( \sum_{\ell=1}^{k}  L_{ \bm{K}^{*} - \bm{K}^{\ell} }  \cdot \left(\sum_{t=\tau_{\ell}}^{\tau_{\ell}+T_{\ell}} \omega_{t} \right) \right) .
\end{align*}
Note that by Lemma \ref{lem:norm-k} we have $ L_{\bm{K}^{*} - \bm{K}^{\ell} } \leq C \| \bm{v}^{*} - \bm{v}^{\ell} \|_{\infty}$
and from Proposition \ref{prop:error_at_t} we have \[\| \bm{v}^{\ell} - \bm{v}^{*} \|_{\infty}  \leq \lambda^{\ell} \| \bm{v}^{0} - \bm{v}^{*} \|_{\infty} +\lambda^{\ell} \sum_{t=1}^{\ell-1} \dfrac{\Omega_{t}}{T_{t}\lambda^{t}} \] where the $\epsilon_{t}$ is replaced by $\Omega_{t} / T_{t},$ the precision attained after $T_{t}$ steps of PDA with payoff matrix $\bm{K}^{t}$.
This implies that we can have the following upper bound:
\[ L_{ \bm{K}^{\ell} - \bm{K}^{*} }  \leq C \lambda^{\ell} \| \bm{v}^{0} - \bm{v}^{*} \|_{\infty} +C\lambda^{\ell} \sum_{t=1}^{\ell-1} \dfrac{\Omega_{t}}{T_{t}\lambda^{t}} . \]
Overall, \eqref{eq:bound_L_eq_2} satisfies
\begin{align*}
& \dfrac{1}{S_{T}} \left( \sum_{\ell =1}^{k} \sum_{t=\tau_{\ell}}^{\tau_{\ell}+T_{\ell}} \omega_{t} (  \LL^{\bm{K}^{*}-\bm{K}^{\ell}}(\bm{x}^{t},\bm{y}) - \LL^{\bm{K}^{*}-\bm{K}^{\ell}}(\bm{x},\bm{y}^{t})) \right)      \leq \dfrac{2R_{X}R_{Y}C}{S_{T}} \sum_{\ell=1}^{k}  \left( \sum_{t=\tau_{\ell}}^{\tau_{\ell}+T_{\ell}} \omega_{t} \right) \cdot  \left( \lambda^{\ell} err_{0} +  \lambda^{\ell}\sum_{t=1}^{\ell-1} \dfrac{\Omega_{t} }{T_{t}\lambda^{t}} \right) \\
& \leq e_{4} + e_{5},
\end{align*}
where 
\begin{align*}
e_{4}  & =  \dfrac{2 err_{*,0} R_{X}R_{Y}C}{S_{T}} \sum_{\ell=1}^{k} \lambda^{\ell} \cdot \left( \sum_{t=\tau_{\ell}}^{\tau_{\ell}+T_{\ell}} \omega_{t} \right),  \\
e_{5} & =  \dfrac{2 R_{X}R_{Y}\Omega C}{S_{T} } \sum_{\ell=1}^{k} \lambda^{\ell} \left(  \sum_{t=1}^{\ell-1} \dfrac{1}{T_{t}\lambda^{t}} \right) \cdot \left( \sum_{t=\tau_{\ell}}^{\tau_{\ell}+T_{\ell}} \omega_{t} \right).
\end{align*}

The term $e_{1}$ is the upper bound that we would obtain if we had known the matrix $\bm{K}^{*}=\bm{K}[\bm{v}^{*}]$ from the start. The $e_{2}$ term comes from updating the value vector $\bm{v}^{\ell}$ to $\bm{v}^{\ell+1}$ at the end of the epoch $\ell$. The $e_{3}$ term comes from $\bm{v}^{\ell+1}$ being only an $O(1/T_{\ell})$ approximation of $F(\bm{v}^{\ell})$. The $e_{4}$ and $e_{5}$ terms are related to the error between $\bm{v}^{\ell}$ and $\bm{v}^{*}$.

We  emphasize the importance of warm-starting for Algorithm \ref{alg:PD-RMDP}. Crucially, by warm-starting the PDA algorithm at VI epoch $\ell+1$ using the last iterate of the PDA algorithm at VI epoch $\ell$, we are able maintain a telescopic sum from $t=0$ to $T=T_{1} + ... + T_{k}$. Without warm-starts, we would end up with $k$ independent telescopic sums (one per VI epoch). This would give an $e_{1}$ term of 
$\left(\sum_{\ell=1}^{k}\omega_{\tau_{\ell} + T_{\ell}}\right) \Omega / S_{T}$
 which is significantly worse than $\omega_{T} \Omega / S_{T}$, the $e_{1}$ term of Proposition \ref{thm:fom-vi rate}.

\paragraph{Convergence rates in terms of $T$}

We now investigate the convergence rate of the terms $e_{1}, ..., e_{5}$ in terms of the number of PD iterations $T$.
We have the following lemma.
For the sake of clarity, we hide in the $O \left( \cdot \right)$ notation any dependency on $S$ and $A$ (i.e. on $R_{X}, R_{Y}, \Omega$). 
\begin{lemma}\label{prop:rate-PD-general} Let $p, q \in \N$.
At time step $t \geq 0$, let $\omega_{t} = t^{p}, T_{\ell} = \ell^{q}.$ Then $\tau_{\ell} =O( \ell^{q+1}), T = O(k^{q+1}), S_{T} = O(k^{(p+1)(q+1)})$. Moreover,
\begin{align*}
e_{1} & = O \left(  \dfrac{1}{T} \right),  
e_{2}  = O \left( \dfrac{\lambda^{T^{1/(q+1)}}}{T} \right),  
e_{3}  = O \left( \dfrac{1}{T^{2q/(q+1)}} \right), \\
e_{4}  & = O \left( \dfrac{\lambda^{T^{1/(q+1)}}}{T^{1/(q+1)}} \right),
e_{5}  = O \left(\dfrac{1}{T^{q/(q+1)}} \right).
\end{align*}
\end{lemma}
\begin{proof}
Let $\omega_{t} = t^{p}, T_{\ell} = \ell^{q}$, for $t \geq 1$ and $p,q \in \N$.
We have
\begin{align*}
T & = \sum_{\ell=1}^{k} T_{\ell} = \sum_{\ell=1}^{k} \ell^{q} = k^{q+1}, \\
\tau_{\ell} & = \sum_{i=1}^{\ell} T_{i} \ell^{q+1}, \\
S_{T} & = \sum_{t=1}^{T} \omega_{t} = \sum_{t=1}^{k^{q+1}} t^{p} = k^{(q+1)(p+1)} = T^{p+1}.
\end{align*}
%Recall that
%\begin{align*}
%e_{1}  & = \dfrac{1}{S_{T}}\omega_{T} \Omega,
%e_{2} = \dfrac{4err_{1,0}R_{X}R_{Y}a}{S_{T}} \sum_{\ell=1}^{k} \omega_{\tau_{\ell}}  \lambda^{\ell}, 
%e_{3}  =\dfrac{4R_{X}R_{Y}a}{S_{T}} \sum_{\ell=1}^{k} \omega_{\tau_{\ell}}  \lambda^{\ell} \left( \sum_{t=0}^{\ell-1} \left( \dfrac{\Omega_{t}}{T_{t}} + \dfrac{\Omega_{t-1}}{T_{t-1}} \right) \dfrac{1}{\lambda^{t}} \right),\\
%e_{4} & = 2 err_{*,0} R_{X}R_{Y}c \dfrac{1}{S_{T}} \sum_{\ell=1}^{k} \lambda^{\ell} \cdot \left( \sum_{t=\tau_{\ell}}^{\tau_{\ell}+T_{\ell}} \omega_{t} \right), 
%e_{5} = 2 R_{X}R_{Y}c \dfrac{1}{S_{T} } \sum_{\ell=1}^{k} \lambda^{\ell} \left(  \sum_{t=1}^{\ell-1} \dfrac{\Omega_{t} }{T_{t}\lambda^{t}} \right) \cdot \left( \sum_{t=\tau_{\ell}}^{\tau_{\ell}+T_{\ell}} \omega_{t} \right).
%\end{align*}
For the sake of readability in the next bounds we hide the $O(\cdot )$ notations.
\paragraph{Bounds on $e_{1}$.} We have
\[ e_{1}  =  \dfrac{\omega_{T}}{S_{T}} =  \dfrac{T^{p}}{T^{p+1}} = \dfrac{1}{T}.\]
\paragraph{Bounds on $e_{2}$.} We have
\begin{align*}
e_{2}  & = \dfrac{1}{T^{p+1}} \sum_{\ell=1}^{k} \tau_{\ell}^{p} \lambda^{\ell} 
  =  \dfrac{1}{T^{p+1}} \sum_{\ell=1}^{k} \ell^{p(q+1)} \lambda^{\ell} =  \dfrac{1}{T^{p+1}} k^{p(q+1)} \lambda^{k} 
 =  \dfrac{1}{T^{p+1}} T^{p} \lambda^{T^{1/(q+1)}} 
 =  \dfrac{\lambda^{T^{1/(q+1)}}}{T}.
\end{align*}
\paragraph{Bounds on $e_{3}$.} We have
\begin{align*}
e_{3}  & = \dfrac{1}{k^{(p+1)(q+1)} }\sum_{\ell=1}^{k} \ell^{p(q+1)} \lambda^{\ell} \dfrac{1}{\ell^{q}\lambda^{\ell}}   = \dfrac{1}{k^{(p+1)(q+1)}} \sum_{\ell=1}^{k} \ell^{p(q+1)-q}  = \dfrac{1}{k^{(p+1)(q+1)}} k^{p(q+1)-q+1}  = \dfrac{1}{k^{2q}} 
  = \dfrac{1}{T^{2q/(q+1)}}.
\end{align*}
\paragraph{Bounds on $e_{4}$.}
First we need to compute the term $\sum_{t=\tau_{\ell}}^{\tau_{\ell}+T_{\ell}} \omega_{t}$.
We have
\begin{align*}
\sum_{t=\tau_{\ell}}^{\tau_{\ell}+T_{\ell}} \omega_{t} & = \sum_{t=\ell^{q+1}}^{\ell^{q+1}+\ell^{q}} t^{p}  = \sum_{t=0}^{\ell^{q}} \left( \ell^{q+1} + t \right)^{p}  = \sum_{t=0}^{\ell^{q}} \sum_{u=0}^{p} {p \choose u} \ell^{u(q+1)} t^{p-u}   = \sum_{u=0}^{p}  {p \choose u} \ell^{u(q+1)}   \sum_{t=0}^{\ell^{q}}  t^{p-u} \\
& = \sum_{u=0}^{p}  {p \choose u} \ell^{u(q+1)} \left( \ell^{q} \right)^{p-u+1}  = \sum_{u=0}^{p}  {p \choose u} \ell^{uq + u + qp - qu + q} = \sum_{u=0}^{p}  {p \choose u} \ell^{ u + (p+1)q }   = \ell^{(p+1)q} \sum_{u=0}^{p}  {p \choose u} \ell^{ u  } \\
& = \ell^{(p+1)q} (1+\ell)^{p} 
 = \ell^{(p+1)q + p}
 = \ell^{(p+1)(q+1)-1}.
\end{align*}
Now we have
\begin{align*}
e_{4}  & = \dfrac{1}{k^{(p+1)(q+1)}} \sum_{\ell=1}^{k} \lambda^{\ell}  \ell^{(p+1)(q+1)-1} = \dfrac{1}{k^{(p+1)(q+1)}} \lambda^{k} k^{(p+1)(q+1)-1} 
 = \dfrac{ \lambda^{k}}{k} 
 = \dfrac{\lambda^{T^{1/(q+1)}}}{T^{1/(q+1)}}.
\end{align*}
\paragraph{Bounds on $e_{5}$.} Let us bound $e_{5}$.
\begin{align*}
e_{5} & = \dfrac{1}{k^{(p+1)(q+1)}} \sum_{\ell=1}^{k} \lambda^{\ell} \dfrac{1}{\ell^{q}\lambda^{\ell}} \ell^{(p+1)(q+1)-1}  =  \dfrac{1}{k^{(p+1)(q+1)}} \sum_{\ell=1}^{k} \ell^{pq + p}  \\
 & = \dfrac{1}{k^{(p+1)(q+1)}}  k^{pq+p + 1} 
 =  \dfrac{1}{k^{q}} 
 = \dfrac{1}{T^{q/(q+1)}}. 
\end{align*}
\end{proof}
Theorem \ref{thm:fom-vi rate} follows directly from the previous lemma and our bound involving $e_{1}, ..., e_{5}$.
\section{Proof of Proposition \ref{prop:prop-fast-y-update-ell2-ell1}} \label{app:prop-fast-y-update}
Let
$ B_{2}(\bm{y}_{0}, \alpha) = \{ \bm{y} \in \R^{A \times S} \; | \; \dfrac{1}{2} \| \bm{y} - \bm{y}_{0} \|_{2}^{2} \leq \alpha \}.$
\subsection{$\ell_{2}$ setup for $y$}
The proximal update for $y$ becomes
\begin{align*}
\min \; & \langle\bm{y}_{s}, \bm{d}_{s}\rangle + \dfrac{1}{2\sigma}  \| \bm{y} - \bm{y'} \|^{2}_{2} \\
& \bm{y} = \left( \bm{y}_{a} \right)_{a \in \A} \in \left( \Delta(S) \right)^{A}, \\
& \bm{y} \in B_{2}(\bm{y}_{0}, \alpha).
\end{align*}
\paragraph{Introduce Lagrange multiplier for ball constraint.}
Let us write the Lagrangian function $F(\bm{y},\mu)$, where we introduce a Lagrangian multiplier $\mu \geq 0$ for the ball constraint, but we leave the simplex constraint unchanged:
\[ F(\bm{y},\mu) = \langle\bm{y}_{s}, \bm{d}_{s}\rangle + \dfrac{1}{2\sigma}  \| \bm{y} - \bm{y'} \|^{2}_{2} + \dfrac{\mu}{2}  \left( \| \bm{y} - \bm{y}_{0} \|^{2}_{2} - 2 \alpha \right).\]
Let us show that we can compute $\arg \min_{\bm{y} \in \left( \Delta(S) \right)^{A}} F(\bm{y},\mu)$ in complexity $O( AS \log(S))$.
Indeed,
\begin{align*}
 \arg \min_{\bm{y} \in \left( \Delta(S) \right)^{A}} F(\bm{y},\mu)  =
% \arg \min_{\bm{y} \in \left( \Delta(S) \right)^{A}} \langle\bm{y}_{s}, \bm{d}_{s}\rangle + \dfrac{1}{2 \sigma}  \| \bm{y} - \bm{y'} \|^{2}_{2} + \dfrac{\mu}{2}  \left( \| \bm{y} - \bm{y}_{0} \|^{2}_{2} - 2 \alpha \right) 
  \arg \min_{\bm{y} \in \left( \Delta(S) \right)^{A}} \dfrac{1}{2} \| \bm{y} - \dfrac{\sigma}{1+\sigma\mu} \left(  \dfrac{1}{\sigma}\bm{y'} + \mu \bm{y}_{0} - \bm{d} \right) \|_{2}^{2} &.
\end{align*}
%\begin{align*}
%\arg \min_{\bm{y} \in \left( \Delta(S) \right)^{A}} F(\bm{y},\mu) & = \arg \min_{\bm{y} \in \left( \Delta(S) \right)^{A}} \langle\bm{y}_{s}, \bm{d}_{s}\rangle + \dfrac{1}{2 \sigma}  \| \bm{y} - \bm{y'} \|^{2}_{2} + \dfrac{\mu}{2}  \left( \| \bm{y} - \bm{y}_{0} \|^{2}_{2} - 2 \alpha \right) \\
%& = \arg \min_{\bm{y} \in \left( \Delta(S) \right)^{A}} \sum_{a=1}^{A} \sum_{s'=1}^{S} d_{as'}y_{as'} + \dfrac{1}{2\sigma} \left( y_{as'} - y_{as'}' \right)^{2} + \dfrac{\mu}{2}   \left( y_{as'} - y_{0,as'} \right)^{2} \\
%& = \arg \min_{\bm{y} \in \left( \Delta(S) \right)^{A}} \sum_{a=1}^{A} \sum_{s'=1}^{S} \dfrac{1+\sigma\mu}{2\sigma} y_{as'}^{2} - \left( \dfrac{1}{\sigma} y_{as'}' + \mu y_{0,as'}-d_{as'} \right) y_{i} \\
%& = \arg \min_{\bm{y} \in \left( \Delta(S) \right)^{A}} \dfrac{1}{2} \| \bm{y} - \dfrac{\sigma}{1+\sigma\mu} \left(  \dfrac{1}{\sigma}\bm{y'} + \mu \bm{y}_{0} - \bm{d} \right) \|_{2}^{2}.
%\end{align*}
Therefore, we can reduce $\arg \min_{\bm{y} \in \left( \Delta(S) \right)^{A}} F(\bm{y},\mu)$ to solving $A$ Euclidean projections on the simplex $\Delta(S)$. Each Euclidean projection on the simplex $\Delta(S)$ can be done in $O(S\log(S))$ \citep{euclidean-projection}.
\paragraph{Binary search for optimal Lagrange multiplier $\mu^{*}.$}
Note that by definition, $q : \mu \mapsto F(\bm{x}^{*}(\mu),\mu)$ is a concave function on $\R_{+}$. Therefore, if we have an upper bound $\bar{\mu}$ on $\mu^{*}$ an optimal Lagrange multiplier, we can binary search the interval $[0,\bar{\mu}]$ to find a maximum of $q$.
\paragraph{Upper bound on the Lagrange multiplier.}
Note that 
\begin{align}
q(\mu)& =-\mu \alpha +   \min_{\bm{y} \in \left( \Delta(S) \right)^{A} } \langle \bm{y},\bm{d} \rangle + \dfrac{1}{2\sigma} \| \bm{y} - \bm{y'} \|_{2}^{2} + \dfrac{\mu}{2} \| \bm{y} - \bm{y}_{0} \|_{2}^{2} \nonumber \\
& \leq -\mu \alpha +  \langle \bm{y}_{0},\bm{d} \rangle + \dfrac{1}{2\sigma} \| \bm{y}_{0} - \bm{y'} \|_{2}^{2}. \label{eq:upper-bound-on-q}
\end{align}
Note that $q:\mu \mapsto q(\mu)$ is concave on $\R^{+}$. Therefore if we found $\bar{\mu}$ such that $q(\bar{\mu}) \leq q(0)$, we can claim that $\mu^{*} \in [0,\bar{\mu}]$, where $\mu^{*}$ attains the maximum of $q$. Using our upper bound \eqref{eq:upper-bound-on-q} on $q(\cdot)$ we know that we can choose any $\bar{\mu}$ such that
$-\mu \alpha +  \langle \bm{y}_{0},\bm{d} \rangle + \dfrac{1}{2\sigma} \| \bm{y}_{0} - \bm{y'} \|_{2}^{2} \leq q(0)$,
i.e. we choose an upper bound $\bar{\mu}$ as
\begin{equation*}
%\label{eq:upper-bound-lagrange}
\bar{\mu} = \dfrac{1}{\alpha} \left( \langle \bm{y}_{0},\bm{d} \rangle + \dfrac{1}{2\sigma} \| \bm{y}_{0} - \bm{y'} \|_{2}^{2} - q(0) \right).
\end{equation*}

\subsection{$\ell_{1}$ setup for $y$}
Let us fix $\beta \in \R$. For $\bm{d'} \in \R^{A \times S}, d'_{as'} = d_{as'} - (\beta/\sigma) \log(y'_{as'}),$ we can write the proximal update as
\begin{equation}
\begin{array}{rl}
\arg\min_{\bm{y}} \; & \langle \bm{y}, \bm{d'}_{s} \rangle + \dfrac{\beta}{\sigma} \sum_{a=1}^{A} \sum_{s'=1}^{S} y_{as'}\log y_{as'} \\
& \bm{y} = \left( \bm{y}_{a} \right)_{a \in \A} \in \left( \Delta(S) \right)^{A}, \\
& \bm{y} \in B_{2}(\bm{y}_{0}, \alpha).
\end{array}
\label{eq:l1 proximal mapping}
\end{equation}
\paragraph{Introduce Lagrange multiplier for ball constraint.}
Let us write the Lagrangian function $F(\bm{y},\mu)$, where we introduce a Lagrange multiplier $\mu \geq 0$ for the ball constraint, but we leave the simplex constraint unchanged.
\begin{align*}
  F(\bm{y},\mu) &= \langle  \bm{y}, \bm{d'}_{s} \rangle + \dfrac{\beta}{\sigma} \sum_{a=1}^{A} \sum_{s'=1}^{S} y_{as'}\log y_{as'} + \dfrac{\mu}{2}  \left( \| \bm{y} - \bm{y}_{0} \|^{2}_{2} - \alpha \right) \\
   &= \sum_{a=1}^{A}   \langle  \bm{y}_a, \bm{d'}_{sa} \rangle + \dfrac{\beta}{\sigma} \sum_{s'=1}^{S} y_{as'}\log y_{as'}  + \dfrac{\mu}{2}  \left(  \| \bm{y}_a - \bm{y}_{0,a} \|^{2}_{2} - \alpha \right).
\end{align*}
The key observation is that $F(\bm{y},\mu)$ is separable over the actions $a$. Therefore, in order to solve $\arg\min_{\bm{y} \in (\Delta(S))^A} F(\bm{y},\mu)$ we can solve $A$ subproblems, where for each $a=1,\ldots, A$ we solve the problem
\begin{equation}
\begin{aligned}\label{eq:sub-problem-a}
\arg\min_{\bm{y}_a \in \Delta(S)}
  \langle  \bm{y}_a, \bm{d'}_{sa} \rangle  & + \dfrac{\beta}{\sigma} \sum_{s'=1}^{S} y_{as'}\log y_{as'}+ \dfrac{\mu}{2}  \left(  \| \bm{y}_a - \bm{y}_{0,a} \|^{2}_{2} - \alpha \right).
\end{aligned}
\end{equation}
  
\paragraph{Introduce Lagrange multiplier for simplex constraint.}
We now introduce a further relaxation for each problem \eqref{eq:sub-problem-a}, by relaxing the simplex constraint using a Lagrange multiplier $\nu$ as follows:
\begin{align}
  & \arg\min_{\bm{y}_a \geq 0}
  \langle  \bm{y}_a, \bm{d'}_{sa} \rangle + \dfrac{\beta}{\sigma} \sum_{s'=1}^{S} y_{as'}\log y_{as'}  + \dfrac{\mu}{2}  \left(  \| \bm{y}_a - \bm{y}_{0,a} \|^{2}_{2} - \alpha \right) + \nu(\sum_{s'} y_{as'} - 1) \nonumber\\
  = &
  \arg\min_{\bm{y}_a \geq 0}
  \langle  \bm{y}_a, \bm{d'}_{sa} + \nu \rangle + \dfrac{\beta}{\sigma} \sum_{s'=1}^{S} y_{as'}\log y_{as'} + \dfrac{\mu}{2}  \| \bm{y}_a - \bm{y}_{0,a} \|^{2}_{2} \nonumber\\
  = &
  \arg\min_{\bm{y}_a \geq 0}
  \dfrac{\beta}{\sigma} \sum_{s'=1}^{S} y_{as'}\log y_{as'} + \dfrac{\mu}{2}  \| \bm{y}_a - \bm{y}_{0,a} + \dfrac{1}{\mu}(d'_{sa} + \nu) \|^{2}_{2} \nonumber\\
  = &
  \arg\min_{\bm{y}_a \geq 0}
   \sum_{s'=1}^{S} y_{as'}\log y_{as'} + \dfrac{\sigma\mu}{2\beta}  \| \bm{y}_a - \bm{y}_{0,a} + \dfrac{1}{\mu}(d'_{sa} + \nu) \|^{2}_{2}.
  \label{eq:var-wise simplex update}
\end{align}

We now arrive at a problem that decomposes into simple variable-wise updates: the negative entropy proximal mapping. For each variable $y_{a's}$ the update \eqref{eq:var-wise simplex update} is known to be equal (\cite{combettes2011proximal}) to
\begin{equation}\label{eq:lambert-update}
y_{as'} = \dfrac{\beta }{\sigma\mu} W\left( \dfrac{\sigma\mu}{\beta} \exp\left(\dfrac{\sigma\mu}{\beta}\left(y_{0,as'} - \dfrac{1}{\mu}(d'_{sas'}+\nu)\right) - 1\right) \right)
\end{equation}
where $W$ is the principal branch of the \emph{Lambert W function}, which is defined as the inverse of $w \mapsto w\log w$. The inverse is unique for $w\in [0,\infty)$. This function is not simple, but it can be computed quickly, and has standard implementations in the major numerical computing languages (e.g. in SciPy). As a heuristic benchmark, evaluating $W(a), a \in \mathbb{R}_{++}$ using SciPy takes about twice as long as evaluating $exp(a)$ (using numpy libraries for all function evaluations), based on generating 1000 random numbers in [0,1000].
Now we may find the appropriate $\nu^{*}$ such that the sum-to-one constraint is satisfied by binary search $\nu$.

\paragraph{Binary search for $\nu^{*}$.} Let $\nu \in \R$ and
 and $\bm{y}_{a}(\nu)$ the associated solution obtained from \eqref{eq:lambert-update}. If $\sum_{s'}^{S} y_{a}(\nu)_{s'} > 1$, then $\nu$ is a lower bound on $\nu^{*}$. Similarly, if $\sum_{s'}^{S} y_{a}(\nu)_{s'} < 1$, then $\nu$ is an upper bound on $\nu^{*}$. Since we know that $\nu^{*} > - \infty$, we can explore the set $\{ -2^{\ell} \; | \; \ell \geq 0 \}$ until we found a lower bound on $\nu^{*}$. If in this set we also found $\nu$ such that $\sum_{s'}^{S} y_{a}(\nu)_{s'} > 1$ then we also obtain an upper bound on $\nu^{*}$. Otherwise, we can explore the set $\{ 2^{\ell}  \;| \; \ell \geq 0 \}$ to find an upper bound on $\nu^{*}$.

Finally we get that we can reduce $\arg \min_{\bm{y} \in \left( \Delta(S) \right)^{A}} F(\bm{y},\mu)$ to a problem that can be solved in $\log(1/\epsilon)$ time, when treating evaluations of the Lambert W function as a constant.

Now that we have a method for computing $\arg \min_{\bm{y} \in \left( \Delta(S) \right)^{A}} F(\bm{y},\mu)$, we can now binary search the Lagrange multiplier $\mu$ in order to find a feasible solution to \eqref{eq:l1 proximal mapping}. 
\paragraph{Upper bound on $\mu^{*}$.} We know that $\mu^{*} \in [0,+\infty),$ where $\mu^{*}$ is an argmax of
\begin{align*}
q :\mu \mapsto \min_{\bm{y}_{a} \in \Delta(S)} \langle \bm{y}_{a},d_{sa} \rangle & + \dfrac{\beta}{\sigma} KL(\bm{y}_{a},\bm{y'}_{a}) + \mu \left( \| \bm{y}_{a} - \bm{y}^{0}_{a} \|^{2} - \alpha \right).
\end{align*} 
\begin{itemize}
\item There is a closed form solution for $q(0)$ since  this is the proximal update for the relative entropy.
\item We know that 
\[ q(\mu) \leq -\mu \alpha + \langle \bm{y}^{0}_{a},\bm{d}_{sa} \rangle + \dfrac{\beta}{\sigma} KL\left(\bm{y}^{0}_{a},\bm{y'}_{a} \right).\]
\item Therefore an upper bound $\bar{\mu}$ for $\mu^{*}$ is
\[ \bar{\mu} = \dfrac{1}{\alpha} \left( \langle \bm{y}^{0}_{a},\bm{d}_{sa} \rangle + \dfrac{\beta}{\sigma} KL\left(\bm{y}^{0}_{a},\bm{y'}_{a} \right) - q(0) \right).\]
\end{itemize}
We can then perform a binary search for $\mu^{*}$ in $[0,\bar{\mu}]$ exactly as for the $\ell_{2}$ setup.
\paragraph{Choice of the parameter $\beta$.}
Now we need to choose $\beta$ such that $\psi$ becomes strongly convex modulus $1$. If we set $\beta = \frac{A}{2}$ then we get strong convexity modulus $1$ with respect to the $\ell_1$ norm. To show this, we use the second-order definition of strong convexity:
\[
  \langle \nabla^2\psi(y)h, h \rangle \geq \|h\|_1^2, \forall y \in Y, h \in \mathbb{R}^{AS}
\]
Taking an arbitrary $h \in \mathbb{R}^{AS}$ we get from Cauchy-Schwarz:
\begin{align*}
  \left(\sum_a \sum_{s'} h_{as'}\right)^2
  & = \sum_a \left(\sum_{s'} h_{as'}\right)^2 + \sum_{a,a'} \left(\sum_{s'} h_{as'}\right)\left(\sum_{s'}h_{a's'}\right) \\
 &  \leq   \sum_a \frac{A}{2}\left(\sum_{s'} h_{as'}\right)^2    =  \sum_a \frac{A}{2}\left(\sum_{s'} \frac{h_{as'}}{\sqrt{y_{as'}}}\sqrt{y_{as'}}\right)^2 \\
 &  \leq   \sum_a \frac{A}{2}\|\sqrt{y_a}\|_2^2 \left(\sum_{s'} \frac{h_{as'}^2}{y_{as'}}\right)^2 \ =  \sum_a \frac{A}{2}\left(\sum_{s'} \frac{h_{as'}^2}{y_{as'}}\right)^2,
\end{align*}
which shows strong convexity modulus $1$ with respect to the $\ell_1$ norm.
\begin{remark} It may be possible to choose a stronger constant $\beta$, following \cite{juditsky-nemirovski-2011}, Chapter 5, pages 23-24. However, this would require to introduce a modified norm for element $(\bm{y}_{sa})_{a \in \A}$ of the set $\PP_{s}$. We leave this (potential) improvement for future work.
\end{remark}

\section{Details on the complexities of Theorem \ref{th:main-complexity}}\label{app:complexity-table}
\subsection{Summary of proximal setups for ellipsoidal uncertainty sets}
We first present a summary of the different sets and constants defined in this paper. 

\paragraph{Sets.}
\begin{enumerate}
\item $X = \Delta(A)$,
\item $Y = \left(\Delta(S) \right)^{A} \bigcap B_{2}(\bm{y}_{0},\alpha)$.
\end{enumerate}

\paragraph{Diameter and complexities.}
We call $R$ the maximum of the considered norm on the considered set $Z$: $R = \max_{z \in Z} \| z \|_{Z}$. We call $\Theta$ the maximum of the Bregman divergence $D$ on the considered set $Z$: $\Theta = \max_{z,z' \in Z} D(z,z')$. The complexity of computing the proximal update up to $\epsilon$  is $C_{\epsilon}$. We have:
\begin{enumerate}
\item For $\| \cdot \|_{X} = \| \cdot \|_{1}$:
\begin{itemize}
\item $\psi_{X} =$ entropy,
\item $R_{X} = O(1), \Theta_{X} = O(\log(A))$,
\item $comp_{\epsilon} = O(A)$.
\end{itemize}
\item For $\| \cdot \|_{X} = \| \cdot \|_{2}$:
\begin{itemize}
\item $\psi_{X} = (1/2) \| \cdot \|_{2}^{2}$,
\item $R_{X} = O(1), \Theta_{X} = O(1)$,
\item $comp_{\epsilon} = O(A\log(A))$.
\end{itemize}
\item For $\| \cdot \|_{Y} = \| \cdot \|_{1}$:
\begin{itemize}
\item $\psi_{Y} =$ sum-entropy,
\item $R_{Y} = O(A), \Theta_{Y} = O(A^{2}\log(S))$,
\item $comp_{\epsilon} = O(AS\log^{2}(\epsilon^{-1}))$.
\end{itemize}
\item For $\| \cdot \|_{Y} = \| \cdot \|_{2}$:
\begin{itemize}
\item $\psi_{Y} = (1/2) \| \cdot \|_{2}^{2}$,
\item $R_{Y} = O(\sqrt{A}), \Theta_{Y} = O(A)$,
\item $comp_{\epsilon} = O(AS\log(S)\log(\epsilon^{-1}))$.
\end{itemize}
\end{enumerate}
\subsection{Convergence in terms of number of iterations}
Note that the results of Theorem \ref{th:main-complexity} follows from the rate in Theorem \ref{thm:fom-vi rate} and the value of $R_{X}, R_{Y}, \Theta_{X}, \Theta_{Y}$ given in the above section. Here we provide details of the convergence rate computation in the case $q=2$.

\subsection{Overall complexity analysis for $(\| \cdot \|_{X}, \| \cdot \|_{Y}) = (\| \cdot \|_{1}, \| \cdot \|_{1})$}\label{app:conv-ell-1}
\paragraph{Step sizes}
For the $\ell_{1}$ setup, following Lemma \ref{lem:norm-k}, we have $L_{\bm{K}^{\ell}}= \lambda \| \bm{v}^{\ell} \|_{\infty},$ for any epoch $\ell$. Note that at epoch $\ell$, by construction, the vector $\bm{v}^{\ell}$ corresponds to the reward obtained after $\ell$ periods by the sequence $(\bar{\bm{x}}^{\tau_{\ell}},\bar{\bm{y}}^{\tau_{\ell}}, ..., \bar{\bm{x}}^{0},\bar{\bm{y}}^{0}).$
This implies that
$ \| \bm{v}^{\ell} \|_{\infty} \leq r_{\infty}(1-\lambda^{\ell+1})(1-\lambda)^{-1} \leq r_{\infty}(1-\lambda)^{-1},$ where $r_{\infty} = \max_{s,a}  c_{sa}$.
Therefore in the $\ell_{1}$ setup we can choose
\begin{align*}
 \tau & = \dfrac{1-\lambda}{\lambda r_{\infty} A } \dfrac{\sqrt{\log(A)}}{\sqrt{\log(S)}}, \\
  \sigma & = \dfrac{(1-\lambda)A}{\lambda r_{\infty}}  \dfrac{\sqrt{\log(A)}}{\sqrt{\log(S)}} \\
\Omega & = 2 A  \sqrt{\dfrac{\log(S)}{\log(A)}} \dfrac{\lambda r_{\infty}}{1-\lambda}.
\end{align*}
%In the $\ell_{2}$, the same argument along with the equivalence between the $\ell_{2}$ and the $\ell_{1}$ norms yields
%\begin{align*}
%\sigma & = \tau = \dfrac{1-\lambda}{\lambda r_{\infty} \sqrt{S}}, \\
%\Omega & = \dfrac{ 2 \lambda r_{\infty} \sqrt{S}}{1-\lambda} \left( \Theta_{X} + \Theta_{Y} \right).
%\end{align*}

We combine the definitions of the terms $e_{1}, ..., e_{5}$ from Proposition \ref{thm:fom-vi rate} (which includes the constants $R, \Theta, \Omega$ and $C$) with the convergence rates of Proposition \ref{prop:rate-PD-general}. Since $e_{5}$ has the slowest convergence rate, we use this term to give the overall number of arithmetic operations for Algorithm \ref{alg:PD-RMDP} to return an $\epsilon$-optimal solution to the robust MDP problem.
\paragraph{Convergence rate of PD.} For $q=2$, the error bounds of Proposition \ref{prop:rate-PD-general} become:
 \begin{align*}
e_{1} & = O \left( \dfrac{A \sqrt{\log(S)}}{T \sqrt{\log(A)}} \right), 
e_{2} = O \left( \dfrac{A\lambda^{T^{1/3}} }{T} \right), \\
e_{3} & =  O \left(  \dfrac{ A^{2} \sqrt{\log(S)}}{T^{4/3}\sqrt{\log(A)}} \right)  , 
e_{4}  = O \left( \dfrac{A^{2}\sqrt{ \log(S)}\lambda^{T^{1/3}}}{T^{1/3}\sqrt{\log(A)}}  \right),  \\
e_{5} & =  O \left(\dfrac{A^{2} \sqrt{\log(S)}}{T^{2/3}\sqrt{\log(A)}}  \right). 
\end{align*}
\paragraph{Complexity of PD update.}
For each epoch $\ell=1, ..., k$, solving each proximal update with accuracy $\epsilon>0$, the complexity of epoch $\ell$ is as follows.
\begin{align*}
comp_{\ell} & = O \left( \left(  AS^{2}\log^{2}(\epsilon^{-1})\right) T_{ell} \right).
\end{align*}
The overall complexity after $T=T_{1} + ... + T_{k}$ iterations is
\[ comp = O\left(  \left(  AS^{2} \log^{2}(\epsilon^{-1}) \right) T \right).\]
Since the $e_{5}$ term is the slowest to converge, for $q=2$ the number of arithmetic operations in order to obtain a $\epsilon$-optimal pairs in the robust MDP problem is 
$ O \left(A^{4}S^{2} \left( \dfrac{\log(S)}{\log(A)}\right)^{0.75} \log^{2}(\epsilon^{-1}) \epsilon^{-1.5}\right).$
 \subsection{Overall complexity analysis for $(\| \cdot \|_{X}, \| \cdot \|_{Y}) = (\| \cdot \|_{2}, \| \cdot \|_{2})$}\label{app:conv-ell-2}
 \paragraph{Step sizes}
 The same argument as in the $\ell_{1}$ setup, along with the equivalence between $\| \cdot \|_{2}$ and $\| \cdot \|_{\infty}$ in $\R^{S}$, yields
 \begin{align*}
 \tau & = \dfrac{1-\lambda}{\lambda r_{\infty} \sqrt{A} \sqrt{S}},\\
 \sigma & = \dfrac{(1-\lambda)\sqrt{A}}{\lambda r_{\infty} \sqrt{S}} \\
\Omega & = 2 \sqrt{A} \sqrt{S} \dfrac{\lambda r_{\infty}}{1-\lambda}.
\end{align*}
\paragraph{Convergence rate of PD.}
The error bounds of Proposition \ref{prop:rate-PD-general} become, for $q=2$,:
 \begin{align*}
e_{1} & = O \left( \dfrac{\sqrt{AS}}{T} \right), 
e_{2}= O \left( \dfrac{\sqrt{AS}\lambda^{T^{1/3}} }{T} \right), 
e_{3}  =  O \left(  \dfrac{ AS}{T^{4/3}} \right)  , \\
e_{4} & = O \left( \dfrac{AS\lambda^{T^{1/3}}}{T^{1/3}}  \right), 
e_{5}  =  O \left(\dfrac{AS}{T^{2/3}}  \right). 
\end{align*}
\paragraph{Complexity of PD update.}
For each epoch $\ell=1, ..., k$, the complexity of epoch $\ell$ is as follows.
\begin{align*}
comp_{\ell} = O \left( \left( SA\log(A)  + AS^{2} \log(S)\log(\epsilon^{-1})\right) T_{ell} \right).
\end{align*}
The overall complexity after $T=T_{1} + ... + T_{k}$ iterations is
\[ comp = O\left(  \left( SA\log(A) + AS^{2} \log(S)\log(\epsilon^{-1}) \right) T \right).\]
Typically, $\log(A) \leq S$, and we have 
$ comp = O\left(   AS^{2} \log(S)\log(\epsilon^{-1}) T \right).$
Therefore, for $q=2$, the number of arithmetic operations in order to obtain a $\epsilon$-optimal pairs in the robust MDP problem is
$ O \left( A^{2.5}S^{3.5} \log(S) \log(\epsilon^{-1})\epsilon^{-1.5} \right).$

\section{Kullback-Leibler uncertainty set}\label{app:KL-uncertainty-set}
We present here our complexity result for the KL uncertainty set. Recall that the KL uncertainty set is defined as
\begin{align*}
\PP & = \times_{s \in \X} \PP_{s}, \\
\PP_{s} & =\{\left( \bm{y}_{sa} \right)_{a \in \A} \in (\Delta(S))^{A} \; | \; \sum_{a \in \A} KL(\bm{y}_{sa},\bm{y}^{0}_{sa}) \leq \alpha\}.
\end{align*}

We now prove Proposition~\ref{prop:prop-fast-y-update-KL}.
%Consider proximal updates \eqref{eq:prox_update_y_simple} for this choice of uncertainty set.  We have the following proposition.
% Interestingly, we remark that the complexities of the proximal updates for the $\ell_{1}$ and the $\ell_{2}$ proximal setups are now flipped for KL uncertainty set (compared to ellipsoidal uncertainty set).
As the proof follows closely the lines of the proofs for the proximal updates on the ellipsoidal uncertainty set, for the sake of conciseness we only present an outline here.
\begin{proof}
\noindent
\textbf{$\ell_{2}$ setup.} We introduce a Lagrange multiplier for the KL constraint, and the proximal update boils down to solving $A$ subproblems, each consisting ot optimizing the sum of a linear form, an entropy function and an $\ell_{2}$ distance. This is equivalent to solving subproblems of the form \eqref{eq:sub-problem-a}. Therefore, the $\ell_{2}$ proximal update for a KL uncertainty set can be approximated within accuracy $\epsilon$ in $O\left( AS  \log^2(\epsilon^{-1})\right)$.

\noindent
\textbf{$\ell_{1}$ setup.}   For the $\ell_{1}$ setup, we can introduce a Lagrange multiplier for the KL constraint; the objective becomes separable into $A$ subproblems, each requiring to optimize (over the simplex of size $\Delta(S)$) the sum of a linear form and two KL terms, which brings down to optimizing, over the simplex, the sum of a linear form and a KL term. This can be computed in closed-form, and the $\ell_{2}$ proximal update boils down to a bisection search onto the Lagrange multiplier. Therefore, the $\ell_{2}$ proximal update for a KL uncertainty set can be approximated within accuracy  $O\left( AS  \log(\epsilon^{-1}) \right)$.
\end{proof}

\section{Performance measures for our simulations}\label{app:upper-bounds}
\paragraph{Computing (DG).}
In order to compute (DG) for a pair $\bm{x},\bm{y}$, we need to evaluate $\max_{\bm{y'} \in \PP} R(\bm{x},\bm{y'}) $ and $\min_{\bm{x'} \in \Pi} R(\bm{x'},\bm{y}).$
Following \cite{Kuhn}, $\max_{\bm{y'} \in \PP} R(\bm{x},\bm{y'}) $ can be computed by finding the fixed point of the following operator, which is a contraction of factor $\lambda$:
\[ F^{\bm{x}}(\bm{v})_{s} = \max_{\bm{y'}_{s} \in \PP} \sum_{a=1}^{A} x_{sa} \left( c_{sa} + \lambda \bm{y'}^{\top}\bm{v} \right), \forall \; s \in \X.\]
Moreover, computing $\min_{\bm{x'} \in \Pi} R(\bm{x'},\bm{y})$ is equivalent to solving the (nominal) MDP with fixed kernel $\bm{y} \in \PP$. Following \cite{Puterman}, Chapter 6.3, this can be solved by iterating the following contraction of factor $\lambda$:
\[ F^{\bm{y}}(\bm{v})_{s} = \min_{\bm{x_{s}} \in \Delta(A)} \sum_{a=1}^{A} x_{sa} \left( c_{sa} + \lambda \bm{y'}^{\top}\bm{v} \right), \forall \; s \in \X.\]
Each of these iterative algorithms can be stopped as soon $\| \bm{v}^{\ell+1} - \bm{v}^{\ell} \|_{\infty} < 2 \lambda \epsilon (1-\lambda)^{-1}$, which ensures $\epsilon$-optimality of the current iterates \citep{Puterman}, Chapter 6.3.

We present in the next figure the running times to compute (DG), both with Algorithm \ref{alg:VI} and Algorithm \ref{alg:AVI}. In particular, we generate 10 random Garnet MDP instances (see simulation section in the main body and next section), some random policies in $\Pi$, kernels in $\PP$ and vectors in $\R^{S}$ and we compute (DG). We present the logarithm of the average running times to obtain $\epsilon$-approximations of the quantities of interest, for $\epsilon=0.25$ and $\alpha = \sqrt{n_{branch} \times A}.$ We present our results for $\lambda=0.6$ in Figure \ref{fig:UB_comp_lambda_60} and for $\lambda=0.8$ in Figure \ref{fig:UB_comp_lambda_80}. We notice that computing (DG) quickly becomes very expensive, even using Algorithm \ref{alg:AVI}.

  \begin{figure*}[htp]
  \begin{subfigure}{0.4\textwidth}
\centering
  \includegraphics[width=1.0\linewidth]{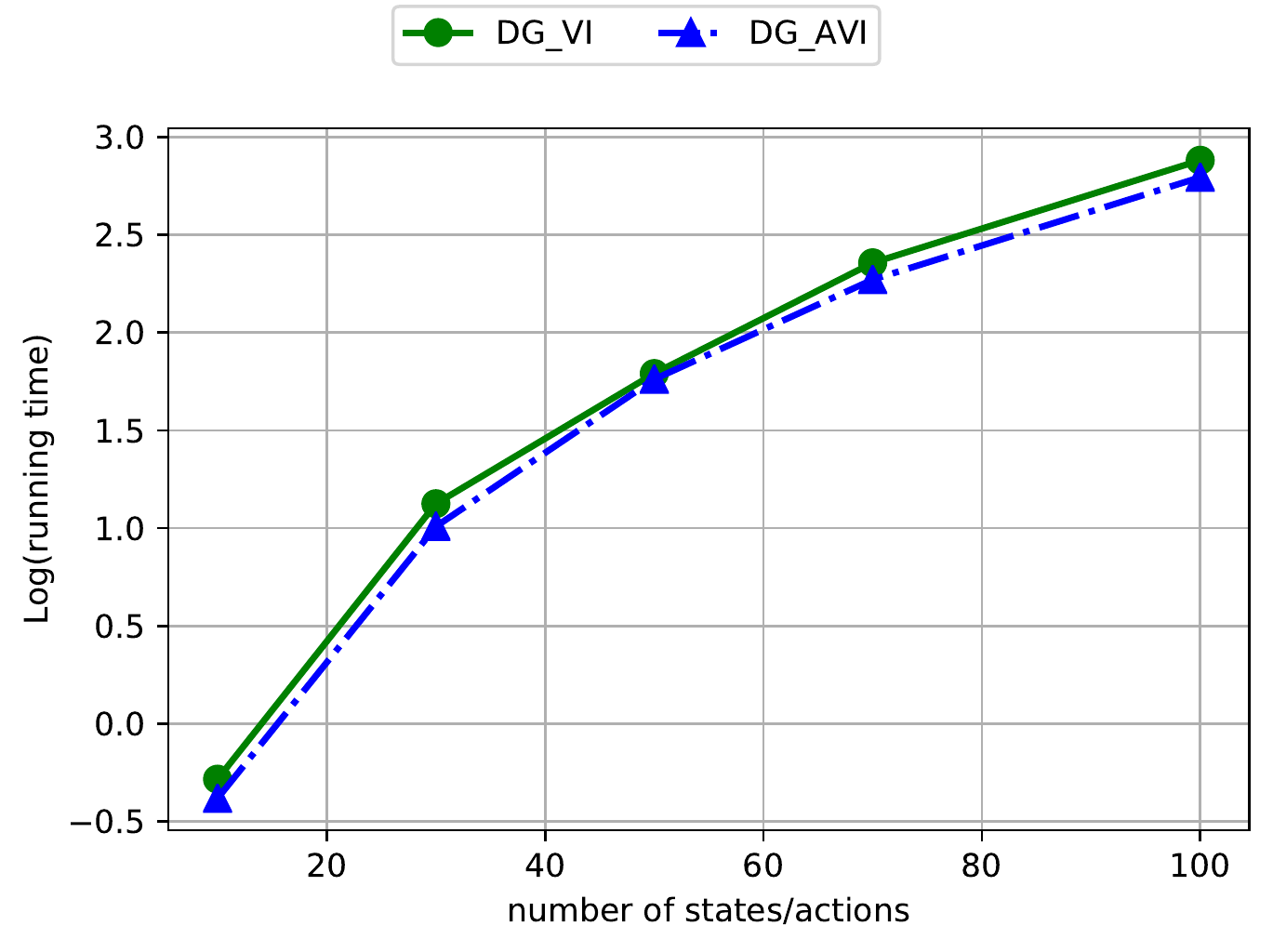}
  \captionof{figure}{Running times for $\lambda=0.6$.}
  \label{fig:UB_comp_lambda_60}
  \end{subfigure}
  \begin{subfigure}{0.4\textwidth}
\centering
  \includegraphics[width=1.0\linewidth]{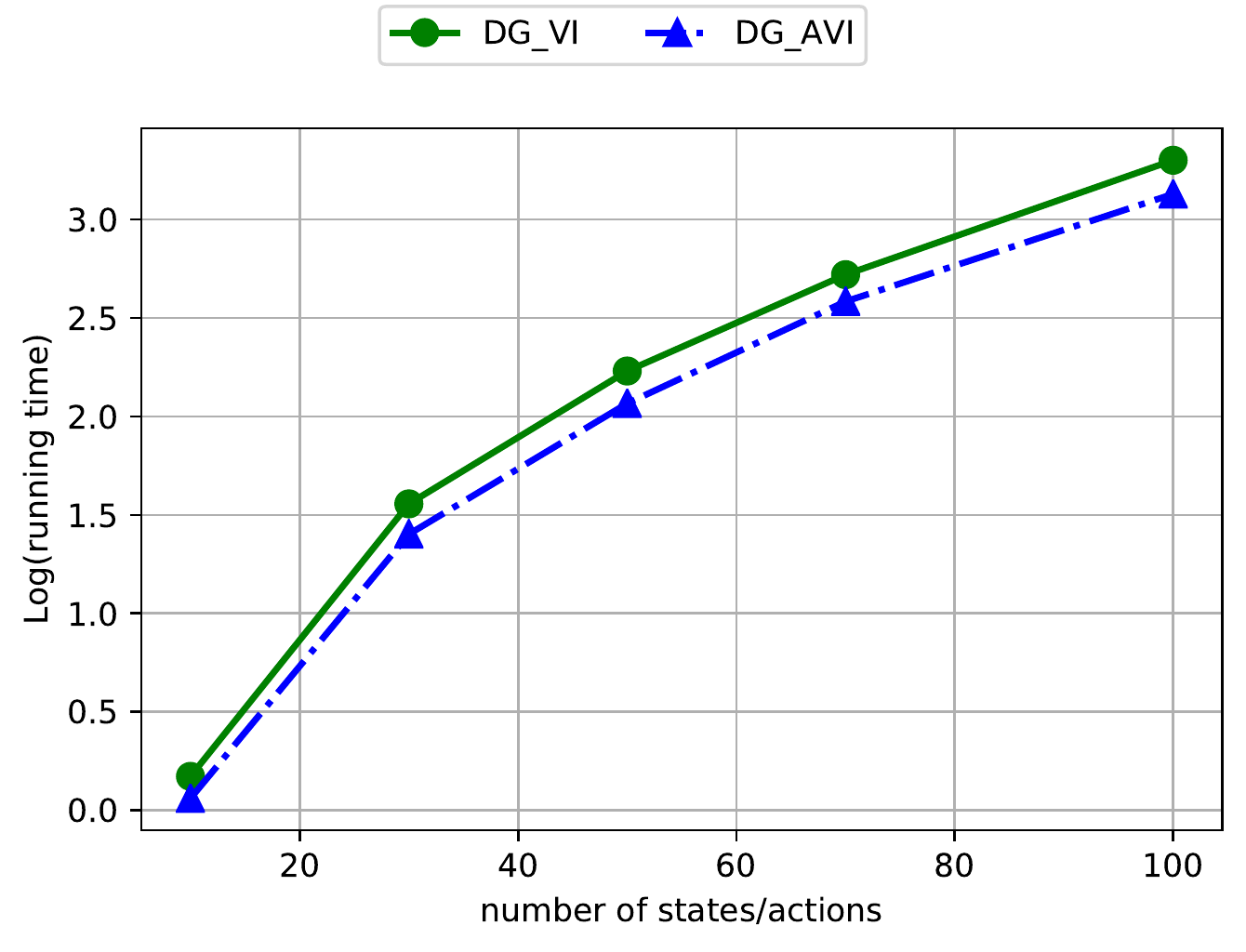}
  \captionof{figure}{Running times for $\lambda=0.8$.}
  \label{fig:UB_comp_lambda_80}
  \end{subfigure}
  \caption{Average running time (s) to compute the duality gap of a Garnet MDP instance.}
\end{figure*}

%\begin{figure}[H]
%\centering
%  \includegraphics[width=1.0\linewidth]{figure_2/Running_Times_Upper_Bounds_Computations_Lamb_60}
%  \captionof{figure}{Running times for $\lambda=0.6$.}
%  \label{fig:UB_comp_lambda_60}
%\end{figure}
%\begin{figure}[H]
%  \centering
%  \includegraphics[width=1.0\linewidth]{figure_2/Running_Times_Upper_Bounds_Computations_Lamb_80}
%  \captionof{figure}{Running times for $\lambda=0.8$.}
%  \label{fig:UB_comp_lambda_80}
%\end{figure}

\section{Comparison of proximal setups for Algorithm \ref{alg:PD-RMDP}}\label{app:comparison-our-algorithm}
In this appendix we study the empirical performances of Algorithm \ref{alg:PD-RMDP}, in order to identify the best one, which we will then compare to other VI approaches. 

\noindent\textbf{Empirical setup.}
 All the simulations are implemented in Python 3.7.3, and were performed on a laptop with 2.2 GHz Intel Core i7 and 8 GB of RAM. We use Gurobi 8.1.1 to solve any linear or quadratic optimization problems involved. We generate Garnet MDPs (\citet{garnet}), which are an abstract class of MDPs parametrized by a branching factor $n_{branch}$, equal to the number of reachable next states from each state-action pair $(s,a)$. We consider $n_{branch}=0.5$ in our simulations. We draw the rewards parameters at random uniformly in $[0,10]$. We fix a discount factor $\lambda=0.8$. The radius $\alpha$ of the $\ell_{2}$ ball from the uncertainty set \eqref{eq:uncertainty-norm-2} is set to $\alpha = \sqrt{n_{branch} \times A }.$ All of the figures in this section show the logarithm of the performance measures (DG) in terms of the number of PD iteration performed in Algorithm \ref{alg:PD-RMDP}. Apart from Figures \ref{fig:comparison_norms_pq_11}-\ref{fig:comparison_norms_pq_22}, these performance measures are averaged across 10 randomly generated Garnet MDPs.

\noindent\textbf{Impact of proximal setup.}
We fix $S,A=30$ and we present in Figure \ref{fig:comparison_norms_pq_11}-\ref{fig:comparison_norms_pq_22} the Duality Gap (DG) of the current weighted average of the iterates of our algorithm, for three different proximal setups $(\| \cdot \|_{X},\| \cdot \|_{Y} ) \in \{ ( \ell_{1},\ell_{1}), (\ell_{1},\ell_{2}),(\ell_{2},\ell_{2})\}.$ The $(\ell_{2},\ell_{2})$ setup performs the best, even though its theoretical guarantees are worse than the $(\ell_{1},\ell_{1})$ setup (as seen in Theorem  \ref{th:main-complexity}). This disparity between theory and practice is analogous to the case of stationary bilinear min-max problems~\citep{GKG20}. In the rest of the simulations we focus on the $(\ell_{2},\ell_{2})$ setup. Note that Figure~\ref{fig:comparison_norms_pq_11}-\ref{fig:comparison_norms_pq_22} shows performance for a single instance. This is because the $(\ell_{1},\ell_{1})$ setup takes almost a day to run on a single instance of size $(S,A)=(30,30)$ (compared to minutes for the $(\ell_{2}, \ell_{2})$ setup), most likely because of the two interwoven binary searches (see also Appendix \ref{app:prop-fast-y-update}).
%\begin{figure}[H]
%\centering
%\begin{minipage}{.5\textwidth}
%  \centering
%  \includegraphics[width=1.0\linewidth]{figure/Norm_comparison_pq_11}
%  \captionof{figure}{$p,q=1,1$}
%  \label{fig:comparison_norms_pq_11}
%\end{minipage}%
%\begin{minipage}{.5\textwidth}
%  \centering
%  \includegraphics[width=1.0\linewidth]{figure/Norm_comparison_pq_22}
%  \captionof{figure}{$p,q=2,2$}
%  \label{fig:comparison_norms_pq_22}
%\end{minipage}
%\end{figure}

  \begin{figure*}[htp]
  \centering
  \begin{subfigure}{0.4\textwidth}
\centering
  \includegraphics[width=0.8\linewidth]{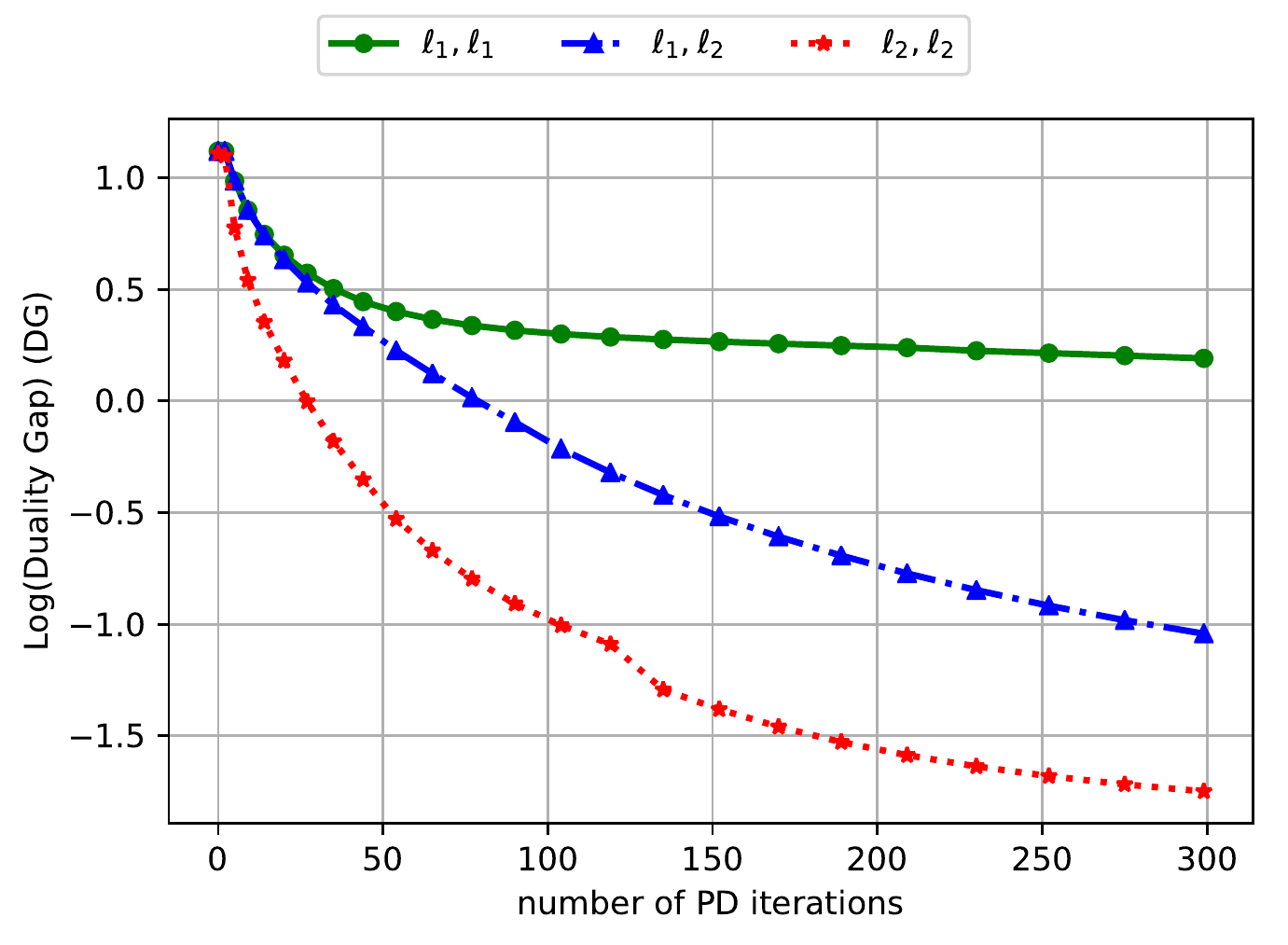}
  \captionof{figure}{Proximal setups comparison for $(p,q)=(1,1)$.}
  \label{fig:comparison_norms_pq_11}
  \end{subfigure}
  \begin{subfigure}{0.4\textwidth}
\centering
  \includegraphics[width=0.8\linewidth]{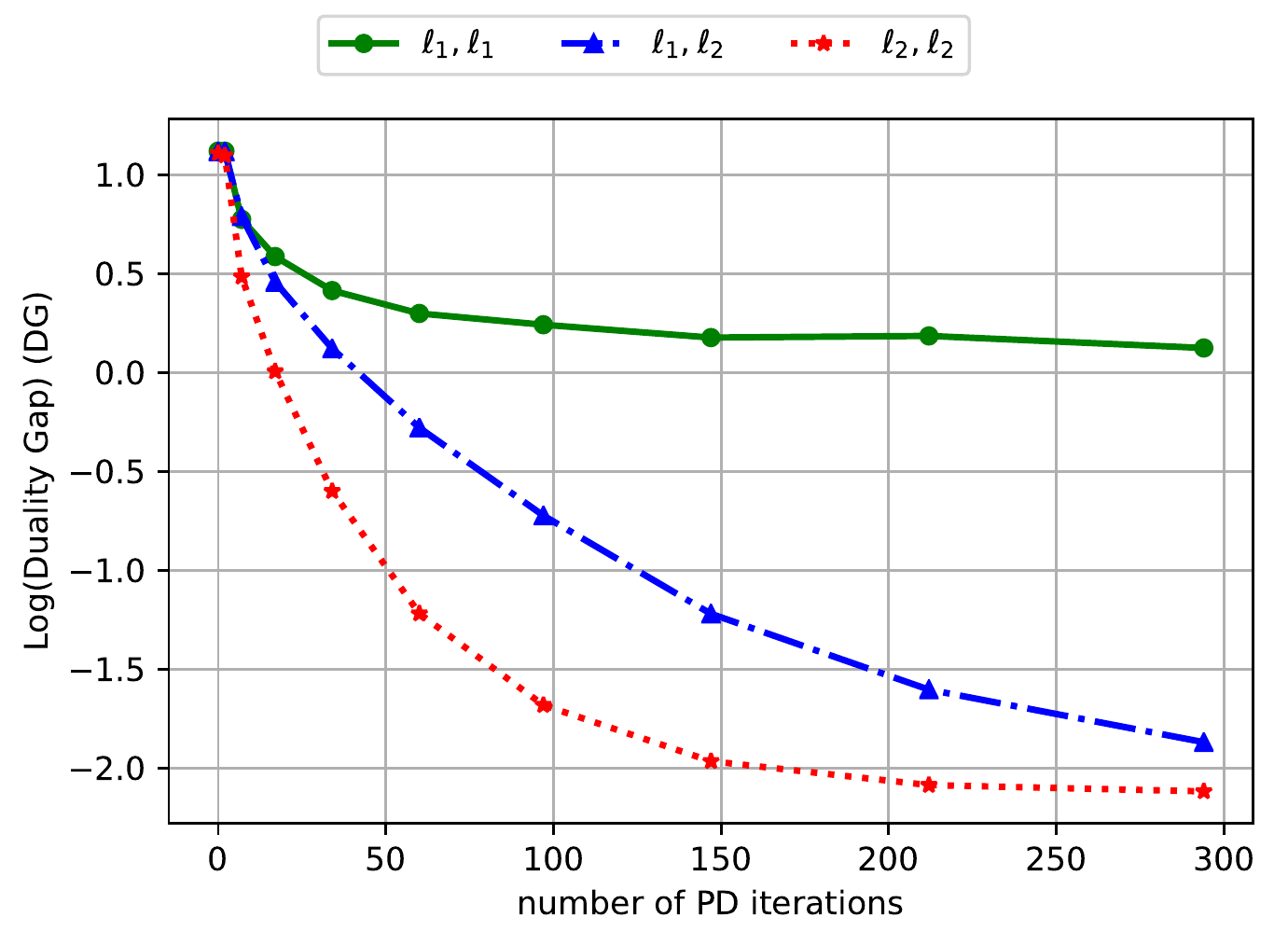}
  \captionof{figure}{Proximal setups comparison for $(p,q)=(2,2)$.}
  \label{fig:comparison_norms_pq_22}
  \end{subfigure}
  \caption{Comparison of proximal setups for various proximal setups.}
\end{figure*}

%\begin{figure}[H]
%\centering
%  \centering
%  \includegraphics[width=0.8\linewidth]{figure_2/Norm_comparison_pq_11_log}
%  \captionof{figure}{Proximal setups comparison for $(p,q)=(1,1)$.}
%  \label{fig:comparison_norms_pq_11}
%\end{figure}
%\begin{figure}[H]  \centering
%  \includegraphics[width=0.8\linewidth]{figure_2/Norm_comparison_pq_22_log}
%  \captionof{figure}{Proximal setups comparison for $(p,q)=(2,2)$.}
%  \label{fig:comparison_norms_pq_22}
%\end{figure}

\noindent\textbf{Impact of epoch scheme.} 
We now investigate the impact of the epoch length $T_{\ell}=\ell^{q}$, parametrized by $q \in \N$.
We fix $(S,A)=(30,30)$ and we focus on the $(\ell_{2},\ell_{2})$ setup. We fix the averaging scheme at $p=1$ and we compare epoch lengths $q=0,1,2.$
The results are shown in Figure \ref{fig:DG_epoch}.
For the performance measure (DG), we find that $q=0,q=1$ and $q=2$ yield comparable convergence rates (in terms of number of PD iterations), with $q=2$ being slightly better than $q=0,q=1$ (note that our theory does not even guarantee convergence for $q=0$).
Note that for $q=0$, our algorithm performs only one PD update at each epoch, before updating the value vector $\bm{v}$, which has a cost of $O(AS^{2})$. This may make $q=0$ significantly slower in practice for large $S,A$, since the value vector updates have a negligible computational cost for $q> 0$ (compared to the numerous PD updates computational costs).

  \begin{figure*}[htp]
  \centering
  \begin{subfigure}{0.4\textwidth}
\centering
  \includegraphics[width=0.8\linewidth]{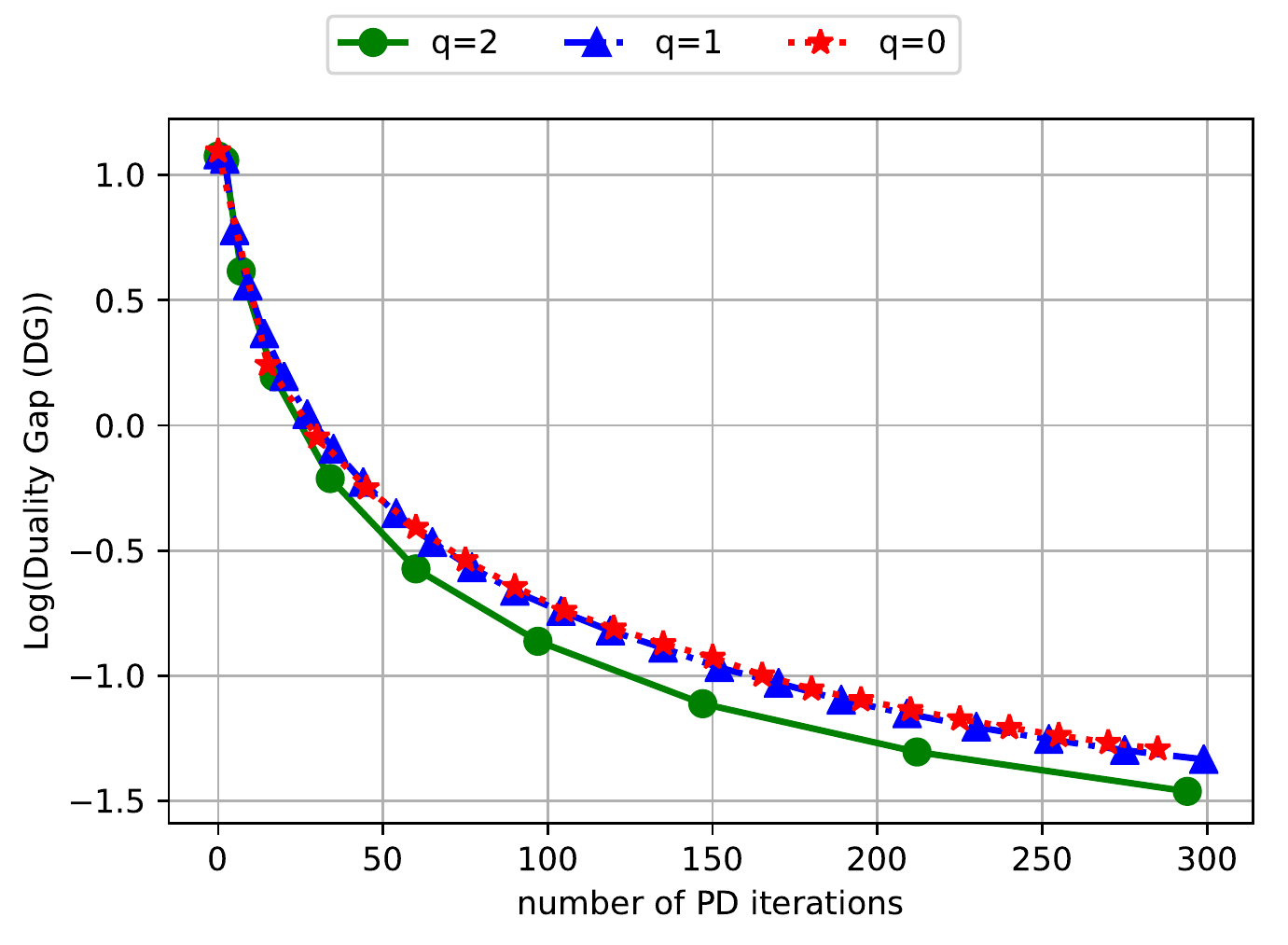}
  \captionof{figure}{ (DG) in terms of number of PD iterations.}
  \label{fig:DG_epoch}
  \end{subfigure}
  \begin{subfigure}{0.4\textwidth}
\centering
  \includegraphics[width=0.8\linewidth]{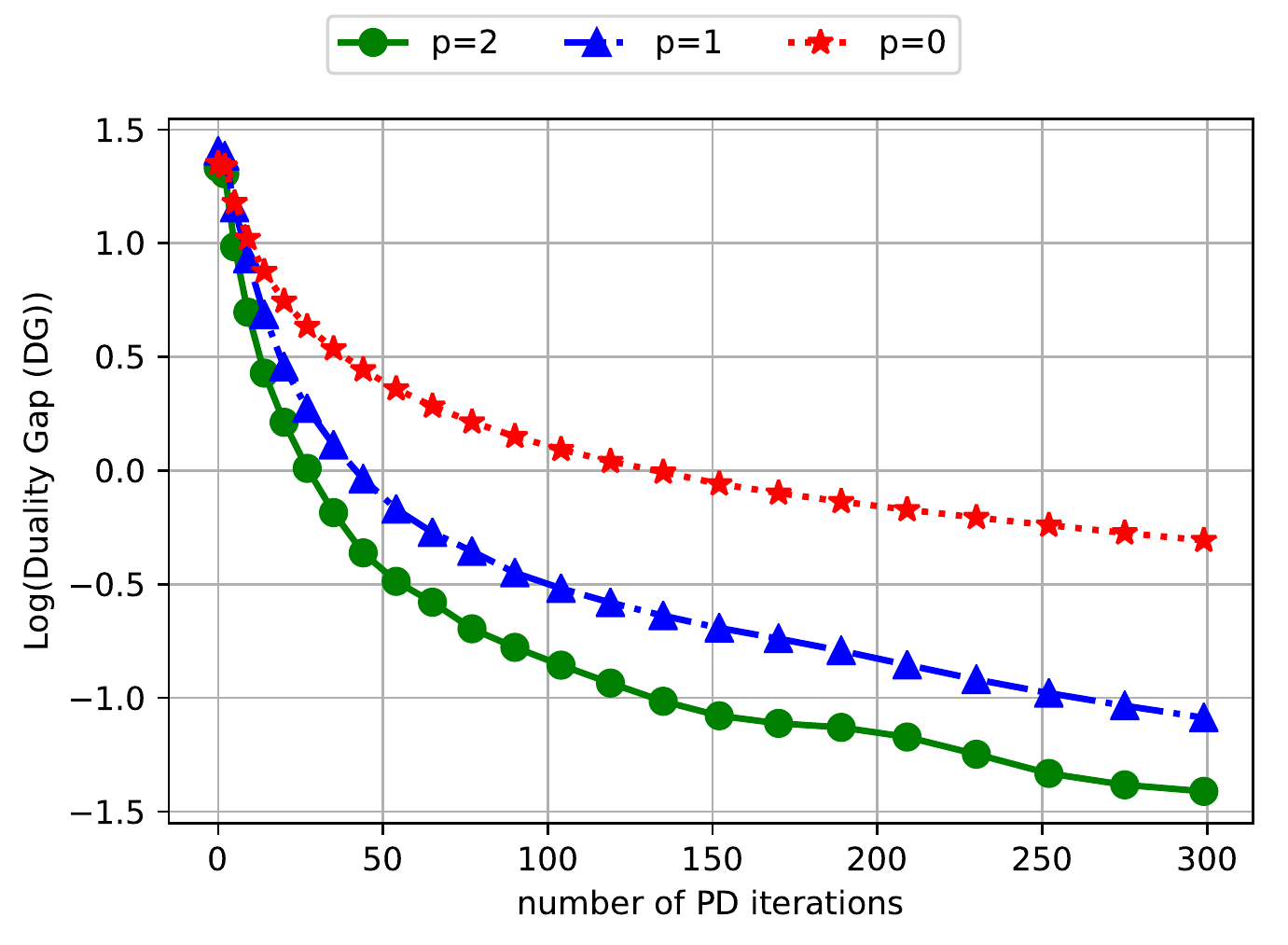}
  \captionof{figure}{ (DG) in terms of number of PD iterations.}
  \label{fig:DG_weight}
  \end{subfigure}
  \caption{Performances of our algorithm for various weights and epoch schemes.}
\end{figure*}

%\begin{figure}[H]
%\centering
%  \centering
%  \includegraphics[width=0.8\linewidth]{figure_2/epoch_comparison_DG_2}
%  \captionof{figure}{ (DG) in terms of number of PD iterations}
%  \label{fig:DG_epoch}
%\end{figure}

\noindent\textbf{Impact of weight scheme.} We now investigate the impact of the weight scheme $\omega_{t}=t^{p}$ used to average iterates.%, parametrized by $p \in \N$.
We fix $(S,A)=(30,30)$, $q=2$ and use the $(\ell_{2},\ell_{2})$ setup. We compare $p=0,1,2.$
Figure \ref{fig:DG_weight} shows that increasing averages ($p=1,p=2$) perform better than uniform average ($p=0$), even though our convergence guarantees are independent of $p$. Similar observations have been made in zero-sum games and other convex-concave saddle-point problems~\citep{GKG20}.

%\begin{figure}[H]
%\centering
%\begin{minipage}{.3\textwidth}
%  \centering
%  \includegraphics[width=1.0\linewidth]{figure/Errors_k_10_p_0_q_2_ell_2_lamb_80_SA_30}
%  \captionof{figure}{$p=0$}
%  \label{fig:p_zero}
%\end{minipage}%
%\begin{minipage}{.3\textwidth}
%  \centering
%  \includegraphics[width=1.0\linewidth]{figure/Errors_k_10_p_1_q_2_ell_2_lamb_80_SA_30}
%  \captionof{figure}{$p=1$}
%  \label{fig:p_one}
%\end{minipage}
%\begin{minipage}{.3\textwidth}
%  \centering
%  \includegraphics[width=1.0\linewidth]{figure/Errors_k_10_p_2_q_2_ell_2_lamb_80_SA_30}
%  \captionof{figure}{$p=2$}
%  \label{fig:p_two}
%\end{minipage}
%\end{figure}

%\begin{figure}[H]
%\centering
%  \includegraphics[width=0.8\linewidth]{figure_2/weight_comparison_DG_2}
%  \captionof{figure}{ (DG) in terms of number of PD iterations}
%  \label{fig:DG_weight}
%\end{figure}

\begin{remark} In our simulations we set $\lambda=0.8$. Of course, our algorithm works for any $\lambda \in (0,1).$ However, the performance guarantees of Algorithm \ref{alg:PD-RMDP} may degrade for $\lambda \rightarrow 1$, as some of the constants in the $O \left( \cdot \right)$ notations of Theorem \ref{th:main-complexity} depend on $1/(1-\lambda)$, a situation similar to the complexity of Value Iteration \ref{alg:VI}. Moreover, when $\lambda \rightarrow 1$, computing the duality gap (DG) becomes very slow: computing the minimizer and maximizer requires iterating contraction mappings, each with improvement factor $\lambda$ (see Appendix \ref{app:upper-bounds}). These last two limitations are not a particular shortcoming of our algorithm but are inherent to MDPs.
\end{remark}

\section{MDP instance inspired from healthcare}\label{sec:healthcare-example}
We present here an example of the nominal kernel $\bm{y}_{0}$ of the healthcare MDP instance that we introduced in our simulation section. We show here an example with $S=5$ health condition states and the mortality state.  The state $1$ corresponds to a healthy condition while $S=5$ is more likely to lead to mortality. The transition kernel for general $S \geq 5$ are generated in the same fashion. In order to sample $N$ kernels $\bm{y}_{1}, ..., \bm{y}_{N}$ around the nominal transition $\bm{y}_{0}$, we generate random Garnet MDPs $\bm{\tilde{y}}_{1}, ..., \bm{\tilde{y}}_{N}$ with $n_{branch}=20 \%$ and we obtain $N$ samples $\bm{y}_{1}, ..., \bm{y}_{N}$ around $\bm{y}_{0}$ as 
\[\bm{y}_{i} = 0.95 \cdot \bm{y}_{0} + 0.05 \cdot \bm{\tilde{y}}_{i}, i = 1, ..., N.\]
We choose the coefficients $(0.95,0.05)$ in the above convex combination so that we obtain $\bm{y}_{i}$ as small deviations from $\bm{y}_{0}$.
  \begin{figure*}[htp]
  \centering
  \begin{subfigure}{0.3\textwidth}\centering
  \includegraphics[width=0.8\linewidth]{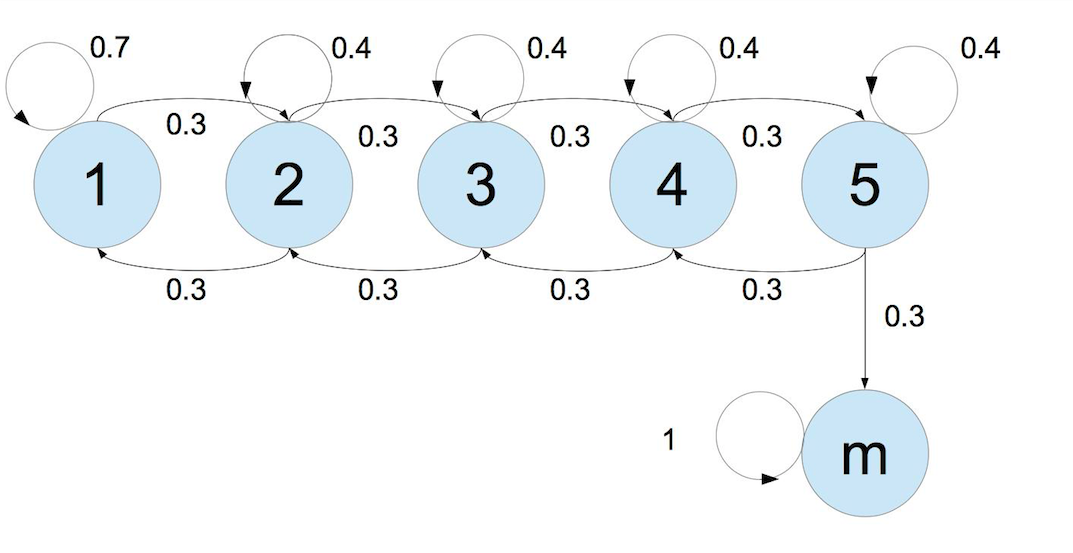}
  \captionof{figure}{ Transition for action = \textit{low drug level}.}
  \label{fig:MDP-1}
\end{subfigure}
  \begin{subfigure}{0.3\textwidth}\centering
\centering
  \includegraphics[width=0.8\linewidth]{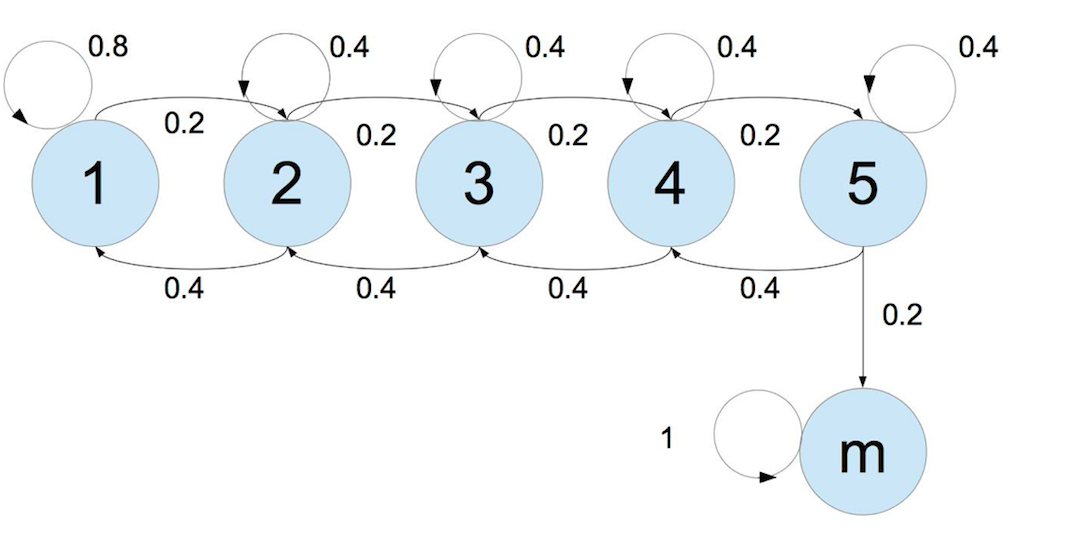}
  \captionof{figure}{ Transition for action = \textit{medium drug level}.}
  \label{fig:MDP-2}
\end{subfigure}
  \begin{subfigure}{0.3\textwidth}\centering
\centering
  \includegraphics[width=0.8\linewidth]{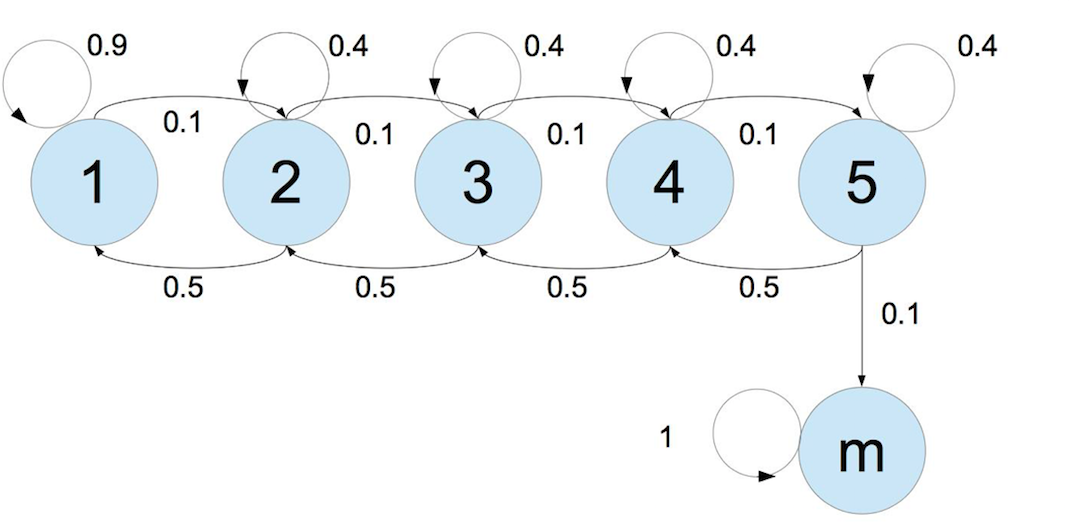}
  \captionof{figure}{ Transition for action = \textit{high drug level}.}
  \label{fig:MDP-3}
\end{subfigure}
\caption{Transition rates for the healthcare MDP instance.}
\end{figure*}
\section{Details on machine replacement example}\label{app:machine-instance}
We present here the nominal transition kernel $\bm{y}_{0}$ associated with the machine replacement example introduced in our numerical experiments section. Here we show an instance where there are $10$ states: 8 states related to the condition of the machine, and two repair states. The instances for larger number of states are constructed in the same fashion by adding some condition states for the machine. To generate $N$ samples around the nominal kernel $\bm{y}_{0}$, we create small perturbations with Garnet MDP instances.
  \begin{figure*}[htp]
  \centering
  \begin{subfigure}{0.45\textwidth}\centering  
         \includegraphics[width=0.8\linewidth]{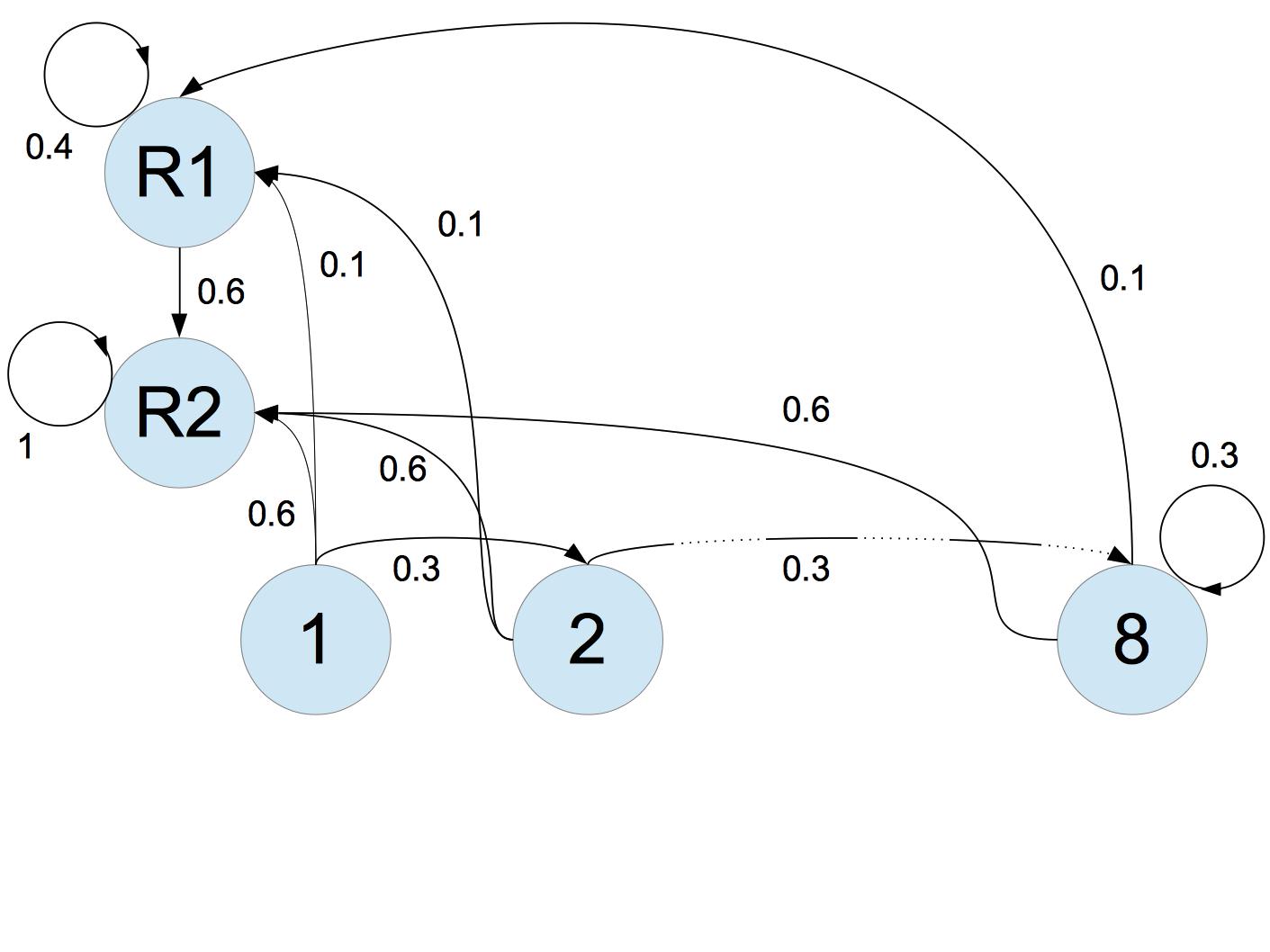}
         \caption{Nominal transition for action = \textit{ repair } in our machine replacement MDP.}
           \label{fig:Machine-MDP-1}
     \end{subfigure}
  \begin{subfigure}{0.45\textwidth}\centering  
         \centering
         \includegraphics[width=0.8\linewidth]{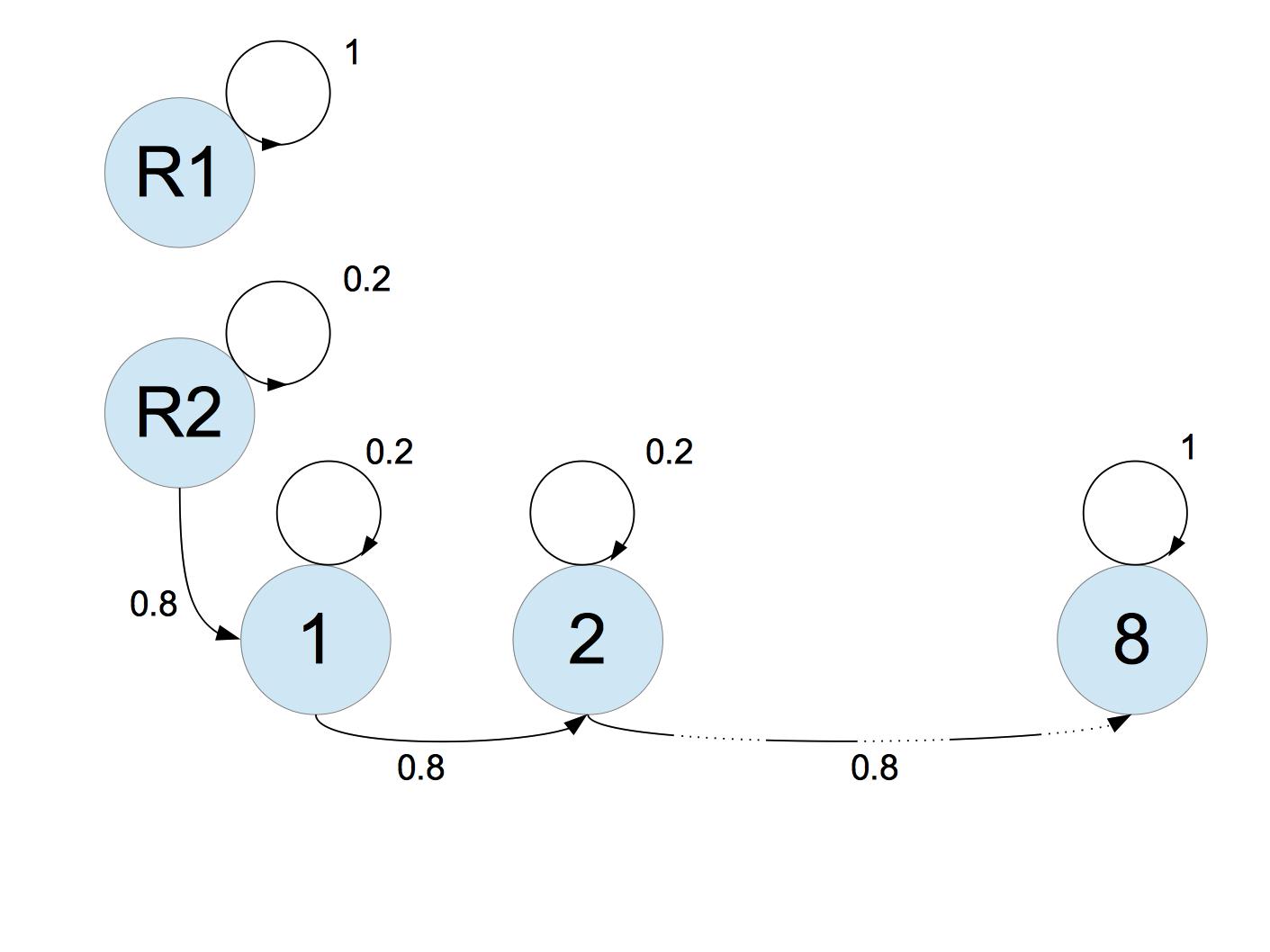}
         \caption{Nominal transition for action = \textit{ no repair } in our machine replacement MDP.}
         \label{fig:Machine-MDP-2}
              \end{subfigure}
\caption{Transition rates for the machine replacement MDP instance.}
\end{figure*}
\section{Details on numerical implementations}\label{app:simu}
%\noindent\textbf{Empirical setup.} Our empirical setup is the same of the comparisons of our various proximal setups, presented in Appendix \ref{app:comparison-our-algorithm}.

%\noindent\textbf{Setup for Value Iteration.}
% We compare our methods to VI, as well as \textit{Accelerated Value Iteration} (AVI, \cite{GGC}), an algorithm extending the acceleration scheme from convex optimization \citep{nesterov-1983,nesterov-book} to the fixed-point iteration scheme of Value Iteration. In order to obtain an $\epsilon$-solution of the robust MDP problem with \ref{alg:VI} and AVI, we use the stopping condition $\| \bm{v}_{s+1} - \bm{v}_{s} \|_{\infty} \leq \epsilon \cdot (1-\lambda) \cdot (2 \lambda)^{-1}$ (Chapter 6.3 in \citet{Puterman}). We initialize the algorithms with $\bm{v}_{0}=\bm{0}$ and, for Accelerated Value Iteration, $\bm{v}_{1}=F \left( \bm{v}_{0} \right)$.
%At epoch $\ell$ of \ref{alg:VI} and AVI, we warm-start each computation of $F(\bm{v}^{\ell})$ with the optimal solution obtained from the previous epoch $\ell-1$.  We present details about our implementation of \ref{alg:VI} and AVI in Appendix \ref{app:simu}.
%
%All the simulations are coded in Python 3.7.3, and were performed on a laptop with 2.2 GHz Intel Core i7 and 8 GB of RAM. We use Gurobi 8.1.1 to solve any linear or quadratic optimization problems involved.

\paragraph{Value Iteration.} At every epoch of Value Iteration \ref{alg:VI}, we need to compute $F(\bm{v})$ for the current value vector $\bm{v} \in \R^{S}$, where
\begin{equation*}
F(\bm{v})_{s} = \min_{\bm{x}_{s} \in \Delta(A)} \max_{\bm{y}_{s} \in \PP_{s}} \sum_{a=1}^{A} x_{sa} \left( c_{sa} + \lambda \cdot \bm{y}_{sa}^{\top}\bm{v} \right), \forall \; s \in \X.
\end{equation*} 
In order to solve this program, we could use duality in the inner maximization program, and turn the computation of $F(\bm{v})_{s}$ into a large (minimization) convex program with linear objective, some constraints and a conic quadratic constraint (see Corollary 3 in \cite{Kuhn}). However, we decide to take an alternate approach which results in a simpler optimization program, namely, a convex  program with some linear constraints and a quadratic constraint. In particular, from convex duality we have, for any $s \in \X$,
\begin{align}
F(\bm{v})_{s} & = \min_{\bm{x}_{s} \in \Delta(A)} \max_{\bm{y}_{s} \in \PP_{s}}  \sum_{a=1}^{A} x_{sa} \left( c_{sa} + \lambda \cdot \bm{y}_{sa}^{\top}\bm{v} \right)   \nonumber \\
& = \max_{\bm{y}_{s} \in \PP_{s}}  \min_{\bm{x}_{s} \in \Delta(A)}  \sum_{a=1}^{A} x_{sa} \left( c_{sa} + \lambda \cdot \bm{y}_{sa}^{\top}\bm{v} \right). \label{eq:T_v_max_min}
\end{align}

%Now we have
%\begin{align*}
%\min_{\bm{x}_{s} \in \Delta(A)}  \sum_{a=1}^{A} x_{sa} \left( c_{sa} + \lambda \cdot \bm{y}_{sa}^{\top}\bm{v} \right) 
%& = \min_{\bm{x}_{s} \geq \bm{0}, \bm{x}_{s}^{\top}\bm{e}=1 }  \sum_{a=1}^{A} x_{sa} \left( c_{sa} + \lambda \cdot \bm{y}_{sa}^{\top}\bm{v} \right) \\
%& = \min_{\bm{x}_{s} \geq \bm{0} } \max_{\mu \in \R}  \sum_{a=1}^{A} x_{sa} \left( c_{sa} + \lambda \cdot \bm{y}_{sa}^{\top}\bm{v} \right) + \mu \left( 1-\sum_{a=1}^{A} x_{sa} \right)  \\
%& = \min_{\bm{x}_{s} \geq \bm{0} } \max_{\mu \in \R }\mu +  \sum_{a=1}^{A} x_{sa} \left( c_{sa} + \lambda \cdot \bm{y}_{sa}^{\top}\bm{v} -\mu \right) \\
%& = \max_{\mu \in \R} \min_{\bm{x}_{s} \geq \bm{0} } \mu +  \sum_{a=1}^{A} x_{sa} \left( c_{sa} + \lambda \cdot \bm{y}_{sa}^{\top}\bm{v} -\mu\right),
%\end{align*}
%and therefore
%\begin{align*}
%\min_{\bm{x}_{s} \in \Delta(A)}  \sum_{a=1}^{A} x_{sa} \left( c_{sa} + \lambda \cdot \bm{y}_{sa}^{\top}\bm{v} \right) = \max \; & \mu \\
%& \mu \in \R, \\
%& c_{sa} + \lambda \bm{y}_{sa}^{\top}\bm{v} \geq \mu, \forall \; a \in \A.
%\end{align*}
Applying convex duality twice, we obtain
\begin{equation}\label{eq:T_v_simu}
\begin{aligned}
F(\bm{v})_{s} = \max \; & \mu \\
& \mu \in \R, \bm{y} \in \PP_{s},\\
& c_{sa} + \lambda \bm{y}_{sa}^{\top}\bm{v} \geq \mu, \forall \; a \in \A.
\end{aligned}
\end{equation}

In our simulations, we use the formulation \eqref{eq:T_v_simu} in order to obtain the value of $F(\bm{v})_{s}$. Given the definition of $\PP_{s}$ as \eqref{eq:uncertainty-norm-2}, formulation \eqref{eq:T_v_max_min} is a linear program with linear constraints and one quadratic constraint. Following \cite{BenTal-Nemirovski}, we can solve \eqref{eq:T_v_simu} up to accuracy $\epsilon$ in a number of arithmetic operations in $O\left(S^{3.5}A^{3.5}\log(1/\epsilon) \right).$ We warm-start each of this optimization problem with the optimal solution found in the previous epoch of \ref{alg:VI}.

We would like to note that a priori, the optimal pair in $F(\bm{v})$ in the min-max formulation as in \eqref{eq:T_max-def} may not be the same pair attaining the max-min formulation as in \eqref{eq:T_v_max_min}. However, we are only interested in the scalar value of $F(\bm{v})_{s}$, in order to run \ref{alg:VI} and obtain $\bm{v}^{*}$, the fixed-point of the operator $T$ defined in \eqref{eq:v-star-fixed-point}. Once we have obtained the vector $\bm{v}^{*}$, we can eventually solve $F(\bm{v}^{*})$ in its min-max form only once, in order to obtain the pair $(\bm{x}^{*},\bm{y}^{*})$ in $F(\bm{v}^{*})$ in its min-max formulation. Alternately, the authors in \cite{Ho} provide a method to recover the optimal solution of the min-max problem \eqref{eq:T_max-def} from the optimal solution of the max-min problem \eqref{eq:T_v_max_min}, in the case where $\PP$ is a weighted $\ell_{1}$ ball centered around $\bm{y}^{0}$.

\paragraph{Accelerated Value Iteration.} \cite{GGC} interpret the vector $\left( \bm{I} - T \right) (\bm{v})$ as the gradient of some function at the vector $\bm{v}$. Adapting the acceleration scheme from convex optimization (\cite{nesterov-1983, nesterov-book}) to an accelerated iterative algorithm for computing $\bm{v}^{*}$ leads to \textit{Accelerated Value Iteration}, which significantly outperforms Value Iteration and variants  when the discount factor is close to $1$ \citep{GGC}. In particular, for any sequences of scalar $(\alpha_{s})_{s \geq 0}$ and $(\gamma_{s})_{s \geq 0} \in \R^{\N}$, Accelerated Value Iteration (AVI) is defined as
\begin{equation}\label{alg:AVI}\tag{AVI}
\bm{v}_{0},\bm{v}_{1} \in \R^{S}, \begin{cases}
    \bm{h}_{t}=\bm{v}_{t}+\gamma_{t}\cdot \left( \bm{v}_{t}-\bm{v}_{t-1} \right), \\
	\bm{v}_{t+1} \gets \bm{h}_{t}-\alpha_{t} \left( \bm{h}_{t}- T \left( \bm{h}_{t} \right) \right), \end{cases} \forall \; t \geq 1.
\end{equation}
Following \cite{GGC}, we choose step sizes as\begin{equation*}
\alpha_{s}  = \alpha = 1/(1+\lambda),
 \gamma_{s} = \gamma = \left(1-\sqrt{1- \lambda^{2}}\right)/\lambda, \forall s \; \geq 1.
 \end{equation*}
We use \eqref{eq:T_v_simu} in order to compute $F(\bm{h})$ for \ref{alg:AVI}.

\paragraph{Gauss-Seidel Value Iteration.} We also consider Gauss-Seidel Value Iteration (GS-VI), a popular asynchronous variant of \ref{alg:VI} \citep{Puterman},
 where 
$
v^{t+1}_{s} = \max_{a \in \A} r_{sa} + \lambda \cdot \sum_{s'=1}^{s-1} P_{sas'}v^{t+1}_{s'} + \lambda \cdot \sum_{s'=s}^{n} P_{sas'}v^{t}_{s'}.$

\paragraph{Anderson Value Iteration.} We also consider Anderson VI (referred to as \textit{Anderson} in our figures), see \cite{ref-c}. In order to compute the next iterates $\bm{v}^{t+1}$, Anderson VI computes weights $\alpha_{0}, ..., \alpha_{m}$ and updates $\bm{v}^{t+1}$ as a linear combination of the last $(m+1)$-iterates $F(\bm{v}^{t}), ..., F(\bm{v}^{t-m})$:
\[ \bm{v}^{t+1} = \sum_{i=0}^{m} \alpha_{i} F(\bm{v}^{t-m+i}).\]
The weights $\bm{\alpha} \in \R^{m+1}$ are updated at every iteration, see Algorithm 1 and Equation (1) in \cite{ref-c} for further details. There is no heuristics for choosing $m$; we choose $m=5$ in our numerical experiments.
\section{Numerical experiments for KL uncertainty set}\label{app:simu-KL}
We present here our numerical results for the KL uncertainty set. We consider the healthcare instance, the machine replacement instance and the random Garnet MDP instances, introduced in Section \ref{sec:simu}.  The numerical setup is the same as for ellipsoidal uncertainty sets. Note that for the KL uncertainty set, we can not compare to Value Iteration directly, as there is no direct convex reformulation for the Bellman update 
\begin{align*}
F(\bm{v})=\min_{\bm{x} \in \Delta(A)} \max \; & \sum_{a \in \A} x_{a} \left(r_{sa}+\lambda \bm{y}^{\top}_{sa}\bm{v}\right) \\
& \left(\bm{y}_{sa}\right) \in \left( \Delta(S) \right)^{A}, \\
& \sum_{a \in \A} KL \left(\bm{y}_{sa},\bm{y}^{0}_{sa}\right) \leq \alpha.
\end{align*}
For a KL uncertainty set, computing a proximal update with the $\ell_{1}$ setup requires one binary search, and computing a proximal update with the $\ell_{2}$ setup requires two interwoven binary searches, as evidenced in Proposition \ref{prop:prop-fast-y-update-KL}. Therefore, we focus on the $\ell_{1}$ setup for the $y$-player and the $\ell_{1}$ setup for the $x$-player. We present the running times to compute an optimal policy on various instances below (healthcare and machine replacement examples, Garnet MDPs with high and low connectivity).

As the convergences times of our algorithm are longer for KL uncertainty sets than for ellipsoidal uncertainty sets, we only compute optimal solutions up to a number of states of $50$. For the Garnet MDP instances, the number of actions is equal to the number of states, while there are two actions for the machine replacement instance and three actions for the healthcare instance.

Note that we also observe longer convergence rates for the $(\ell_{1},\ell_{1})$ setup in Figures \ref{fig:comparison_norms_pq_11}-\ref{fig:comparison_norms_pq_22}. Our FOM-VI algorithm finds a solution to the $s$-rectangular robust MDP problem with KL uncertainty sets but for the Garnet MDP instances the running time greatly increases, compared to the KL uncertainty sets. The running times for the more realistic healthcare and machine replacement instances also increase but remains below 100 seconds for up to 50 states.

\begin{figure}[H]
\centering
  \includegraphics[width=0.5\linewidth]{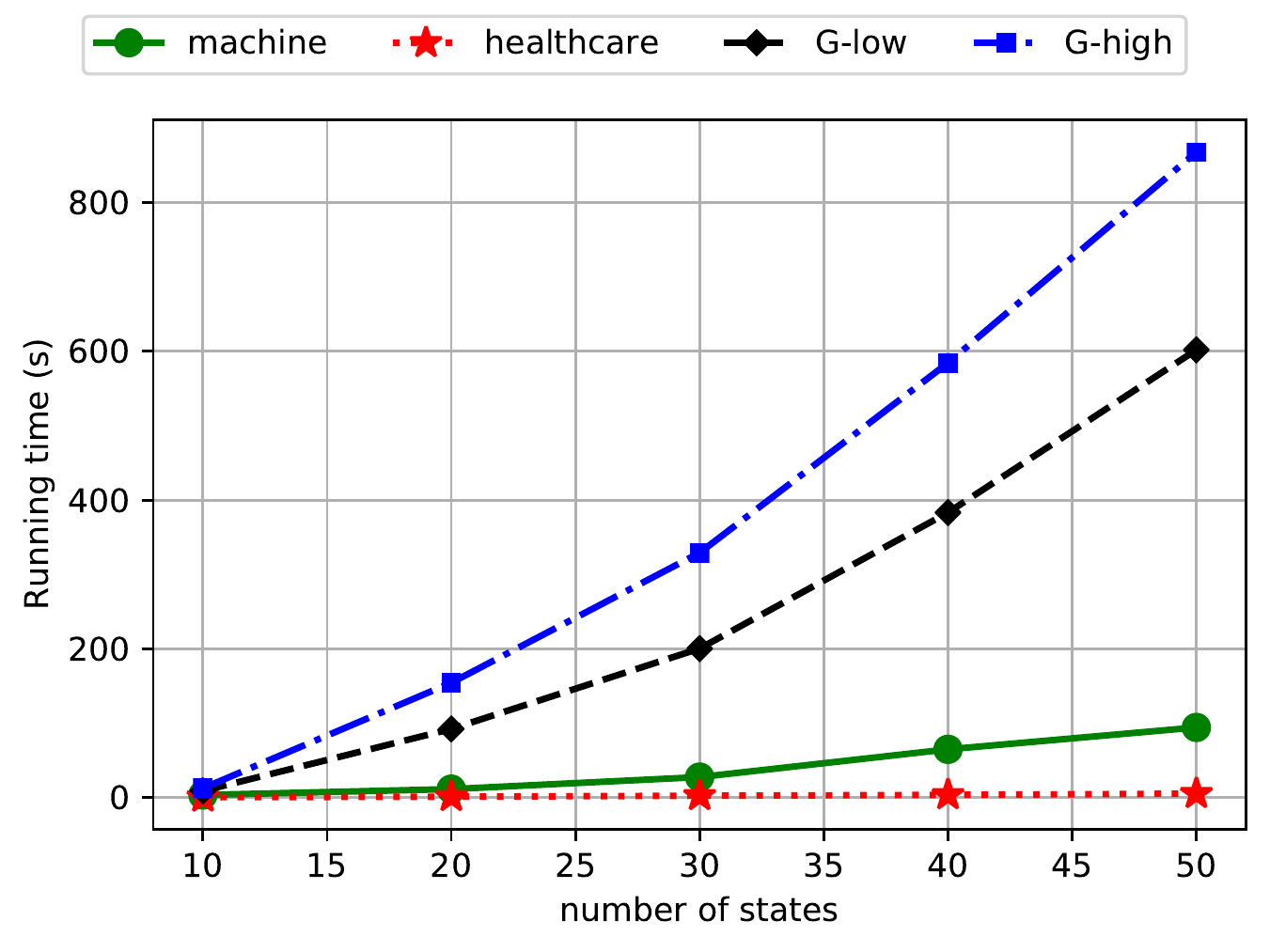}
  \captionof{figure}{Performances of FOM-VI on four different instances for KL uncertainty sets.}
  \label{fig:comparison_all_KL}
  \end{figure}

%  \begin{figure}[H]
%\centering
%  \includegraphics[width=1.0\linewidth]{figure_2/Algorithm_Comparisons_healthcare_KL}
%  \captionof{figure}{Comparison on our healthcare instance.}
%  \label{fig:comparison_real_instance_KL}
%  \end{figure}
%  
%\begin{figure}[H]
%\centering
%  \includegraphics[width=1.0\linewidth]{figure_2/Algorithm_Comparisons_Garnet_50_KL}
%  \captionof{figure}{Comparison on random Garnet MDP instances  with high connectivity (50 \% of reachable next states from any state-action pair).}
%  \label{fig:comparison_alg_all_small_KL}
%  \end{figure}
%  
%  \begin{figure}[H]
%\centering
%  \includegraphics[width=1.0\linewidth]{figure_2/Algorithm_Comparisons_Garnet_20_KL}
%  \captionof{figure}{Comparison on random Garnet MDP instances with low connectivity (20 \% of reachable next states from any state-action pair).}
%  \label{fig:comparison_alg_all_small-low-density_KL}
%  \end{figure}
%%% Local Variables:
%%% mode: latex
%%% TeX-master: "../main"
%%% End:

\end{document}